\documentclass[10pt,oneside,reqno]{amsart}
\usepackage{hyperref}
\usepackage[nonewpage]{imakeidx}
\makeindex[columns=3, title=Index of symbols]
\usepackage{amsmath}
\usepackage{amssymb}
\usepackage{amsxtra}
\usepackage{amsthm}

\allowdisplaybreaks[4]
\usepackage{stmaryrd}
\usepackage{bm}
\usepackage{bbm}
\usepackage{upgreek}
\usepackage{thmtools}
\usepackage{mathtools}
\usepackage{geometry}
\usepackage{enumitem}
\usepackage{comment}

\usepackage{mathpazo}
\usepackage{domitian}
\usepackage[T1]{fontenc}

\AtBeginDocument{
 \DeclareSymbolFont{AMSb}{U}{msb}{m}{n}
 \DeclareSymbolFontAlphabet{\mathbb}{AMSb}
 }
 
\let\SSec\S

\newcommand\lopen{\mathopen{]}}
\newcommand\ropen{\mathclose{[}}

\newcommand\rmd{\mathrm{d}}
\newcommand\rme{\mathrm{e}}
\newcommand\rmi{\mathrm{i}}

\newcommand\cA{\mathcal{A}}

\newcommand\cC{\mathcal{C}}

\newcommand\cO{\mathcal{O}}

\newcommand\cS{\mathcal{S}}

\newcommand\cX{\mathcal{X}}

\newcommand\bF{\mathbf{F}}

\newcommand\bI{\mathbf{I}}
\newcommand\bJ{\mathbf{J}}

\renewcommand\AA{\mathbb{A}}

\newcommand\CC{\mathbb{C}}

\newcommand\FF{\mathbb{F}}

\newcommand\QQ{\mathbb{Q}}
\newcommand\RR{\mathbb{R}}

\newcommand\ZZ{\mathbb{Z}}

\newcommand\mf[1]{\mathfrak{#1}}

\newcommand\A{\mathsf{A}}
\newcommand\B{\mathsf{B}}

\newcommand\G{\mathsf{G}}

\newcommand\M{\mathsf{M}}
\newcommand\N{\mathsf{N}}

\renewcommand\S{\mathsf{S}}
\newcommand\T{\mathsf{T}}

\newcommand\Z{\mathsf{Z}}

\newcommand\GL{\mathsf{GL}}

\newcommand\GSp{\mathsf{GSp}}

\newcommand\triv{\mathbbm{1}}
\newcommand\dpii{2\uppi\rmi}
\newcommand\bs{\backslash}

\renewcommand\Re{\mathop{\mathrm{Re}}}
\renewcommand\Im{\mathop{\mathrm{Im}}}
\DeclareMathOperator\Tr{Tr}
\DeclareMathOperator\sgn{sgn}
\DeclareMathOperator\supp{supp}

\DeclareMathOperator\res{res}
\DeclareMathOperator\fp{f.p.}
\DeclareMathOperator\pv{p.v.}

\DeclareMathOperator\rad{rad}
\DeclareMathOperator\Res{Res}
\DeclareMathOperator\Ind{Ind}

\newcommand\fin{\mathrm{fin}}
\newcommand\id{\mathrm{id}}
\newcommand\el{\mathrm{ell}}
\newcommand\hyp{\mathrm{hyp}}
\newcommand\unip{\mathrm{unip}}
\newcommand\reg{\mathrm{reg}}
\newcommand\Kl{\mathrm{Kl}}
\DeclareMathOperator\vol{vol}
\DeclareMathOperator\orb{orb}
\renewcommand\leq{\leqslant}
\renewcommand\geq{\geqslant}

\newcommand\legendresymbol[2]{\genfrac{(}{)}{}{}{#1}{#2}}

\makeatletter
\def\subsection{\@startsection{subsection}{2}%
  \z@{3pt\@plus0pt}{-.5em}%
  {\normalfont\bfseries}}
\makeatother

\makeatletter
\def\subsubsection{\@startsection{subsubsection}{2}%
  \z@{1pt\@plus0pt}{-.5em}%
  {\normalfont\itshape}}
\makeatother

\makeatletter
\def\@seccntformat#1{%
  \protect\textup{\protect\@secnumfont
    \ifnum\pdfstrcmp{subsection}{#1}=0 \bfseries\fi
    \csname the#1\endcsname
    \protect\@secnumpunct
  }%
}  
\makeatother

\newtheoremstyle{THEOREM}{2.5pt}{2pt}{\itshape}{}{\bfseries}{.}{.5em}{\thmname{#1}\thmnumber{ #2}\thmnote{ (#3)}}

\newtheoremstyle{DEFINITION}{2.5pt}{2pt}{}{}{\bfseries}{.}{.5em}{\thmname{#1}\thmnumber{ #2}\thmnote{ (#3)}}

\newtheoremstyle{EXERCISE}{2pt}{2pt}{}{}{\scshape}{.}{.5em}{\thmname{#1}\thmnumber{ #2}\thmnote{ (#3)}}

\theoremstyle{THEOREM}
\newtheorem{theorem}{Theorem}[section]
\newtheorem{lemma}[theorem]{Lemma}
\newtheorem{proposition}[theorem]{Proposition}
\newtheorem{corollary}[theorem]{Corollary}

\theoremstyle{DEFINITION}

\theoremstyle{EXERCISE}
\newtheorem{remark}[theorem]{Remark}

\numberwithin{equation}{section}

\makeatletter
\renewenvironment{proof}[1][\proofname]{\par
  \vspace{-6pt}
  \pushQED{\qed}
  \normalfont \topsep6\p@\@plus6\p@\relax
  \trivlist
  \item[\hskip\labelsep\rmfamily\bfseries
    #1\@addpunct{:}]\ignorespaces
}{
  \popQED\endtrivlist\@endpefalse
  \vspace{-6pt}
}
\makeatother



\begin{document}
\setlength \lineskip{3pt}
\setlength \lineskiplimit{3pt}
\setlength \parskip{1pt}
\setlength \partopsep{0pt}

\title{Beyond endoscopy for $\GL_2$ over $\QQ$ with ramification 2: bounds towards the Ramanujan conjecture}
\author{Yuhao Cheng}
\address{Qiuzhen College, Tsinghua University, 100084, Beijing, China}
\email{chengyuhaomath@gmail.com}
\date{15 August, 2026}
\begin{abstract}
We continue generalizing Altu\u{g}'s work on $\GL_2$ over $\QQ$ in the unramified setting for \emph{Beyond Endoscopy} to the ramified case where ramification occurs at $S=\{\infty,q_1,\dots,q_r\}$ with $2\in S$, after generalizing the first step. We establish a new proof of the $1/4$ bound towards the Ramanujan conjecture for the trace of the cuspidal part in the ramified case, which is also provided by adapting Altuğ's original approach. The proof proceeds in three stages: First, we estimate the contributions from the non-elliptic parts of the trace formula. Then, we apply the main result from our previous work to isolate the $1$-dimensional representations within the elliptic part. Finally, we employ technical analytic estimates to bound the remainder terms in the elliptic part.
\end{abstract}

\maketitle
 
\tableofcontents

\section{Introduction}
\subsection{Background}
Let $\G$ be a connected reductive algebraic group over the field $\QQ$ of rational numbers and let $\AA$ denote the ad\`ele ring of $\QQ$. Let $^L\G$ be the $L$-group of $\G$. Let $\pi$ be an automorphic representation of $\G$ over $\QQ$ and $\rho$ be a finite dimensional complex representation of $^L\G$. For a finite set $S$ of places of $\QQ$ containing the archimedean place 
$\infty$, the partial $L$-function $L^{S}(s,\pi,\rho)$ can be defined (cf. \cite{getz2024} for details) and expressed as a Dirichlet series
\[
L^{S}(s,\pi,\rho)=\sum_{\substack{n=1\\ \gcd(n,S)=1}}^{+\infty}\frac{a_{\pi,\rho}(n)}{n^s}
\]
for $\Re s$ sufficiently large.

The famous \emph{Ramanujan conjecture} states that 
\[
|a_{\pi,\rho}(n)|\ll_{\rho,\varepsilon} n^{\varepsilon}
\]
for any $\varepsilon>0$.

While the conjecture is expected to hold for $\GL_m$, counterexamples exist for certain representations of groups like $\GSp_4$ \cite{howe1979counterexample}. For holomorphic modular forms, Deligne's proof \cite{deligne1974,deligne1980} via the \emph{Weil conjectures} establishes the bound. The primary challenge lies in nonholomorphic Maass forms. A common approach is to establish bounds of the form 
\[
|a_{\pi,\rho}(n)|\ll_{\rho,\varepsilon} n^{\varrho+\varepsilon},
\]
where $\varrho\geq 0$ is called a \emph{bound towards the Ramanujan conjecture}. The trivial bound $\varrho=1/2$ has been progressively improved: For $\G=\GL_m$, the current bounds are $\varrho=7/64$ ($m=2$), $5/14$ ($m=3$), $9/22$ ($m=4$), and $\varrho=1/2-1/(m^2+1)$ for general $m$. We refer to \cite{blomer2013role,sarnak2005} for comprehensive surveys.

Applying the trace formula gives
\begin{equation}\label{eq:traceformulacusp}
\sum_{\pi}m_\pi a_{\pi,\rho}(n)\prod_{v\in S}\Tr(\pi_v(f_v))=I_{\mathrm{cusp}}(f^{n,\rho})
\end{equation}
for $\gcd(n,S)=1$,
where $f^{n,\rho}=\bigotimes_{v\in S}f_v\otimes\bigotimes_{v\notin S}'f_v^{n,\rho}$, with $f_v^{n,\rho}$ spherical for $v\notin S$, and $\Tr(\pi_p(f_p^{n,\rho}))=a_{\pi,\rho}(p^{v_p(n)})$ for $p\notin S$.

Any bound $\varrho$ towards the Ramanujan conjecture implies
\begin{equation}\label{eq:boundramanujan}
I_{\mathrm{cusp}}(f^{n,\rho})\ll_{\rho,\varepsilon} n^{\varrho+\varepsilon}.
\end{equation}
Thus, \eqref{eq:boundramanujan} provides a trace formula analog of these bounds.

\subsection{The main result in this paper and proof strategy}
In this paper, we will generalize the main result in \cite{altug2017} by considering the case over $\QQ$ with ramification and the function $f^{n}$ for any $n$ such that $\gcd(n,S)=1$. 
Let $S=\{\infty,q_1,\dots,q_r\}$ for primes $q_1,\dots,q_r$ such that $2\in S$.
For any prime number $p$ and $m\in \ZZ_{\geq 0}$, we define
\[
\cX_p^{m}=\{X\in \M_2(\ZZ_p)\,|\, \mathopen{|}\det X\mathclose{|}_p = p^{-m}\}.
\]
For example, for $m=0$, we have $\cX_p^{m}=\GL_2(\ZZ_p)$. 

For any $\gcd(n,S)=1$, we set $f^{n}=\bigotimes_{v\in \mf{S}}f^{n}_v$, where $\mf{S}$ denotes the set of places of $\QQ$, and
\begin{enumerate}[itemsep=0pt,parsep=0pt,topsep=2pt,leftmargin=0pt,labelsep=3pt,itemindent=9pt,label=\textbullet]
  \item If $v=p\notin S$, we define $f^n_v=p^{-n_p/2}\triv_{\cX_p^{n_p}}$, where $n_p=v_p(n)$.
  \item If $v=q_i\in S$ and is a prime, we define $f^n_v=f_{q_i}\in C_c^\infty(\G(\QQ_{q_i}))$.
  \item If $v=\infty$, we define $f^n_v=f_\infty\in C_c^\infty(Z_+\bs \G(\RR))$. 
\end{enumerate}
By Hecke operator theory, $f_v^{n}$ is spherical for $v\notin S$, and $\Tr(\pi_p(f_p^{n}))=a_{\pi}(p^{n_p})$ for all $p\notin S$, so that \eqref{eq:traceformulacusp} holds.

Unlike the previous paper \cite{cheng2025}, we require $f_\infty$ to be compactly supported due to technical constraints in handling the hyperbolic part of the trace formula. (Actually, this assumption is not important by using a limit process.)

Our main theorem in this paper is
\begin{theorem}\label{thm:totalestimate}
For any $\varepsilon>0$ and $f^n$ as above, we have
\[
I_{\mathrm{cusp}}(f^n)\ll n^{\frac 14+\varepsilon},
\]
where the implied constant only depends on $f_\infty$, $f_{q_i}$ and $\varepsilon$.
\end{theorem}

This establishes a new proof of the $1/4$ bound towards the Ramanujan conjecture using \emph{only} the Arthur-Selberg trace formula over 
$\QQ$ with ramification, generalizing the result of \cite{altug2017} for the unramified case, although in some sense it is weaker. (Altu\u{g} established the bound in the unramified case with $\varepsilon=0$ for primes $n=p$.) Kuznetsov \cite{kuznetsov1981} derived this $1/4$ bound via the Petersson-Kuznetsov formula and Weil's Kloosterman sum bound \cite{weil1948}.

The proof of \autoref{thm:totalestimate} is a generalization of \cite{altug2017}.
The key point is to isolate the traces of the nontempered representations in the spectral side in the elliptic part of the trace formula which gives the main contribution $n^{1/2}$, which was proved in \cite{cheng2025} that generalized the main result of \cite{altug2015}.

The proof is divided into the following steps:
\begin{enumerate}[itemsep=0pt,parsep=0pt,topsep=0pt,leftmargin=0pt,labelsep=3pt,itemindent=9pt,label=\textbullet]
  \item \textbf{Estimates of non-elliptic parts.} We estimate all the terms in both sides of the Arthur-Selberg trace formula except for the elliptic term in the geometric side and the nontempered part (that is, the sum of traces of $1$-dimensional representations) in the spectral side. We find all the contributions are $O(n^\varepsilon)$.
  \item \textbf{Estimate of the easy terms in the elliptic part.}  Using the main result of \cite{cheng2025}, we can split the elliptic term in the geometric side into the following terms: the $1$-dimensional term, the Eisenstein term, $\Sigma(0)$, $\Sigma(\square)$, and $\Sigma(\xi\neq 0)$. The $1$-dimensional term is cancelled by the corresponding term in the spectral side, the Eisenstein term and $\Sigma(0)$ are both $O(n^\varepsilon)$, and $\Sigma(\square)$ is also $O(n^\varepsilon)$ by using the approximate functional equation in the reverse direction.
  \item \textbf{Estimate of $\Sigma(\xi\neq 0)$.} This is the hardest part. First, we need to estimate the Fourier transform of functions with singularities in the semilocal space $\QQ_S$. We first do local estimates. The archimedean part was mainly done by Altu\u{g} in \cite{altug2017}. The nonarchimedean part is just direct computations. By combining them together, we obtain the result for the estimate in $\QQ_S$. Secondly, we need to bound the generalized Kloosterman sums, which was mainly done in \cite{altug2017}. Finally, we do direct estimates to obtain the result $O(n^{1/4+\varepsilon})$.
\end{enumerate}

\begin{remark}
Further improvement beyond $O(n^{1/4+\varepsilon})$ would require exploiting cancellations in generalized Kloosterman sums, rather than relying on individual bounds.
\end{remark}

\subsection{Notations}
\begin{enumerate}[itemsep=0pt,parsep=0pt,topsep=0pt,leftmargin=0pt,labelsep=3pt,itemindent=9pt,label=\textbullet]
  \item $\# X$ denotes the number of elements in a set $X$.
  \item For $A\subseteq X$, $\triv_A$ denotes the characteristic function of $A$, defined by $\triv_A(x)=1$ for $x\in A$ and $\triv_A(x)=0$ for $x\notin A$.
  \item $\triv$ also denotes the trivial character or the trivial representation.
  \item For $x\in \RR$, $\lfloor x\rfloor$ denotes the greatest integer that is less than or equal to $x$, and $\lceil x\rceil$ denotes the smallest integer that is greater than or equal to $x$.
  \item We often use the notation $a\equiv b\,(n)$ to denote $a\equiv b\pmod n$.
  \item For $D\equiv 0,1\pmod 4$, $\legendresymbol{D}{\cdot}$ denotes the Kronecker symbol.
  \item If $R$ is a ring (which we \emph{always} assume to be commutative with $1$), $R^\times$ denotes its group of units.
  \item For $S=\{\infty, q_1,\dots,q_r\}$, $\ZZ^S$\index{zuppers@$\ZZ^S$} denotes the ring of $S$-integers in $\QQ$, that is
    \[
    \ZZ^S=\{\alpha\in \QQ\ |\ v_p(\alpha)\geq 0\ \text{for all}\ p\notin S\}.
    \]
    Additionally, we define \index{qs@$\QQ_S$} \index{qsfin@$\QQ_{S_\fin}$}
    \[
    \QQ_S=\prod_{v\in S}\QQ_v=\RR\times \QQ_{q_1}\times\dots\times\QQ_{q_r}\qquad\QQ_{S_\fin}=\prod_{v\in S_\fin}\QQ_v=\QQ_{q_1}\times\dots\times\QQ_{q_r},
    \]
  and \index{zsfin@$\ZZ_{S_\fin}$}
    \[
    \ZZ_{S_\fin}=\prod_{v\in S_\fin}\ZZ_v=\ZZ_{q_1}\times\dots\times\ZZ_{q_r}.
    \]
  \item For $n\in \ZZ$, we write $\gcd(n,S)=1$, or $n\in \ZZ_{(S)}$\index{zlowers@$\ZZ_{(S)}$}, if 
  $p\nmid n$ for all $p\in S$.
  We write $n\in \ZZ_{(S)}^{>0}$\index{zlowers0@$\ZZ_{(S)}^{>0}$} if additionally $n>0$.
  \item $\mf{S}$\index{smf@$\mf{S}$} denotes the set of places of $\QQ$.
  \item Let $p$ be a prime. For $a\in\QQ$, we define $a_{(p)}$ to be the $p$-part of $a$, that is, $a_{(p)}=p^{v_p(a)}$, and we define $a^{(p)}=a/a_{(p)}$. Moreover, we define \index{alowerq@$a_{(q)}$}\index{aupperq@$a^{(q)}$} 
  \[
    a_{(q)}=\prod_{i=1}^{r}q_i^{v_{q_i}(a)}\quad\text{and}\quad a^{(q)}=\frac{a}{a_{(q)}}=\prod_{p\notin S}p^{v_{p}(a)}.
  \]
  \item Let $n\in \ZZ_{>0}$. $\bm{d}(n)$\index{d@$\bm{d}(n)$} denotes the number of divisors of $n$. $\bm{\omega}(n)$\index{omega@$\bm{\omega}(n)$} denotes the number of prime divisors of $n$. $\bm{\phi}(n)$\index{phi@$\bm{\phi}(n)$} denotes the Euler totient function. $\bm{\mu}(n)$\index{mu@$\bm{\mu}(n)$} denotes the M\"obius function, defined by
  \[
  \bm{\mu}(n)=\begin{cases}
    1, & \text{if $n=1$}, \\
    (-1)^m, & \text{if $n$ is a product of $m$ distinct primes}, \\
    0, & \text{otherwise}.
  \end{cases}
  \]
  $\bm{\Lambda}(n)$ \index{lambda@$\bm{\Lambda}(n)$} denotes the von Mangoldt function, defined by
  \[
  \bm{\Lambda}(n)=\begin{cases}
    \log p, & \text{if $n=p^m$ with $m\in \ZZ_{>0}$}, \\
    0, & \text{otherwise}.
  \end{cases}
  \]
  \item $\Gamma(s)$ denotes the gamma function, defined by
  \[
  \Gamma(s)=\int_{0}^{+\infty}\rme^{-x}x^s\frac{\rmd x}{x}
  \]
  for $\Re s>0$, analytically continued to $\CC$. $\zeta(s)$ denotes the Riemann zeta function, defined by
  \[
    \zeta(s)=\sum_{n=1}^{+\infty}\frac{1}{n^s}
  \]
  for $\Re s>1$, analytically continued to $\CC$.
  \item $(\sigma)$ denotes the vertical contour from $\sigma-\rmi\infty$ to $\sigma+\rmi\infty$.
  \item Let $f$ be a meromorphic function near $z=z_0$ and suppose that 
  \[
  f(z)=\sum_{n\in \ZZ}a_n(z-z_0)^n
  \]
  is its Laurent expansion near $z_0$, then we denote \index{fp@$\fp_{z=z_0}f(z)$}$\fp_{z=z_0}f(z)=a_0$.
  \item We define $\cS$\index{scal@$\cS$} be the set of smooth functions $\Phi$ on $\lopen 0,+\infty\ropen$ such that for any $k\in \ZZ_{\geq 0}$, $\Phi^{(k)}(x)$ are of rapid decay as $x\to +\infty$.
  \item We define \index{ex@$\rme(x)$} \index{einftyx@$\rme_\infty(x)$} $\rme(x)=\rme_\infty(x)=\rme^{\dpii x}$. For a prime $p$ and $x\in \QQ_p$, we define $\langle x\rangle_p$ to be the "fractional part" of $x$. Namely, if $x=\sum_{n\geq -N}a_np^n\in \QQ_p$, then 
      \[
      \langle x\rangle_p=\sum_{n=-N}^{-1}a_np^n\in \QQ.
      \]
      We then define \index{epy@$\rme_p(y)$}$\rme_p(x)=\rme(-\langle x\rangle_p)$ for $x\in \QQ_p$. \emph{Note the minus sign}.
  \item We use $f(x)=O(g(x))$ or $f(x)\ll g(x)$ to denote that there exists a constant $C$ such that $|f(x)|\leq C|g(x)|$ for all $x$ in a specified set. If the constant depends on other variables, they will be subscripted under $O$ or $\ll$.
  \item The notation $f(x)\asymp g(x)$ indicates that $f(x)\ll g(x)$ and $g(x)\ll f(x)$. If the constant depends on other variables, they will be subscripted under $\asymp$.
\end{enumerate}

\section{Preliminaries}\label{sec:preliminaries}
In this section we recall some useful definitions and properties in the previous paper \cite{cheng2025}. Readers familiar with \cite{cheng2025} may skip this section except for \autoref{subsec:modifiednorm} and \autoref{subsec:singularities}.

\subsection{The modified $p$-adic norm and properties}\label{subsec:modifiednorm}
For any prime $p$, the \emph{modified norm} $|\cdot|_p'$\index{1ynorm@$\vert \cdot\vert_p'$} is defined as follows: 

For $p\neq 2$ and $y\in \QQ_p$, we define $|y|_p' = p^{-2\lfloor v_p(y)/2\rfloor}$.  
This satisfies $|y|_p' = |y|_p$ if $v_p(y)$ is even, and $|y|_p' = p |y|_p$ if $v_p(y)$ is odd.

For $p=2$ and $y\in \QQ_p$, we define
\[
|y|_p' = \begin{cases}
  p^{-v_p(y)}=|y|_p & \text{if $v_p(y)$ is even, $y_0\equiv 1\,(4)$}, \\
  p^{-v_p(y)+2}=p^2|y|_p & \text{if $v_p(y)$ is even, $y_0\equiv 3\,(4)$}, \\
  p^{-v_p(y)+3}=p^3|y|_p & \text{if $v_p(y)$ is odd},
\end{cases}
\]
where $y_0=yp^{-v_p(y)}$.

For any regular element $\gamma\in\G(\QQ_p)$, we denote $T_\gamma=\Tr\gamma$ and $N_\gamma=\det\gamma$. $k_\gamma$ is defined such that $p^{k_\gamma}=|T_\gamma^2-4N_\gamma|_p'^{-1/2}$. This coincides with the original definition in \cite{cheng2025} (see Proposition 2.6 of loc. cit.).

We now establish some properties of the modified norms:
\begin{lemma}\label{lem:modifynorm}
We have the following results for modified norms:
\begin{enumerate}[itemsep=0pt,parsep=0pt,topsep=0pt, leftmargin=0pt,labelsep=2.5pt,itemindent=15pt,label=\upshape{(\arabic*)}]
  \item $|x|_p'\asymp_p |x|_p$ for all $x\in \QQ_p$. 
\item If $a\in 1+p^2\ZZ_p$, then $|ax|_p'=|x|_p'$ for all $x\in \QQ_p$.
\item  If $a$ is a square in $\QQ_p^\times$, then $|ax|_p'=|a|_p|x|_p'$ for all $x\in \QQ_p$.
\end{enumerate}
\end{lemma}
\begin{proof}
These results follow immediately from the definition.
\end{proof}

\begin{corollary}
$|ax|_p'\asymp_p |a|_p'|x|_p'$ for all $a,x\in \QQ_p$.
\end{corollary}
\begin{proof}
By \autoref{lem:modifynorm}, $|ax|_p' \asymp |ax|_p = |a|_p |x|_p \asymp |a|_p' |x|_p'$.
\end{proof}

\begin{lemma}\label{lem:ladiconverge}
If $\Re s>-1$, then the integrals
\[
\int_{\ZZ_p} |x|_p^{s}\rmd x  \qquad\text{and}\qquad \int_{\ZZ_p} |x|_p'^{s}\rmd x 
\]
both converge.
\end{lemma}
\begin{proof}
The first integral is the infinite sum
\[
\sum_{u=0}^{+\infty}p^{-su}p^{-u}\left(1-\frac1p\right),
\]
which converges absolutely for $\Re s>-1$. The second integral converges by \autoref{lem:modifynorm} and the comparison $|x|_p'^{s} \asymp|x|_p^{s}$.
\end{proof}

\subsection{The elliptic part of the trace formula and measure normalizations}
Let $\mf{S}$ be the set of places of $\QQ$ and let $\AA$ be the ad\`ele ring of $\QQ$. For $f\in C_c^\infty(\G(\AA)^1)=C_c^\infty(Z_+\bs \G(\AA))$ of the form $f=\bigotimes_{v\in \mf{S}}' f_v$, the elliptic part of the trace formula is
\begin{equation}\label{eq:elliptictrace}
I_\el(f)=\sum_{\gamma\in \G(\QQ)^\#_\el}\vol(\gamma)\prod_{v\in \mathfrak{S}}\orb(f_v;\gamma),
\end{equation}
where $ \G(\QQ)^\#_\el$ denotes the set of elliptic conjugacy classes in $\G(\QQ)$, and
\[
\orb(f_v;\gamma)=\int_{\G_\gamma(\QQ_v)\bs \G(\QQ_v)} f_v(g^{-1}\gamma g)\rmd g,
\qquad
\vol(\gamma)=\int_{Z_+\G_\gamma(\QQ)\bs \G_\gamma(\AA)} \rmd g,
\]
where $Z_+$ denotes the connected component of the identity matrix in the center $\Z(\RR)$ of $\G(\RR)$ and $\G_\gamma$ denotes the centralizer of $\gamma$. The measures of $\G_\gamma(\QQ_v)$ and $\G(\QQ_v)$ are normalized as follows. For $\G(\QQ_v)$, 
\begin{enumerate}[itemsep=0pt,parsep=0pt,topsep=0pt,leftmargin=0pt,labelsep=3pt,itemindent=9pt,label=\textbullet]
  \item If $v=p$ is a prime, we normalize the Haar measure on $\G(\QQ_p)$ such that the volume of $\G(\ZZ_p)$ is $1$.
  \item If $v=\infty$, we choose an arbitrary Haar measure (which does not affect the following discussions).
\end{enumerate} 

By Proposition 2.1 of \cite{cheng2025}, for elliptic $\gamma$, there exists a quadratic extension $E/\QQ$ such that $\G_\gamma\cong\Res_{E/\QQ}\GL_1$. For hyperbolic $\gamma$, we may take $E=\QQ\times \QQ$ so that such identification still holds. Hence $\G_\gamma(\QQ_v)\cong E_v^\times$.
\begin{enumerate}[itemsep=0pt,parsep=0pt,topsep=0pt,leftmargin=0pt,labelsep=3pt,itemindent=9pt,label=\textbullet]
  \item If $v=p$ is a prime, we know that $E_p$ is either $\QQ_p^2$ or a quadratic extension of $\QQ_p$. In the first case, we let $\G_\gamma(\ZZ_p)=\ZZ_p^\times \times \ZZ_p^\times$. In the second case, we let $\G_\gamma(\ZZ_p)=\cO_{E_p}^\times$.
  In each case, we normalize the Haar measure on $\G_\gamma(\QQ_p)=E_p^\times$ such that the volume of $\G_\gamma(\ZZ_p)$ is $1$.
  \item If $v=\infty$, we know that $E_v$ is either $\RR^2$ or $\CC$. If $E_v=\RR^2$, we define the measure on $E_v^\times=\RR^\times \times \RR^\times$ as $\rmd x\rmd y/|xy|$, where $(x,y)$ is the coordinate of $\RR^2$. The action of $Z_+$ on $E_v^\times$ is $a\cdot(x,y)=(ax,ay)$, and we define the measure on $Z_+\bs E_v^\times$ to be the measure of the quotient of $E_v^\times$ by the measure $\rmd a/a$ on $Z_+$. If $E_v=\CC$, we choose the  measure on $E_v^\times=\CC^\times$ to be $2\rmd r\rmd \theta/r$, where we use the polar coordinate $(r,\theta)$ on $\CC^\times$. The action of $Z_+$ on $E_v^\times$ is $a\cdot z=az$, and we define the measure on $Z_+\bs E_v^\times$ to be the measure of the quotient of $E_v^\times$ by the measure $\rmd a/a$ on $Z_+$.
\end{enumerate} 

\begin{remark}
These measure normalizations are taken from \cite{finis2011} which is used in the previous paper \cite{cheng2025}, instead of the normalizations in \cite{langlands2004} which is used in Altu\u{g}'s first two papers \cite{altug2015,altug2017}.
\end{remark}

\subsection{Singularities of the orbital integrals}\label{subsec:singularities}
For any $\nu=(\nu_1,\nu_2,\cdots,\nu_r)\in \ZZ^r$ we write $q^\nu=q_1^{\nu_1}q_2^{\nu_2}\cdots q_r^{\nu_r}$.  For a fixed sign $\pm$ and $\nu\in \ZZ^r$, we define \index{omegax@$\omega_\infty(x)$}
\[
\omega_\infty(x)=\begin{cases}
             0, & x^2\mp 1>0, \\
             1, & x^2\mp 1<0
           \end{cases}
\]
for $x\in \RR$ with $x^2\neq \pm 1$ and \index{omegay@$\omega_p(y)$}
\[
\omega_p(y)=\legendresymbol{(y^2\mp 4nq^\nu)|y^2\mp 4nq^\nu|_{p}'}{p}
\]
for $p\in S$ and $y\in \QQ_p$ with $y^2\neq \pm 4nq^\nu$. If $p=q_i$, we often write $\omega_p=\omega_i$\index{omegayi@$\omega_i(y_i)$}. For $\iota\in \{0,1\}$, we define
\[
X_{\iota}=\{x\in \RR\ |\ \omega_\infty(x)=\iota\}\index{xiota@$X_{\iota}$}
\]
and for $\epsilon_i\in \{0,\pm 1\}$, we define
\[
Y_{\epsilon_i}=\{y_i\in \QQ_{q_i}\ |\ \omega_i(y_i)=\epsilon_i\}. \index{yepsilon@$Y_{\epsilon_i}$} 
\]
Note that $\omega_\infty$ and $\omega_p$ \emph{do} depend on $n$, $\pm$ and $\nu$, but we usually omit them from writing.

\subsubsection{Nonarchimedean case}
For any prime $p\in S$, we define \index{thetap@$\theta_{p}(\gamma)$}
\[
\theta_{p}(\gamma)=\frac{1}{\mathopen{|}\det\gamma\mathclose{|}_p^{1/2}}\left(1-\frac{\chi(p)}{p}\right)^{-1}p^{-{k_\gamma}}\orb(f_{p};\gamma),
\]
where $\chi(p)=\chi_\gamma(p)$ is the quadratic character depending on whether $E_p/\QQ_p$ is split, inert, or ramified. According to Corollary 2.12 of \cite{cheng2025}, the local behavior of $\theta_{p}$ at $z=aI$ is
\begin{equation}\label{eq:shalikalocal}
\theta_{p}(\gamma)=\lambda_1\left(1-\frac{\chi(p)}{p}\right)^{-1}p^{-{k_\gamma}} \frac{1-\chi(p)}{1-p}+\lambda_2.
\end{equation}

Clearly $\theta_{p}(\gamma)$ is invariant under conjugation. Thus $\theta_{p}(\gamma)$ can be parametrized by $T=\Tr\gamma$ and $N=\det\gamma$, i.e., $\theta_{p}(\gamma)=\theta_p(T,N)$\index{thetapyn@$\theta_{p}(y,N)$}. 
Since $\theta_p(\gamma)$ is smooth away from the center, $\theta_p(T,\pm nq^\nu)$ is smooth except at $T^2=\pm 4nq^\nu$. Moreover, for any $\epsilon\in \{0,\pm 1\}$, by \eqref{eq:shalikalocal} we have
\[
\theta_p(y,\pm nq^\nu)=\theta_{p,1}(y,\pm nq^\nu)+\theta_{p,0}(y,\pm nq^\nu),
\]
where $\theta_{p,\tau}(y,\pm nq^\nu)={|y^2\mp 4nq^\nu|_p'}^{\tau/2}\psi_\tau(y)$ for $\tau\in \{0,1\}$ with $\psi_\tau(y)$ a smooth, compactly supported function on $Y_\epsilon$.

\subsubsection{Archimedean case}
We recall the following result for the archimedean orbital integral (cf. \cite[Theorem 2.13]{cheng2025}):

\begin{theorem}\label{thm:archimedeanintegral}
For any $f_\infty\in C_c^\infty(\G(\RR))$, any maximal torus $\T(\RR)$ in $\G(\RR)$ and any $z$ in the center of $\G(\RR)$, there exists a neighborhood $N$ in $\T(\RR)$ of $z$ and smooth functions $g_1,g_2\in C^\infty(N)$ (depending on $f_\infty$ and $z$) such that
\begin{equation}\label{eq:archimedeanintegral}
\orb(f_\infty;\gamma)=g_1(\gamma)+\frac{|\gamma_1\gamma_2|^{1/2}}{|\gamma_1-\gamma_2|}g_2(\gamma)
\end{equation}
for any $\gamma\in \T(\RR)$, where $\gamma_1$ and $\gamma_2$ are the eigenvalues of $\gamma$.  Moreover, $g_1$ and $g_2$ can be extended smoothly to all split and elliptic elements, remaining invariant under conjugation, with $g_1(\gamma)=0$ if $\T(\RR)$ is split, and $g_2$ can further be extended smoothly to the center. If $f_\infty$ is $Z_+$-invariant, then $g_1$ and $g_2$ are also $Z_+$-invariant.
\end{theorem}
Since $g_1$ and $g_2$ are invariant under conjugation, we can parametrize them by $T_\gamma$ and $N_\gamma$ as in the nonarchimedean case.

Since $T_{z\gamma}=aT_\gamma$ and  $N_{z\gamma}=a^2N_\gamma$ for $z=aI$ with $a>0$, we have $g_i(T_\gamma,N_\gamma)=g_i(aT_\gamma,a^2N_\gamma)$
for $i=1,2$ and any $a>0$. If we let $a=1/(2\sqrt{|N_\gamma|})$, we know that $g_i$ depends only on $T_\gamma/(2\sqrt{|N_\gamma|})$ and $\sgn N_\gamma$ for $i=1,2$. Thus we can write
\[
\orb(f_\infty;\gamma)=g_1^{\sgn N_\gamma}\legendresymbol{T_\gamma}{2\sqrt{|N_\gamma|}}+\left|\frac{T_\gamma^2-4N_\gamma}{N_\gamma}\right|^{-1/2}g_2^{\sgn N_\gamma}\legendresymbol{T_\gamma}{2\sqrt{|N_\gamma|}}.
\]
Hence \index{thetainfty@$\theta_\infty(x)$}
\begin{equation}\label{eq:orbittheta}
\orb(f_\infty;\gamma)\left|\frac{T_\gamma^2-4N_\gamma}{N_\gamma}\right|^{1/2} 
      = \theta_\infty^{\sgn N_\gamma}\legendresymbol{T_\gamma}{2\sqrt{|N_\gamma|}},
\end{equation}
where
\begin{equation}\label{eq:deftheta}
\theta_\infty^\pm(x)=2|x^2\mp 1|^{1/2}g_1^\pm(x)+g_2^\pm(x)=\theta_{\infty,1}^\pm(x)+\theta_{\infty,0}^\pm(x).
\end{equation}
Then $\theta_{\infty,0}^\pm(x)\in C_c^\infty(\RR)$, and $\theta_{\infty,1}^+(x)=|x^2\mp 1|^{1/2}\varphi_1(x)$ with $\varphi_1(x)$ supported in $[-1,1]$ and is smooth inside and up to the boundary, $\theta_{\infty,1}^-(x)=0$. Moreover, near $x=\pm 1$, there exist asymptotic expansions \cite[Lemme 3.3]{langlands2013}
\begin{equation}\label{eq:thetaasymp}
\theta^+_{\infty,1}(\pm(1-x))\sim|x|^{\frac12}\sum_{j=0}^{+\infty}a_j^{\pm}x^j,\quad \theta^+_{\infty,0}(\pm(1-x))\sim\sum_{j=0}^{+\infty}b_j^{\pm}x^j.
\end{equation}

\subsection{Global notations} For $\nu\in \ZZ^r$, we usually denote $\nu_i$ by its $i^{\mathrm{th}}$ component. We define \index{qnu@$q^\nu$} $q^\nu=q_1^{\nu_1}\dots q_r^{\nu_r}$.
For $\alpha,\beta\in \ZZ^r$, we define \index{minab@$\min\{\alpha,\beta\}$} $\min\{\alpha,\beta\}\in \ZZ^r$ such that the $i^{\mathrm{th}}$ component of $\min\{\alpha,\beta\}$ is $\min\{\alpha_i,\beta_i\}$.

For any $y=(y_1,\dots,y_r)\in \QQ_{S_\fin}=\QQ_{q_1}\times\dots\times\QQ_{q_r}$ and $N\in \QQ$, we define \index{thetaqy@$\theta_{q}(y,N)$} \index{1ynormq@$\vert \cdot\vert_q'$} \index{eqy@$\rme_q(y)$}
\[
\theta_{q}(y,N)=\prod_{i=1}^{r}\theta_{q_i}(y_i,N),\qquad|y|_q'=\prod_{i=1}^{r}|y_i|_{q_i}',\qquad \rme_q(y)=\prod_{i=1}^{r}\rme_{q_i}(y_i).
\]
Moreover, for $\tau=(\tau_1,\dots,\tau_r)\in \{0,1\}^r$, we define
\[
\theta_{q,\tau}(y,N)=\prod_{i=1}^{r}\theta_{q_i,\tau_i}(y_i,N).
\]
We usually embed $\QQ$ in $\QQ_S$ or $\QQ_{S_\fin}$ diagonally. 

\subsection{The partial Zagier $L$-function and the approximate functional equation} 
For $\delta\in \ZZ^S$, we define the \emph{partial Zagier $L$-function} as \index{Lssdelta@$L^S(s,\delta)$}
\[
L^{S}(s,\delta)=\sum_{f^2\mid \delta}\frac{1}{f^{2s-1}}L^{S}\left(s,\legendresymbol{\delta/f^2}{\cdot}\right),
\]
where the sum of $f$ is over all $\ZZ_{(S)}^{>0}$ such that $\delta/f^2\in \ZZ^S$, and
\[
L^{S}\left(s,\legendresymbol{\delta/f^2}{\cdot}\right)=\sum_{k\in \ZZ_{(S)}^{>0}}\frac{1}{k^s}\legendresymbol{\delta/f^2}{k}=\prod_{p\notin S}\left(1-\legendresymbol{\delta/f^2}{p}p^{-s}\right)^{-1}
\]
for $s$ such that the series converges absolutely, extended to $\CC$ by analytic continuation. 

Write $\delta=\sigma^2D$, where $D\equiv 0,1\pmod 4$ is the fundamental discriminant. Finally, we set $\tau_\delta=(\sigma^{(q)})^2D=\delta/\sigma_{(q)}^2$.
\begin{theorem}[Functional equation of $L^{S}(s,\delta)$, \cite{cheng2025} Theorem 3.3]\label{thm:funceqndelta}
Suppose that $\delta$ is not a square. We have
\begin{equation}\label{eq:funceqndelta}
L^{S}(s,\delta)=\legendresymbol{|\delta|_{\infty,q}'}{\uppi}^{\frac12-s} \prod_{i=1}^{r}\frac{1-\epsilon_i q_i^{-s}}{1-\epsilon_i q_i^{s-1}} \frac{\Gamma((\iota+1-s)/2)}{\Gamma((\iota+s)/2)}L^{S}(1-s,\delta),
\end{equation}
where $|\delta|_{\infty,q}'=|\tau_\delta|=|\delta|_\infty\prod_{i=1}^{r}|\delta|_{q_i}'$, $\epsilon_i=\epsilon_{i,\delta}=\legendresymbol{\tau_\delta}{q_i}\in \{0,\pm 1\}$, 
$\iota=\iota_\delta=0$ if $\delta>0$, and $\iota=\iota_\delta=1$ if $\delta<0$.
\end{theorem}
\begin{remark}
The theorem also remains valid when $\delta\in \ZZ^S$ is a square, with an identical proof.
\end{remark}
Let \index{fx@$F(x)$}
\[
F(x)=\frac{1}{2K_0(2)}\int_{x}^{+\infty}\rme^{-t-1/t}\frac{\rmd t}{t},
\]
where $K_s(y)$ is the \emph{modified Bessel function of the second kind}. The \emph{Mellin transform} of $F$ is
\begin{equation}\label{eq:mellinf}
\widetilde{F}(s)=\frac{1}{s}\frac{K_s(2)}{K_0(2)}.
\end{equation}
As shown in \cite[Proposition 3.7]{cheng2025}, $\widetilde{F}$ is an odd function, and for $s=\sigma+\rmi t\in \CC$ such that $\sigma_1\leq \sigma\leq \sigma_2$, we have
\begin{equation}\label{eq:frapiddecay}
\widetilde{F}(s)\ll_{\sigma_1,\sigma_2} |s|^{|\sigma|-1}\rme^{-\uppi|t|/2}.
\end{equation}

\begin{theorem}[Approximate functional equation for $L^{S}(s,\delta)$, \cite{cheng2025} Theorem 3.9]\label{thm:afe}
For $\delta\in \ZZ^S$ that is not a square, and any $A>0$, we have
\begin{equation}\label{eq:approximatefe}
\begin{split}
   L^{S}(s,\delta) & =\sum_{f^2\mid \delta}\sum_{k\in \ZZ_{(S)}^{>0}}\frac{1}{f^{2s-1}}\frac{1}{k^s}\legendresymbol{\delta/f^2}{k}F\legendresymbol{kf^2}{A}\\
     & +|\delta|_{\infty,q}'^{\frac12-s}\sum_{f^2\mid \delta}\sum_{k\in \ZZ_{(S)}^{>0}}\frac{1}{f^{1-2s}}\frac{1}{k^{1-s}}\legendresymbol{\delta/f^2}{k}V_{\iota,\epsilon,s}\legendresymbol{kf^2A}{|\delta|_{\infty,q}'},
\end{split}
\end{equation}
where $\epsilon=(\epsilon_1,\dots,\epsilon_r)\in \{-1,0,1\}^r$, and
\[
V_{\iota,\epsilon,s}(x)=\frac{\uppi^{s-1/2}}{2\uppi \rmi}\int_{(\sigma)}\widetilde{F}(u)\prod_{i=1}^{r}\frac{1-\epsilon_iq_i^{-s+u}}{1-\epsilon_iq_i^{s-u-1}}\frac{\Gamma(\frac{\iota+1-s+u}{2})}{\Gamma(\frac{\iota+s-u}{2})}(\uppi x)^{-u}\rmd u,
\]
$\sigma\in \RR$ is chosen such that $\sigma+\Re s>1$, and $(\sigma)$ denotes the vertical contour form $\sigma-\rmi\infty$ to $\sigma+\rmi\infty$.
\end{theorem}
\begin{corollary}[\cite{cheng2025}, Corollary 3.10]\label{cor:afe1}
For $\delta\in \ZZ^S$ that is not a square, and any $A>0$, we have
\begin{equation}\label{eq:approximatefeat1}
L^{S}(1,\delta)=\sum_{f^2\mid \delta}\sum_{k\in \ZZ_{(S)}^{>0}}\frac{1}{kf}\legendresymbol{\delta/f^2}{k}\left(F\legendresymbol{kf^2}{A}+\frac{kf^2}{\sqrt{|\delta|_{\infty,q}'}}V_{\iota,\epsilon}\legendresymbol{kf^2A}{|\delta|_{\infty,q}'}\right),
\end{equation}
where \index{viex@$V_{\iota,\epsilon}(x)$}
\[
V_{\iota,\epsilon}(x)=V_{\iota,\epsilon,1}(x)=\frac{\uppi^{1/2}}{2\uppi \rmi}\int_{(1)}\widetilde{F}(s)\prod_{i=1}^{r}\frac{1-\epsilon_i q_i^{s-1}}{1-\epsilon_i q_i^{-s}}\frac{\Gamma((\iota+s)/2)}{\Gamma((\iota+1-s)/2)}(\uppi x)^{-s}\rmd s.
\]
\end{corollary}

We end this subsection by establishing the rapid decay properties of 
$F(x), V(x)$ and $\widetilde{F}(s), \widetilde{V}(s)$.  Let $\cS$ be the set of smooth functions $\Phi$ on $\lopen 0,+\infty\ropen$ such that for any $k\in \ZZ_{\geq 0}$, $\Phi^{(k)}(x)$ is of rapid decay as $x\to +\infty$ (analog of Schwartz functions). Let $s=\sigma+\rmi t $ be a complex number.
\begin{proposition}\label{prop:propertyf}
We have $F\in \cS$. Moreover, $\widetilde{F}(s)$ has a meromorphic continuation to the whole complex plane, holomorphic except for a simple pole at $s=0$, and is of rapid decay in $t$ when $\sigma$ is fixed.
\end{proposition}
\begin{proof}
By Proposition B.5 of \cite{cheng2025} we have $F\in \cS$. Other assertions follow from \eqref{eq:mellinf} and \eqref{eq:frapiddecay}.
\end{proof}
\begin{proposition}\label{prop:propertyh}
For any $\iota$ and $\epsilon$, we have $V_{\iota,\epsilon}\in \cS$. Moreover, $\widetilde{V_{\iota,\epsilon}}(s)$ is holomorphic on $\Re s>0$, and is of rapid decay in $t$ when $\sigma>0$ is fixed.
\end{proposition}
\begin{proof}
By Proposition B.5 of \cite{cheng2025} we have $V\in \cS$. Now we prove the second assertion.

The Mellin transform of $V_{\iota,\epsilon}$ is
\begin{equation}\label{eq:mellinv}
\widetilde{V_{\iota,\epsilon}}(s)=\widetilde{F}(s)\prod_{i=1}^{r}\frac{1-\epsilon_i q_i^{s-1}}{1-\epsilon_i q_i^{-s}}\frac{\Gamma(\frac{\iota+s}{2})}{\Gamma(\frac{\iota+1-s}{2})}\uppi^{-s+\frac 12}.
\end{equation}

By Stirling formula \cite[Chapter II.0]{tenenbaum2015analytic} we have
\begin{equation}\label{eq:mellinrapid2}
\frac{|\Gamma(\frac{\iota+s}{2})|}{|\Gamma(\frac{\iota+1-s}{2})|}\asymp_A \frac{|t/2|^{\frac{\iota+\sigma-1}{2}}\rme^{-\uppi|t|/4}}{|t/2|^{\frac{\iota-\sigma}{2}} \rme^{-\uppi|t|/4}}=\left|\frac{t}{2}\right|^{\sigma-\frac 12}.
\end{equation}
Also, we have
\begin{equation}\label{eq:mellinrapid3}
\frac{|1-\epsilon_i q_i^{s-1}|}{|1-\epsilon_i q_i^{-s}|}\leq \frac{1+q_i^{A-1}}{1-q_i^{-A}}\ll_A q_i^A\ll_{A,q_i} 1
\end{equation}
for any $s$ such that $\Re s=A$ and $i=1,\dots,r$. 

Since $V_{\iota,\epsilon}\in \cS$, $\widetilde{V_{\iota,\epsilon}}(s)$ is holomorphic on the half plane $\Re s>0$.
Moreover, by the definition \eqref{eq:mellinv} and \eqref{eq:frapiddecay}, \eqref{eq:mellinrapid2}, \eqref{eq:mellinrapid3}, $\widetilde{V_{\iota,\epsilon}}(s)$ is of rapid decay for $t$ when $\sigma>0$ is fixed. 
\end{proof}

\subsection{The generalized Kloosterman sum}\label{subsec:kloosterman}
 The \emph{partial generalized Kloosterman sum} is defined by \index{klkfs@$\Kl_{k,f}^S(\xi,m)$}
\[
\Kl_{k,f}^S(\xi,m)=
  \sum_{\substack{a \bmod kf^2\\ a^2-4m\equiv 0\,(f^2)}}\legendresymbol{(a^2-4m)/f^2}{k}\rme\legendresymbol{a\xi}{kf^2}\rme_{q}\legendresymbol{a\xi}{kf^2}
\]
for $k,f\in \ZZ_{(S)}^{>0}$ and $\xi,m\in \ZZ^S$. 

For any prime $p\notin S$ and $\xi,m\in \ZZ_p$, we define the \emph{local generalized Kloosterman sum} to be \index{klpkf@$\Kl_{p^u,p^v}^{(p)}(\xi,m)$}
\begin{equation}\label{eq:localgeneralizedkloosterman}
\Kl_{p^u,p^v}^{(p)}(\xi,m)=
\sum_{\substack{a \bmod p^{u+2v}\\ a^2-4m\equiv 0\,(p^{2v})}}\legendresymbol{(a^2-4m)/p^{2v}}{p^u}\rme_p\legendresymbol{-a\xi}{p^{u+2v}}.
\end{equation}

Suppose that  $\xi\in \ZZ^S$. Since $\rme(x)\prod_{p}\rme_p(x)=1$ for $x\in \QQ$ and $a\xi/p^{u+2v}\in \ZZ_{\ell}$ if $\ell\notin S\cup\{p\}$, we obtain
\[
\rme_p\legendresymbol{-a\xi}{p^{u+2v}}=\prod_{\ell\notin S}\rme_{\ell}\legendresymbol{a\xi}{p^{u+2v}}^{-1}= \rme\legendresymbol{a\xi}{p^{u+2v}}\rme_{q}\legendresymbol{a\xi}{p^{u+2v}}.
\]
Hence for $\xi\in \ZZ^S$, we have
\begin{equation}\label{eq:localgeneralizedkloostermaneq}
\Kl_{p^u,p^v}^{(p)}(\xi,m)=
\sum_{\substack{a \bmod p^{u+2v}\\ a^2-4m\equiv 0\,(p^{2v})}}\legendresymbol{(a^2-4m)/p^{2v}}{p^u} \rme\legendresymbol{a\xi}{p^{u+2v}}\rme_{q}\legendresymbol{a\xi}{p^{u+2v}}.
\end{equation}

\begin{proposition}\label{prop:equalkloosterman}
Suppose that $\xi,m\in \ZZ_p$. Let $\xi',m'\in \ZZ$ such that $\xi\equiv \xi'\pmod {p^{u+2v}}$ and $m\equiv m'\pmod {p^{u+2v}}$. Then
\[
\Kl_{p^u,p^v}^{(p)}(\xi,m)=\Kl_{p^u,p^v}^{(p)}(\xi',m')=
                               \sum_{\substack{a \bmod p^{u+2v}\\ a^2-4m'\equiv 0\,(p^{2v})}}\legendresymbol{(a^2-4m')/p^{2v}}{p^u} \rme\legendresymbol{a\xi'}{p^{u+2v}}.
\]
\end{proposition} 
The right hand side is precisely the local generalized Kloosterman sum defined by Altu\u{g} in \cite{altug2017}.
\begin{proof}
The first equality follows immediately from the definition \eqref{eq:localgeneralizedkloosterman}. For the second equality, by definition we have
\[
\rme_p\legendresymbol{-a\xi'}{p^{u+2v}}=\rme\left(\left\langle \frac{a\xi'}{p^{u+2v}}\right\rangle_{\!\!p}\right)=\rme\legendresymbol{a\xi'}{p^{u+2v}}.
\]
The result then follows from \eqref{eq:localgeneralizedkloosterman}.
\end{proof}

\section{Contribution of non-elliptic parts}\label{sec:nonelliptic}
Let $\mf{S}$ be the set of places of $\QQ$ and let $\AA$ be the ad\`ele ring of $\QQ$. The (noninvariant) Arthur-Selberg trace formula of $\G=\GL_2$ takes the following form:
\begin{equation}\label{eq:traceformulagl2}
J_{\mathrm{geom}}(f)=J_{\mathrm{spec}}(f),
\end{equation}
where the geometric side is
\begin{equation}\label{eq:geometric}\index{jgeom@$J_{\mathrm{geom}}(f)$}
J_{\mathrm{geom}}(f)=I_\id(f)+I_\el(f)+J_\hyp(f)+J_\unip(f)
\end{equation}
and the spectral side is
\begin{equation}\label{eq:spectral}\index{jspec@$J_{\mathrm{spec}}(f)$}
J_{\mathrm{spec}}(f)=I_{\mathrm{cusp}}(f)+J_{\mathrm{cont}}(f)+\sum_{\mu}\Tr(\mu(f))+\frac14\sum_{\mu}\Tr(M(0,\mu)(\xi_0\otimes\mu)(f)).
\end{equation}
The sum $\mu=\mu_1\boxtimes\mu_2$ runs over all $1$-dimensional representations of $\A(\QQ)\bs \A(\AA)^1= \QQ^\times\bs (\AA^\times)^1\times \QQ^\times\bs (\AA^\times)^1$ with $\mu_1=\mu_2$ (equivalently, runs over all $1$-dimensional representations of $\G(\QQ)\bs \G(\AA)^1$), where $\A$ denotes the diagonal torus of $\G$.

The Arthur-Selberg trace formula appears in numerous references. For example, \cite[Chapter 16]{langlands1970}, \cite[Theorem 6.33]{gelbart1979}, \cite[Theorem 7.14]{knapp1997} and \cite[Theorem 1]{finis2011}. 

We first explain the geometric side. The identity part is
\[\index{iid@$I_\id(f)$}
I_\id(f)=\sum_{z\in \Z(\QQ)}\vol(\G(\QQ)\bs\G(\AA)^1)f(z).
\]

The elliptic part is
\[\index{iell@$I_\el(f)$}
I_\el(f)=\sum_{\gamma\in \G(\QQ)^\#_\el}\vol(\gamma)\orb(f;\gamma),
\]
where $ \G(\QQ)^\#_\el$ denotes the set of elliptic conjugacy classes in $\G(\QQ)$, and
\[
\orb(f;\gamma)=\int_{\G_\gamma(\AA)\bs \G(\AA)} f(g^{-1}\gamma g)\rmd g,
\qquad
\vol(\gamma)=\int_{Z_+\G_\gamma(\QQ)\bs \G_\gamma(\AA)} \rmd g,
\]
where $Z_+$ denotes the connected component of the identity matrix in the center $\Z(\RR)$ of $\G(\RR)$.

The hyperbolic part is
\[\index{jhyp@$J_\hyp(f)$}
J_\hyp(f)=-\frac{1}{2}\sum_{\gamma\in \A(\QQ)_{\reg}}\int_{\A(\AA)\bs \G(\AA)}f(g^{-1}\gamma g)\alpha(H_\B(wg)+H_\B(g))\rmd g,
\]
where $\A$ is the diagonal torus, $\A(\QQ)_\reg$ denotes all the regular elements in $\A(\QQ)$, $w$ denotes the nontrivial element in the Weyl group of $(\G,\A)$, $\alpha$ denotes the positive root in $\mf{sl}_2$, and $H_\B$ denotes the Harish-Chandra map. 

The unipotent part is
\[\index{junip@$J_\mathrm{unip}(f)$}
J_\mathrm{unip}(f)=\sum_{z\in \Z(\QQ)}\fp_{s=1}Z(s,\triv,F_z),
\]
where $\fp$ denotes the finite part, $Z$ denotes the zeta integral in Tate's thesis and for $t\in \AA$,
\[
F_z(t)=\int_{K}f\left(k^{-1}z\begin{pmatrix}
                                 1 & t \\
                                 0 & 1 
                               \end{pmatrix}k\right)\rmd k,
\]
where $K$ denotes the standard maximal compact subgroup of $\G(\AA)$.

Next we provide an explicit description of the spectral side. We follow the notations of \cite{gelbart1979}. The continuous part \index{jcont@$J_\mathrm{cont}(f)$}$J_\mathrm{cont}(f)$ is the sum of the two following terms
\begin{equation}\label{eq:continuouspart1}
-\frac{1}{4\uppi\rmi}\sum_{\mu}\int_{(0)}\frac{m'(s,\mu)}{m(s,\mu)} \Tr\left(\Ind_{\B(\AA)}^{\G(\AA)}(s,\mu)(f)\right)\rmd s
\end{equation}
and
\begin{equation}\label{eq:continuouspart2}
-\frac{1}{4\uppi\rmi}\sum_{v\in \mf{S}}\sum_{\mu}\int_{(0)}\Tr \left(R_v(s,\mu_v)^{-1}R_v'(s,\mu_v)\Ind_{\B(\QQ_v)}^{\G(\QQ_v)}(s,\mu_v)(f_v)\right)\cdot \prod_{w\neq v}\Tr\left(\Ind_{\B(\QQ_w)}^{\G(\QQ_w)}(s,\mu_w)(f_w)\right)\rmd s,
\end{equation}
where $\mu=\mu_1\boxtimes\mu_2$ runs over all $1$-dimensional representations of $\A(\QQ)\bs \A(\AA)^1$. We do \emph{not} require $\mu_1=\mu_2$ for the sum in the continuous part.

The space of induced representations is defined as follows:
As a vector space, $\Ind_{\B(\QQ_v)}^{\G(\QQ_v)}(s,\mu_v)$ is the space of all smooth functions $\psi$ on $\G(\QQ_v)$ that satisfy
\[
\psi(bg)=\mu_1(x)\mu_2(y)\left|\frac{x}{y}\right|_v^{s+1/2}\psi(g)
\] 
for any $g\in \G(\QQ_v)$ and $b=(\begin{smallmatrix} x & z \\ 0 & y \end{smallmatrix})\in \B(\QQ_v)$, and $\G(\QQ_v)$ acts by right translation. The global induced representation is defined similarly. 
$\mu=\mu_1\boxtimes\mu_2$ runs over all $1$-dimensional representations of $\A(\QQ)\bs \A(\AA)^1$, and
\[
m(s,\mu)=\prod_{v}m_v(s,\mu_v)=\prod_{v}\frac{L_v(2s,\mu_1/\mu_2)}{L_v(1+2s,\mu_1/\mu_2) \varepsilon_v(1-2s,\mu_2/\mu_1,\rme_v)} =\frac{\Lambda(1-2s,\mu_2/\mu_1)}{\Lambda(1+2s,\mu_1/\mu_2)},
\]
where $\varepsilon_v$ denotes the $\varepsilon$-factor.
$M_v(s,\mu_v)$ denotes the intertwining operator from the space $\Ind_{\B(\QQ_v)}^{\G(\QQ_v)}(s,\mu_v)$ to $\Ind_{\B(\QQ_v)}^{\G(\QQ_v)}(-s,w(\mu_v))$, where $w(\mu)=\mu_2\boxtimes\mu_1$ if $\mu=\mu_1\boxtimes\mu_2$, and
\[
M(s,\mu)=\prod_{v\in \mf{S}}M_v(s,\mu_v).
\]
Finally, 
$R_v(s,\mu_v)$ is the normalized intertwining operator 
\[
R_v(s,\mu_v)=m_v(s,\mu)^{-1}M_v(s,\mu_v).
\]
If $s=0$ and $\mu$ is trivial, we also denote this representation by $\xi_0$.

\begin{remark}
We use the letter $J$ instead of $I$ in hyperbolic, unipotent and the continuous parts since these terms are not invariant under conjugation, while other terms are invariant distributions. There is also an \emph{invariant trace formula} for $\GL_2$ but we do not need it in this paper.
\end{remark}

The main goal of this section is to prove the following theorem:
\begin{theorem}\label{thm:cuspidalestimate}
For any $\varepsilon>0$, we have
\[
I_{\mathrm{cusp}}(f^n)=I_\el(f^n)-\sum_{\mu}\Tr(\mu(f^n))+O_{f_\infty,f_{q_i},\varepsilon}(n^\varepsilon).
\]
\end{theorem}

\subsection{Estimates of the geometric side}
First we estimate the non-elliptic parts of the geometric side for the function $f^n$.

\begin{proposition}\label{prop:idestimate}
We have
\[
I_\mathrm{id}(f^n)\ll 1,
\]
where the implied constant only depends on $f_\infty$ and $f_{q_i}$.
\end{proposition}
\begin{proof}
Recall that 
\[
I_\id(f^n)=\sum_{z\in \Z(\QQ)}\vol(\G(\QQ)\bs \G(\AA)^1)f^n(z).
\] 
Suppose that $f^n(z)\neq 0$. Then $\mathopen{|}\det z\mathclose{|}_p=p^{-n_p}$ for all $p\notin S$. Hence $\det z=\pm nq^\nu$ for some $\nu\in \ZZ^r$. 

If $n$ is not a square, then $\det z$ is not a square, contradicting $z\in \Z(\QQ)$. Consequently, $f^n(z)= 0$ for all $z\in \Z(\QQ)$ and thus $I_\mathrm{id}(f^n)=0$.

Now suppose that $n$ is a square. In this case, $f^n(z)\neq 0$ only if $z=(\begin{smallmatrix} a & 0 \\  0 & a \end{smallmatrix})$ with $a=\pm n^{1/2}q^{\nu/2}$ for some $\nu/2\in \ZZ^r$. Since
$f_{q_i}(z)$ is compactly supported, there exists $M_i>0$ such that $f_{q_i}(z)=0$ if $\mathopen{|}\det z\mathclose{|}_{q_i}=q_i^{\nu_i}$ does not belong to $[q_i^{-M_i},q_i^{M_i}]$. Hence the choice of $\nu\in \ZZ^r$ is finite and the number is independent of $n$.  
The conclusion then follows from the estimate $f^n(z)\ll_{f_\infty,f_{q_i}} 1$.
\end{proof}

\begin{proposition}\label{prop:unipestimate}
For any $\varepsilon>0$, we have
\[
J_\mathrm{unip}(f^n)\ll n^\varepsilon,
\]
where the implied constant depends only on $f_\infty$, $f_{q_i}$ and $\varepsilon$.
\end{proposition}
\begin{proof}
Recall that 
\[
J_\mathrm{unip}(f^n)=\sum_{z\in \Z(\QQ)}\fp_{s=1}Z(s,\triv,F_z),
\]
where $\fp$ denotes the finite part, $Z$ denotes the zeta integral in Tate's thesis and for $t\in \AA$,
\[
F_z(t)=\int_{K}f^n\left(k^{-1}z\begin{pmatrix}
                                 1 & t \\
                                 0 & 1 
                               \end{pmatrix}k\right)\rmd k,
\]
where $K$ denotes the standard maximal compact subgroup of $\G(\AA)$. First consider the case when $n$ is not a square.
Since
\[
\det\left(k^{-1}z\begin{pmatrix}
                                 1 & t \\
                                 0 & 1 
                               \end{pmatrix}k\right)=\det z,
\]
which is a square, it does not lie in the support of $f^n$ since $f^n(\gamma)\neq 0$ only if $\det \gamma=\pm nq^\nu$, which cannot be a square. Thus $F_z$ is identically $0$ for all $z\in \Z(\QQ)$. Hence $J_\mathrm{unip}(f^n)=0$.

Next, we treat the case when $n$ is a square. In this case, $f^n(z)\neq 0$ (equivalently, $F_z\not\equiv 0$) only if $z=(\begin{smallmatrix} a & 0 \\  0 & a \end{smallmatrix})$ with $a=\pm n^{1/2}q^{\nu/2}$ for some $\nu/2\in \ZZ^r$. Since
$f_{q_i}$ is compactly supported, there exists $M_i>0$ such that 
\[
f_{q_i}\left(k_{q_i}^{-1}z\begin{pmatrix}
                                 1 & t_{q_i} \\
                                 0 & 1 
                               \end{pmatrix}k_{q_i}\right)=0
\] 
if $\mathopen{|}\det z\mathopen{|}_{q_i}=q_i^{\nu_i}\notin[q_i^{-M_i},q_i^{M_i}]$. Hence the choice of $\nu\in \ZZ^r$ is finite and the number is independent of $n$.

We have 
\[
Z(s,\triv,F_z)=\prod_{v\in \mf{S}}Z_v(s,\triv,F_{z,v})
\]
with
\[
F_{z,v}(t_v)=\int_{K_v}f_v^n\left(k_v^{-1}z\begin{pmatrix}
                                 1 & t_v \\
                                 0 & 1 
                               \end{pmatrix}k_v\right)\rmd k_v,
\]
and
\[
Z_v(s,\triv,F_{z,v})=\int_{\QQ_v^\times}F_{z,v}(t_v)|t_v|_v^{s}\rmd^\times t_v.
\]

\underline{\emph{Case 1:}}\ \ $v=p\notin S$.

In this case, $f_v^n$ is bi-$K_v$-invariant. Hence 
\[
F_{z,v}(t_v)=f_v^n\left(n^{1/2}q^{\nu/2}\begin{pmatrix}
                                 1 & t_v \\
                                 0 & 1 
                               \end{pmatrix}\right).
\]
Thus $F_{z,v}(t_v)=p^{-n_p/2}$ if $|t_p|_p\leq p^{n_p/2}$, and is $0$ otherwise. Hence
\[
Z_v(s,\triv,F_{z,v})=p^{-n_p/2}\sum_{u=-n_p/2}^{+\infty}p^{-us}=\frac{p^{-n_p(1-s)/2}}{1-p^{-s}}.
\]

\underline{\emph{Case 2:}}\ \ $v\in S$.

In this case, the integral defining $Z_v(s,\triv,F_{z,v})$ converges absolutely on each compact subset of $\Re s>0$. Hence $Z_v(s,\triv,F_{z,v})$ is holomorphic on $\Re s>0$. 

Therefore
\begin{align*}
Z(s,\triv,F_z)&=\prod_{v\in S}Z_v(s,\triv,F_{z,v})\prod_{p\notin S}\frac{p^{-n_p(1-s)/2}}{1-p^{-s}} \\
&=Z_\infty(s,\triv,F_{z,\infty})\prod_{p\in S}Z_p(s,\triv,F_{z,p})(1-p^{-s})\zeta(s)n^{(s-1)/2}.
\end{align*}
Hence
\begin{align*}
\fp_{s=1}Z(s,\triv,F_z)=&Z_\infty(1,\triv,F_{z,\infty})\prod_{p\in S}Z_p(1,\triv,F_{z,p})(1-p^{-1})\fp_{s=1}\zeta(s)\\
+&\left.\frac{\rmd}{\rmd s}\right|_{s=1}\left(Z_\infty(s,\triv,F_{z,\infty})\prod_{p\in S}Z_p(s,\triv,F_{z,p})(1-p^{-s})n^{(s-1)/2}\right)\\
\ll &1+\left.\frac{\rmd}{\rmd s}\right|_{s=1}\left(Z_\infty(s,\triv,F_{z,\infty})\prod_{p\in S}Z_p(s,\triv,F_{z,p})(1-p^{-s})\right)\\
+&Z_\infty(1,\triv,F_{z,\infty})\prod_{p\in S}Z_p(1,\triv,F_{z,p})(1-p^{-1})\log n 
\ll \log n\ll n^\varepsilon,
\end{align*}
where the implied constant depends only on $f_\infty$, $f_{q_i}$ and $\varepsilon$. Since we have proved that the sum over $\nu$ is finite, we obtain our conclusion.
\end{proof}

\begin{proposition}\label{prop:hypestimate}
For any $\varepsilon>0$, we have
\[
J_\hyp(f^n)\ll n^\varepsilon,
\]
where the implied constant depends only on $f_{\infty}$, $f_{q_i}$ and $\varepsilon$.
\end{proposition}
\begin{proof}
Recall that
\[
J_\hyp(f^n)=-\frac{1}{2}\sum_{\gamma\in \A(\QQ)_{\reg}}\int_{\A(\AA)\bs \G(\AA)}f^n(g^{-1}\gamma g)\alpha(H_\B(wg)+H_\B(g))\rmd g,
\]
where $\A$ is the diagonal torus, $\A(\QQ)_\reg$ denotes all the regular elements in $\A(\QQ)$, $w$ denotes the nontrivial element in the Weyl group of $(\G,\A)$, $\alpha$ denotes the positive root in $\mf{sl}_2$, and $H_\B$ denotes the Harish-Chandra map. 

Suppose that $\gamma\in \A(\QQ)_\reg$ contributes to this sum. Then there exists $g\in \A(\AA)\bs\G(\AA)$ such that $f^n(g^{-1}\gamma g)\neq 0$. Hence $f_v^n(g_v^{-1}\gamma g_v)\neq 0$ for all $v\in \mf{S}$. Thus $\gamma$ is integral in $\G(\QQ_p)$ and $\mathopen{|}\det\gamma\mathclose{|}_p=\mathopen{|}\det(g_\ell^{-1}\gamma g_\ell)\mathclose{|}_p=p^{-n_p}$ for all $p\notin S$. Hence the eigenvalues $\gamma_1,\gamma_2$ of $\gamma$ both lie in $\ZZ^S$ and $\det\gamma=\pm nq^\nu$ for some $\nu=(\nu_1,\dots,\nu_r)\in \ZZ^r$. 

Using the Iwasawa decomposition, we can write $g=uk$, where $u\in \N(\AA)$ and $k\in K$, $K$ denotes the standard maximal compact subgroup of $\G(\AA)$ and $\N$ denotes the unipotent radical. By assumption, we have $f^n(g^{-1}\gamma g)=f^n(k^{-1}u^{-1}\gamma uk)\neq 0$. Hence for all $v\in \mf{S}$ we have $f_v^n(k_v^{-1}u_v^{-1}\gamma u_vk_v)\neq 0$. Thus 
\[
u_v^{-1}\gamma u_v\in k_v\supp(f_v^n)k_v\subseteq K_v\supp(f_v^n)K_v.
\]

For $\gamma=(\begin{smallmatrix} \gamma_1 & 0 \\ 0 & \gamma_2 \end{smallmatrix})$ and $u_v=(\begin{smallmatrix} 1 & x_v \\ 0 & 1 \end{smallmatrix})$, we have
\begin{equation}\label{eq:conjugateunipotent}
u_v^{-1}\gamma u_v=\begin{pmatrix}
           \gamma_1 & x_v(\gamma_2-\gamma_1) \\
           0 & \gamma_2
         \end{pmatrix}.
\end{equation}

For $v=q_i\in S$, $f_v^n=f_{q_i}$ is compactly supported. Hence $u_v^{-1}\gamma u_v$ lies in a compact set. Thus there exists $M_i$ (independent of $n$) such that $|\gamma_1|_{q_i},|\gamma_2|_{q_i}\in [q_i^{-M_i},q_i^{M_i}]$ for $\gamma$ contributing the sum. 

For $v=\infty$, $f_v^n=f_\infty$ is compactly supported modulo $Z_+$. Thus $u_v^{-1}\gamma u_v$ lies in a set that is compact modulo $Z_+$. Hence there exists $\delta>1$ such that 
\[
\delta^{-1}<\left|\frac{\gamma_1}{\gamma_2}\right|<\delta.
\]
for $\gamma$ contributing to the sum.

Since $\gamma_1,\gamma_2\in \ZZ^S$, $\gamma_1\gamma_2=\pm nq^\nu$, we know that $\gamma_1=\pm n_1q^{\alpha}$ and $\gamma_2=\pm  n_2q^{\beta}$ such that $n_1,n_2$ are divisors of $n$. Moreover, $\alpha$ and $\beta$ both have finitely many choices (independent of $n$). Thus the number of such $\gamma$ is $O(\bm{d}(n))$, where $\bm{d}(n)$ is the number of divisors of $n$.

Now we consider the integral
\begin{equation}\label{eq:hyperbolicsingle}
\int_{\A(\AA)\bs \G(\AA)}f^n(g^{-1}\gamma g)\alpha(H_\B(wg)+H_\B(g))\rmd g.
\end{equation}

By Iwasawa decomposition, the integral above can be rewritten as
\[
\int_{\N(\AA)}\int_{K}f^n(k^{-1}u^{-1}\gamma uk)\rmd k \alpha(H_\B(wu))\rmd u,
\]
where $K$ denotes the standard maximal compact subgroup of $\G(\AA)$. Let
\[
F^n(g)=\int_{K}f^n(k^{-1}gk)\rmd k=\prod_{v\in \mf{S}}F_v^n(g_v),
\]
where $F_v^n(g_v)=\int_{K_v}f_v^n(k_v^{-1}g_vk_v)\rmd k_v$. Thus \eqref{eq:hyperbolicsingle} can be rewritten as
\begin{equation}\label{eq:hyperbolicsingle2}
\int_{\N(\AA)}F^n(u^{-1}\gamma u)\alpha(H_\B(wu))\rmd u.
\end{equation}

By definition of $f_v^n$, we have $F_v^n(g_v)\ll_{f_\infty,f_{q_i}} 1$ if $v\in S$, and $F_v^n(g_v)\ll p^{-n_p/2}$ if $v=p\notin S$. Moreover, $F_v^n$ is supported in $\M_2(\ZZ_p)$ for $v=p\notin S$, compactly supported for $v=p\in S$, and compactly supported modulo the center for $v=\infty$. 

For $\gamma=(\begin{smallmatrix} \gamma_1 & 0 \\ 0 & \gamma_2 \end{smallmatrix})$ and $u_v=(\begin{smallmatrix} 1 & x_v \\ 0 & 1 \end{smallmatrix})$, by \eqref{eq:conjugateunipotent}, there exist $L_0,L_1,\dots,L_r>0$ (independent of $n$) such that $F^n_v(u^{-1}\gamma u)\neq 0$ only if 
\[
|x_{q_i}(\gamma_1-\gamma_2)|_{q_i}\leq q_i^{L_i}
\]
for all $i$,
\[
|x_p(\gamma_1-\gamma_2)|_p\leq 1
\]
for all $p\notin S$, and
\[
\frac{|x_\infty(\gamma_1-\gamma_2)|}{\sqrt{|\gamma_1\gamma_2|}}\leq L_0.
\]
Moreover, by the direct formula for the Iwasawa decomposition for $wu$, we obtain
\[
-\alpha(H_\B(wu))=\log(1+x_\infty^2)-2\sum_{v_p(x_p)<0}v_p(x_p)\log p\geq 0.
\]

If we define
\[
D=\left\{x\in \AA\,\middle|\,|x_\infty|\leq L_0\frac{\sqrt{|\gamma_1\gamma_2|}}{|\gamma_1-\gamma_2|},\,|x_{q_i}|_{q_i}\leq q_i^{L_i+v_{q_i}(\gamma_1-\gamma_2)},\,|x_p|_p\leq p^{v_p(\gamma_1-\gamma_2)}\text{\ for $p\notin S$}\right\},
\]
then we have
\begin{align*}
  \eqref{eq:hyperbolicsingle2} & \ll n^{-\frac12}\int_{D} \left(\log(1+x_\infty^2)-2\sum_{v_p(x_p)<0}v_p(x_p)\log p\right)\rmd x\\
   & \ll n^{-\frac12}\frac{\sqrt{|\gamma_1\gamma_2|}}{|\gamma_1-\gamma_2|}\prod_{p}p^{v_p(\gamma_1-\gamma_2)} \left(\log\left(1+\frac{|\gamma_1\gamma_2|}{|\gamma_1-\gamma_2|^2}L_0^2\right)+2\sum_{i=1}^{r}L_i\log q_i+2\sum_{p}v_p(\gamma_1-\gamma_2)\log p\right)\\
   &\ll \log\left(|\gamma_1-\gamma_2|^2+|\gamma_1\gamma_2|L_0^2\right)+1,
\end{align*}
where in the last step we used $|\gamma_1\gamma_2|\asymp n$. Since $|\gamma_1/\gamma_2|\in \lopen \delta^{-1},\delta\ropen$, we obtain
$
|\gamma_1-\gamma_2|\ll \sqrt{|\gamma_1\gamma_2|}
$.
Therefore $\eqref{eq:hyperbolicsingle2}\ll \log n$. Finally,
\[
J_\hyp(f^n)\ll\bm{d}(n)\log n\ll n^\varepsilon,
\]
where the implied constant depends only on $f_\infty$, $f_{q_i}$ and $\varepsilon$.
\end{proof}

\subsection{Estimates of the spectral side}
\begin{proposition}\label{prop:eisensteinestimate}
For any $\varepsilon>0$, we have
\begin{equation}\label{eq:discretecontinuous1}
\sum_{\mu}\Tr((\xi_0\otimes\mu)(f^n))\ll n^\varepsilon,
\end{equation}
and
\begin{equation}\label{eq:discretecontinuous2}
\sum_{\mu}\Tr(M(0,\mu)(\xi_0\otimes\mu)(f^n))\ll n^\varepsilon,
\end{equation}
where the implied constants depend only on $f_\infty$, $f_{q_i}$ and $\varepsilon$.
\end{proposition}
\begin{proof}
$\Ind_{\B(\AA)}^{\G(\AA)}(s,\mu)$ is irreducible when $\Re s=0$. Hence by Schur's lemma, $M(0,\mu)$ is a scalar multiplication. Moreover, by functional equation of the intertwining operator (see Theorem (4.19) of \cite{gelbart1979}, for example), $M(0,\mu)^2=\id$. Hence $M(0,\mu)=\pm\id$. Therefore, 
\begin{equation}\label{eq:intertwining}
\Tr(M(0,\mu)(\xi_0\otimes\mu)(f))=\pm\Tr((\xi_0\otimes\mu)(f)).
\end{equation}
By Corollary 8.7 of \cite{cheng2025}, the sum over $\mu$ in \eqref{eq:discretecontinuous1} is finite and is independent of $n$. Using \eqref{eq:intertwining}, the sum over $\mu$ in \eqref{eq:discretecontinuous2} is also finite. To prove the proposition, it suffices to show that 
\[
\Tr((\xi_0\otimes\mu)(f^n))\ll_{f_{\infty},f_{q_i},\varepsilon} n^\varepsilon
\]
for fixed $\mu$.

By Proposition 8.3, Proposition 8.4 and Proposition 8.5 in \cite{cheng2025}, we have
\begin{align*}
    \Tr((\xi_0\otimes\mu)(f^n))
    = &4\left(\int_{|x|>1}\frac{\theta_\infty^+(x)}{\sqrt{x^2-1}}\rmd x+\mu_{\infty}(-1)\int_{\RR}\frac{\theta_\infty^-(x)}{\sqrt{x^2+1}}\rmd x\right)\prod_{p\mid n}\mu_p(n)(n_p+1)\\
    \times & 2^r\sum_{\nu\in \ZZ^r}q^{\nu}\prod_{i=1}^{r}\left[\left(1-\frac{1}{q_i}\right)^{-1} \int_{|N_i|_{q_i}= q_i^{-\nu_i}}\int_{Y_{i,1}(N_i)} \frac{\theta_{q_i}(T_i,N_i)}{\sqrt{|T_i^2-4N_i|'_{q_i}}}\mu_{q_i}(N_i)\rmd T_i\rmd N_i\right]\\
    \ll &\prod_{p\mid n}(n_p+1)=\bm{d}(n)\ll n^\varepsilon,
\end{align*}
where the implied constant depends only on $f_{\infty}$, $f_{q_i}$, and $\varepsilon$.
\end{proof}

Finally we consider the continuous part.
\begin{proposition}\label{prop:continuouspart}
For any $\varepsilon>0$, we have 
\[
J_{\mathrm{cont}}(f^n)\ll n^\varepsilon,
\]
where the implied constant depends only on $f_{\infty},f_{q_i}$ and $\varepsilon$.
\end{proposition}
We need to do some preparations for proving this proposition.

Suppose that $\varphi=\bigotimes_{v\in \mf{S}}'\varphi_v$ is a smooth, compactly supported function on $\G(\AA)$ that is trivial on $Z_+$. By modifying the proof of Proposition 8.2 in \cite{cheng2025}, we deduce that:

\begin{proposition}\label{prop:eisensteinfolding}
We have
\begin{equation}\label{eq:inducetraceglobal}
\Tr\left(\Ind_{\B(\AA)}^{\G(\AA)}(s,\mu)(\varphi)\right)=\int_{Z_+\bs \A(\AA)}\frac{|x-y|}{|xy|^{1/2}}\left|\frac xy\right|^{s}\mu_1(x)\mu_2(y)\int_{\A(\AA)\bs \G(\AA)}\varphi(g^{-1}tg)\rmd g \rmd t,
\end{equation}
where $\A$ is the diagonal torus in $\G$ with $t=(\begin{smallmatrix} x & 0 \\  0 & y \end{smallmatrix})$. Moreover, for nonarchimedean $v\in \mf{S}$,
\begin{equation}\label{eq:inducetracenonarchimedean}
\Tr\left(\Ind_{\B(\QQ_v)}^{\G(\QQ_v)}(s,\mu_v)(\varphi_v)\right)=\int_{ \A(\QQ_v)}\frac{|x-y|_v}{|xy|_v^{1/2}}\left|\frac xy\right|_v^{s}\mu_{1,v}(x)\mu_{2,v}(y)\int_{\A(\QQ_v)\bs \G(\QQ_v)}\varphi_v(g^{-1}tg)\rmd g \rmd t,
\end{equation}
while for the archimedean place,
\begin{equation}\label{eq:inducetracearchimedean}
\Tr\left(\Ind_{\B(\RR)}^{\G(\RR)}(\mu_\infty,s)(\varphi_\infty)\right)=\int_{ Z_+\bs\A(\RR)}\frac{|x-y|}{|xy|^{1/2}}\left|\frac xy\right|^{s}\mu_{1,\infty}(x)\mu_{2,\infty}(y)\int_{\A(\RR)\bs \G(\RR)}\varphi_\infty(g^{-1}tg)\rmd g \rmd t.
\end{equation}
\end{proposition}

Since
\[
\Tr\left(\Ind_{\B(\AA)}^{\G(\AA)}(s,\mu)(\varphi)\right) =\prod_{v}\Tr\left(\Ind_{\B(\QQ_v)}^{\G(\QQ_v)}(s,\mu_v)(\varphi_v)\right),
\]
it suffices to compute the local traces.

\begin{proposition}\label{prop:unramifiedtraceeisenstein}
\begin{enumerate}[itemsep=0pt,parsep=0pt,topsep=0pt,leftmargin=0pt,labelsep=2.5pt,itemindent=15pt,label=\upshape{(\arabic*)}]
  \item If $\mu_{1,p}$ or $\mu_{2,p}$ is ramified, then
  \[
  \int_{ \A(\QQ_p)}\frac{|x-y|_p}{|xy|_p^{1/2}}\left|\frac xy\right|_p^{s}\mu_{1,p}(x)\mu_{2,p}(y)\orb(\triv_{\cX_p^{m}};t)\rmd t=0.
  \]
  \item If both $\mu_{1,p}$ and $\mu_{2,p}$ are unramified, then
  \[
  \int_{\A(\QQ_p)}\frac{|x-y|_p}{|xy|_p^{1/2}}\left|\frac xy\right|_p^{s}\mu_{1,p}(x)\mu_{2,p}(y)\orb(\triv_{\cX_p^{m}};t)\rmd t=p^{m/2}\sum_{u=0}^{m}\mu_{1,p}(p^u)\mu_{2,p}(p^{m-u})p^{(2u-m)s}.
  \]
\end{enumerate} 
\end{proposition}
\begin{proof}
By Theorem 2.7 of \cite{cheng2025} we know that $\orb(\triv_{\cX_p^{m}};t)=0$ unless $x$ and $y$ are integral and $|xy|_p=p^{-m}$, and in this case we have  $\orb(\triv_{\cX_p^{m}};t)=p^{k_t}=|x-y|_p^{-1}$. Hence
\begin{align*}
  &\int_{ \A(\QQ_p)}\frac{|x-y|_p}{|xy|_p^{1/2}}\left|\frac xy\right|_p^{s}\mu_{1,p}(x)\mu_{2,p}(y)\orb(\triv_{\cX_p^{m}};t)\rmd t\\
  =&p^{m/2}\sum_{u=0}^{m}\int_{x\in p^u\ZZ_p^\times}\int_{y\in p^{m-u}\ZZ_p^\times}\left|\frac xy\right|_p^{s}\mu_{1,p}(x)\mu_{2,p}(y)\rmd^\times y\rmd^\times x\\
  =&p^{m/2}\sum_{u=0}^{m}\mu_{1,p}(p^u)\mu_{2,p}(p^{m-u})p^{(2u-m)s}\int_{a\in\ZZ_p^\times}\int_{b\in \ZZ_p^\times}\mu_{1,p}(a)\mu_{2,p}(b)\rmd^\times a\rmd^\times b,
\end{align*}
which equals 
\[
p^{m/2}\sum_{u=0}^{m}\mu_{1,p}(p^u)\mu_{2,p}(p^{m-u})p^{(2u-m)s}
\] 
if $\mu_{1,p}$ and $\mu_{2,p}$ are unramified and equals $0$ if $\mu_{1,p}$ or $\mu_{2,p}$ is ramified.
\end{proof}

\begin{lemma}\label{lem:characterfinite}
For any prime $p\in S$, there exists $M>0$ such that for all $s$,
\[
\Tr\left(\Ind_{\B(\QQ_p)}^{\G(\QQ_p)}(s,\mu_p)(f_p)\right)=0
\] 
if $\mu_{1,p}$ or $\mu_{2,p}$ is not trivial on $1+p^{M}\ZZ_{p}$.
\end{lemma}
\begin{proof}
Let
\[
\Theta_{p}(t)=\frac{|x-y|_{p}}{|xy|_{p}^{1/2}}\orb(f_{p};t).
\]
By the proof of Lemma 8.6 in \cite{cheng2025}, there exist $M>0$ such that for any $a\in 1+p^{M}\ZZ_{p}$ and $t\in \A(\QQ_{p})$, we have
\[
\Theta_{p}(t)=\Theta_{p}\left(\begin{pmatrix}
                      a & 0 \\
                      0 & 1 
                    \end{pmatrix}t\right).
\]
Therefore for any $a\in 1+p^{M}\ZZ_{p}$, by \eqref{eq:inducetracenonarchimedean} we have
\begin{align*}
\Tr\left(\Ind_{\B(\QQ_p)}^{\G(\QQ_p)}(s,\mu_p) (f_p)\right)&=\int_{\A(\QQ_{p})}\Theta_{p}(t)\left|\frac xy\right|_p^{s}\mu_{1,p}(x)\mu_{2,p}(y)\rmd t\\
&=\int_{\A(\QQ_{p})}\Theta_{p}(t)\left|\frac xy\right|_p^{s}\mu_{1,p}(ax)\mu_{2,p}(y)\rmd t=\mu_{1,p}(a)\Tr\left(\Ind_{\B(\QQ_p)}^{\G(\QQ_p)}(s,\mu_p)(f_p)\right).
\end{align*}
If $\Tr\left(\Ind_{\B(\QQ_p)}^{\G(\QQ_p)}(s,\mu_p)(f_p)\right)\neq 0$, then $\mu_{1,p}(a)=1$ for all $a\in 1+p^{M}\ZZ_{p}$. Similarly, there exists $M'$ such that if $\mu_{2,p}$ is not trivial on $1+p^{M'}\ZZ_{p}$, then $\Tr\left(\Ind_{\B(\QQ_p)}^{\G(\QQ_p)}(s,\mu_p)(f_p)\right)=0$ for all $s$.
\end{proof}

\begin{corollary}\label{cor:sumfinite}
The sum over $\mu$ in \eqref{eq:continuouspart1} is finite.
\end{corollary}
\begin{proof}
By \autoref{lem:characterfinite}, there exist $M_1,\dots,M_r\in \ZZ_{>0}$ such that 
\[
\Tr\left(\Ind_{\B(\QQ_{q_i})}^{\G(\QQ_{q_i})}(s,\mu_{q_i})(f_{q_i})\right)=0
\]
if $\mu_{1,q_i}$ or  $\mu_{2,q_i}$ is not trivial on $1+q_i^{M_i}\ZZ_{q_i}$.

Suppose that $\mu=\mu_1\boxtimes \mu_2$ is a character on $\A(\QQ)\bs \A(\AA)^1=\QQ^\times\bs (\AA^\times)^1\times\QQ^\times\bs (\AA^\times)^1$ contributing the sum \eqref{eq:continuouspart1}. Then we must have
\[
\Tr\left(\Ind_{\B(\QQ_p)}^{\G(\QQ_p)}(s,\mu_p)(f_p)\right)\neq 0
\] 
for any finite place $p$. This means that $\mu_{i,p}$ is unramified for each $p\notin S$ and $\mu_{q_i}$ is trivial on $1+q_i^{M_i}\ZZ_{q_i}$ for each $i$. Hence $\mu_1$ and $\mu_2$ can be considered as characters on
\[
\left(\RR_{>0}\QQ^\times\prod_{p\notin S}\ZZ_p^\times \prod_{i=1}^{r}(1+q_i^{M_i}\ZZ_{q_i})\right)\bs \AA^\times,
\]
which is a finite group by class field theory. Thus, the sum in \eqref{eq:continuouspart1} over $\mu=\mu_1\boxtimes \mu_2$ is finite.
\end{proof}

\begin{proof}[Proof of \autoref{prop:continuouspart}]
We first prove that $\eqref{eq:continuouspart1}\ll_\varepsilon n^\varepsilon$.
By \autoref{cor:sumfinite}, the sum over $\mu$ is finite. Hence it suffices to prove that the integral converges absolutely for each fixed $\mu$ contributing the sum. By \autoref{prop:unramifiedtraceeisenstein} and the definition of $f^n$, we have
\begin{align*}
  &\int_{(0)}\frac{m'(s,\mu)}{m(s,\mu)}\Tr\left(\Ind_{\B(\AA)}^{\G(\AA)}(s,\mu)(f^n)\right)\rmd s 
   = \int_{(0)}\frac{m'(s,\mu)}{m(s,\mu)} \prod_{v}\Tr\left(\Ind_{\B(\QQ_v)}^{\G(\QQ_v)}(s,\mu_v)(f^n_v)\right)\rmd s\\
    =&\int_{(0)}\frac{m'(s,\mu)}{m(s,\mu)} \prod_{v\in S}\Tr\left(\Ind_{\B(\QQ_v)}^{\G(\QQ_v)}(s,\mu_v)(f_v)\right)\prod_{p\notin S}\left(\sum_{u=0}^{n_p}\mu_{1,p}(p^u)\mu_{2,p}(p^{n_p-u}) p^{(2u-n_p)s}\right)\rmd s.
\end{align*}
By \eqref{eq:inducetracenonarchimedean}, for any $v\in S$ and $\Re s=0$ we have
\begin{align*}
\left|\Tr\left(\Ind_{\B(\QQ_v)}^{\G(\QQ_v)}(s,\mu_v)(f_v)\right)\right|&\leq\int_{ \A(\QQ_v)}\frac{|x-y|_v}{|xy|_v^{1/2}}\int_{\A(\QQ_v)\bs \G(\QQ_v)}|f_v(g^{-1}tg)|\rmd g \rmd t\\
&=\int_{ \A(\QQ_v)}\frac{|x-y|_v}{|xy|_v^{1/2}}\orb(|f_v|,t)\rmd t\ll 1
\end{align*}
since the integrand is continuous and compactly supported. 

Now we consider the archimedean part. By Theorem 2.13 of \cite{cheng2025}, the function
\[ 
\frac{|x-y|}{|xy|^{1/2}}\int_{\A(\RR)\bs \G(\RR)}\varphi_v(g^{-1}tg)\rmd g 
\]
is smooth and compactly supported in $t\in Z_+\bs \A(\RR)$.
Hence by \eqref{eq:inducetracearchimedean}, 
\[
\Phi(\rmi \tau)=\Tr\left(\Ind_{\B(\RR)}^{\G(\RR)}(\mu_\infty,\rmi \tau)(\varphi_\infty)\right)
\] 
is a Fourier transform 
of a smooth and compactly supported function for $s=\rmi \tau$. Thus $\Phi(\rmi \tau)$ is a Schwartz function in $t$. Therefore
\begin{equation}\label{eq:estimatects1}
\begin{split}
  \left|\int_{(0)}\frac{m'(s,\mu)}{m(s,\mu)}\Tr\left(\Ind_{\B(\AA)}^{\G(\AA)}(s,\mu)(f^n)\right)\rmd s \right|
    &\ll_{f_\infty, f_{q_i}} \int_{(0)}|\Phi(s)|\left|\frac{m'(s,\mu)}{m(s,\mu)}\right|\prod_{p\notin S}\left(\sum_{u=0}^{n_p}|p^{(2u-n_p)s}|\right) \rmd |s|\\
    &\ll_{f_\infty, f_{q_i}} \int_{(0)}|\Phi(s)|\left|\frac{m'(s,\mu)}{m(s,\mu)}\right|\prod_{p\notin S}(n_p+1)\rmd |s|\\
    &\ll_{f_\infty, f_{q_i}} \bm{d}(n)\ll_{f_\infty, f_{q_i},\varepsilon} n^\varepsilon.
\end{split}
\end{equation}
The only nontrivial step is the last estimate. It is deduced as follows: Recall that
\[
m(s,\mu)=\prod_{v}\frac{L_v(1-2s,\chi^{-1})}{L_v(1+2s,\chi)}=\frac{\Lambda(1-2s,\chi^{-1})}{\Lambda (1+2s,\chi)},
\]
where $\chi=\mu_1/\mu_2$ can be considered as a Dirichlet character, and $\Lambda$ denotes the complete $L$-function. Hence
\[
\frac{m'(s,\mu)}{m(s,\mu)}=-2\frac{\Lambda'(1-2s,\chi^{-1})}{\Lambda(1-2s,\chi^{-1})}-2\frac{\Lambda' (1+2s,\chi)}{\Lambda (1+2s,\chi)}.
\]
Since $\Phi(\rmi t)$ is a Schwartz function, it suffices to show that 
\[
\frac{\Lambda'(1-\rmi t,\chi^{-1})}{\Lambda(1-\rmi t,\chi^{-1})}\quad\text{and}\quad
\frac{\Lambda' (1+\rmi t,\chi)}{\Lambda (1+\rmi t,\chi)}
\]
both have polynomial growth for $t\in \RR$. By symmetry it suffices to show this for the second function. We have
\[
\frac{\Lambda' (1+\rmi t,\chi)}{\Lambda (1+\rmi t,\chi)}=\frac{L' (1+\rmi t,\chi)}{L (1+\rmi t,\chi)}+\frac{1}{2}\frac{\Gamma'(\frac{\rmi t+\iota}{2})}{\Gamma(\frac{\rmi t+\iota}{2})},
\]
where $\iota\in \{0,1\}$.
By the results in Part II, \SSec 8.3 of \cite{tenenbaum2015analytic} (for $\chi$ nontrivial) and Part II, \SSec 3.10 of loc. cit. (for $\chi$ trivial) we know that $L'(1+\rmi t,\chi)/L(1+\rmi t,\chi)$ has polynomial growth. By Weierstrass product formula for $\Gamma(s)$ (Part II, Theorem 0.6 of loc. cit.) we have
\[
-\frac{\Gamma'(s)}{\Gamma(s)}=\upgamma+\frac{1}{s}-\sum_{j=1}^{+\infty}\frac{s}{j(s+j)},
\]
where $\upgamma$ is the Euler-Mascheroni constant. Hence for any $s=\sigma+\rmi t$ with $\sigma\geq 0$ and $|t|\gg 1$ we have
\[
\frac{\Gamma'(s)}{\Gamma(s)}\ll_\sigma 1+\sum_{j=1}^{+\infty}\frac{|s|}{j^2}\ll_\sigma |t|
\]
and hence $\Gamma'(\frac{\rmi t+\iota}{2})/\Gamma(\frac{\rmi t+\iota}{2})$ has polynomial growth. Consequently, $\Lambda'(1+\rmi t,\chi)/\Lambda (1+\rmi t,\chi)$ has polynomial growth. Hence $m'(s,\mu)/m(s,\mu)$ has polynomial growth. This establishes \eqref{eq:estimatects1}. Hence \eqref{eq:continuouspart1} converges absolutely and $\eqref{eq:continuouspart1}\ll n^\varepsilon$. 

Now we prove the estimate for \eqref{eq:continuouspart2}. By the definition of $f^n$ we have
\begin{align*}
   \eqref{eq:continuouspart2}= & -\frac{1}{4\uppi\rmi}\sum_{v\in S}\sum_{\mu}\int_{(0)}\Tr \left(R_v(s,\mu_v)^{-1}R_v'(s,\mu_v)\Ind_{\B(\QQ_v)}^{\G(\QQ_v)}(s,\mu_v)(f^n_v)\right) \\
     &\times \prod_{\substack{w\in S\\w\neq v}}\Tr\left(\Ind_{\B(\QQ_w)}^{\G(\QQ_w)}(s,\mu_w)(f^n_w)\right)\prod_{p\notin S}\left(\sum_{u=0}^{n_p}\mu_{1,p}(p^u)\mu_{2,p}(p^{n_p-u}) p^{(2u-n_p)s}\right)\rmd s.
\end{align*}

Since the spectral side converges absolutely (see \cite{muller2004} for a proof for the case $\G=\GL_n$ and \cite{muller2011} for general $\G$), $J_{\mathrm{cont}}(f)$ converges absolutely. We have proved that \eqref{eq:continuouspart1} converges absolutely. Hence \eqref{eq:continuouspart2} converges absolutely for any $n$. In particular, by taking $n=1$, we have
\[
\sum_{v\in S}\sum_{\mu}\int_{(0)}\left|\Tr \left(R_v(s,\mu_v)^{-1}R_v'(s,\mu_v)\Ind_{\B(\QQ_v)}^{\G(\QQ_v)}(s,\mu_v)(f_v)\right)\right|\prod_{\substack{w\in S\\w\neq v}}\left|\Tr\left(\Ind_{\B(\QQ_w)}^{\G(\QQ_w)}(s,\mu_w)(f_w)\right)\right|\rmd |s|<+\infty.
\]

Since each $\mu_v$ is unitary, we obtain
\begin{align*}
   \eqref{eq:continuouspart2}\ll & \sum_{v\in S}\sum_{\mu}\int_{(0)}\left|\Tr \left(R_v(s,\mu_v)^{-1}R_v'(s,\mu_v)\Ind_{\B(\QQ_v)}^{\G(\QQ_v)}(s,\mu_v)(f^n_v)\right)\right| \\
     &\times \prod_{\substack{w\in S\\w\neq v}}\left|\Tr\left(\Ind_{\B(\QQ_w)}^{\G(\QQ_w)}(s,\mu_w)(f^n_w)\right)\right|\prod_{p\notin S}\left(\sum_{u=0}^{n_p}|p^{(2u-n_p)s}|\right)\rmd |s|\\
   \ll &\sum_{v\in S}\sum_{\mu}\int_{(0)}\left|\Tr \left(R_v(s,\mu_v)^{-1}R_v'(s,\mu_v)\Ind_{\B(\QQ_v)}^{\G(\QQ_v)}(s,\mu_v)(f_v)\right)\right| \\
    &\times\prod_{\substack{w\in S\\w\neq v}}\left|\Tr\left(\Ind_{\B(\QQ_w)}^{\G(\QQ_w)}(s,\mu_w)(f_w)\right)\right|\prod_{p\notin S}(n_p+1)\rmd |s|\ll_{f_\infty, f_{q_i}} \bm{d}(n)\ll_{f_\infty, f_{q_i},\varepsilon} n^\varepsilon.
\end{align*} 
Hence $\eqref{eq:continuouspart2}\ll_{f_\infty, f_{q_i},\varepsilon} n^\varepsilon$. Since $J_{\mathrm{cont}}(f^n)$ is the sum of \eqref{eq:continuouspart1} and \eqref{eq:continuouspart2}, we obtain the desired result.
\end{proof}
Combining the trace formula \eqref{eq:traceformulagl2} with \autoref{prop:idestimate}, \autoref{prop:unipestimate}, \autoref{prop:hypestimate}, \autoref{prop:eisensteinestimate}, and \autoref{prop:continuouspart}, we establish \autoref{thm:cuspidalestimate}.

\section{Contribution of the elliptic part: Part 1}\label{sec:elliptic}
In the following two sections, we estimate the elliptic part and establish our main result. First, we recall the main theorem of \cite{cheng2025}:
\begin{theorem}\label{thm:maintheoremramify}
We have
\[
I_{\el}(f^n)=\sum_{\mu}\Tr(\mu(f^n))-\frac{1}{2}\sum_{\mu}\Tr((\xi_0\otimes\mu)(f^n))-\Sigma(\square)+\Sigma(0)+\Sigma(\xi\neq 0).
\]
See \cite[Theorem 8.10]{cheng2025} for the precise definitions of these terms.
\end{theorem}

Combining \autoref{thm:maintheoremramify} with \autoref{thm:cuspidalestimate} and \autoref{prop:eisensteinestimate}, we obtain
\begin{corollary}\label{cor:estimateellipticremain}
For any $\varepsilon>0$, we have
\[
I_{\mathrm{cusp}}(f^n)=-\Sigma(\square)+\Sigma(0)+\Sigma(\xi\neq 0)+O_{f_\infty,f_{q_i},\varepsilon}(n^\varepsilon).
\]
\end{corollary}

\begin{remark}
The main work of \cite{cheng2025} was isolating the traces of $1$-dimensional representations in the elliptic part of the trace formula. It is easy to see that 
\[
\sum_{\mu}\Tr(\mu(f^n))\asymp n^{\frac12}.
\]
Hence we need to cancel such a term coming from the spectral side in the geometric side of the trace formula.
\end{remark}

Now we set $\vartheta=1/2$ in the formula of \cite[Theorem 8.10]{cheng2025}. 
\subsection{Estimate of $\Sigma(0)$}
The term $\Sigma(0)$ decomposes into the following two terms:
\begin{equation}\label{eq:sigma01}
\begin{split}
&4\sqrt{n}\sum_{\pm}\sum_{\nu\in \ZZ^r}q^{\nu/2}\int_{x\in \RR}\int_{y\in \QQ_{S_\fin}}\theta_\infty^\pm(x)\theta_{q}(y,\pm nq^\nu)
  \left[\frac{1}{\dpii}\int_{(-1)}\widetilde{F}(s)\frac{\zeta(2s+2)}{\zeta(s+2)} \right.\\
  \times&\left.\prod_{i=1}^{r}\frac{1-q_i^{-2s-2}}{1-q_i^{-s-2}}\prod_{p\notin S}\frac{1-p^{-(n_p+1)(s+1)}}{1-p^{-s-1}}\legendresymbol{(4nq^\nu)^{-1/2}}{|x^2\mp 1|_\infty^{1/2}|y^2\mp 4nq^\nu|_q'^{1/2}}^{-s}\rmd s\right]\rmd x\rmd y
\end{split}
\end{equation}
and
\begin{equation}\label{eq:sigma02}  
\begin{split}
  &2\sum_{\pm}\sum_{\nu\in \ZZ^r}\int_{x\in \RR}\int_{y\in \QQ_{S_\fin}}\frac{\theta_\infty^\pm(x)\theta_{q}(y,\pm nq^\nu)}{\sqrt{|x^2\mp 1|_\infty|y^2\mp 4nq^\nu|_q'}}\left[\frac{\sqrt{\uppi}}{\dpii}\int_{\cC_r}\widetilde{F}(s) \frac{\Gamma(\frac{\iota+s}{2})}{\Gamma(\frac{\iota+1-s}{2})}\frac{\zeta(2s)}{\zeta(s+1)} \right.\\
  \times&\left. \prod_{i=1}^{r}\frac{1-\epsilon_i q_i^{s-1}}{1-\epsilon_i q_i^{-s}}\frac{1-q_i^{-2s}}{1-q_i^{-s-1}}\prod_{p\notin S}\frac{1-p^{-(n_p+1)s}}{1-p^{-s}} \legendresymbol{\uppi (4nq^\nu)^{-1/2}}{|x^2\mp 1|_\infty^{1/2}|y^2\mp 4nq^\nu|_q'^{1/2}}^{-s}\rmd s\right]
  \rmd x\rmd y.
\end{split}
\end{equation}
Our goal of this subsection is to bound \eqref{eq:sigma01} and \eqref{eq:sigma02}.

\begin{lemma}\label{lem:eisensteinestimate}
For any $\epsilon_i\in \{0,\pm 1\}$ we have 
\[
\int_{\omega_i(T_i)=\epsilon_i}\frac{\theta_{q_i}(T_i,\pm nq^{\nu})}{\sqrt{|T_i^2\mp 4n q^\nu|_{q_i}'}}\rmd T_i\ll_{f_{q_i}} 1.
\]
In particular,
\[
\int_{\QQ_{q_i}}\frac{\theta_{q_i}(T_i,\pm nq^{\nu})}{\sqrt{|T_i^2\mp 4n q^\nu|_{q_i}'}}\rmd T_i\ll_{f_{q_i}} 1.
\]
\end{lemma}
\begin{proof}
By the structure of $p$-adic numbers we know that $\ZZ_{q_i}^\times/(\ZZ_{q_i}^\times)^2$ is finite. Let $\{a_1,\dots,a_k\}$ be a complete set of representatives of $\ZZ_{q_i}^\times/(\ZZ_{q_i}^\times)^2$ and suppose that $na_j^{-1}\in (\ZZ_{q_i}^\times)^2$. Let $\alpha$ be a square root of $na_j^{-1}$ in $\ZZ_{q_i}^\times$. Then we have
\begin{align*}
  \int_{\omega_i(T_i)=\epsilon_i}\frac{\theta_{q_i}(T_i,\pm nq^{\nu})}{\sqrt{|T_i^2\mp 4n q^\nu|_{q_i}'}}\rmd T_i &\ll \int_{\QQ_{q_i}}\frac{|\theta_{q_i}(T_i,\pm nq^{\nu})|}{\sqrt{|T_i^2\mp 4n q^\nu|_{q_i}'}}\rmd T_i= \int_{\QQ_{q_i}}\frac{|\theta_{q_i}(\alpha T_i,\pm nq^{\nu})|}{\sqrt{|\alpha^2 T_i^2\mp 4n q^\nu|_{q_i}'}}|\alpha|_{q_i}\rmd T_i\\
  &=\int_{\QQ_{q_i}}\frac{|\theta_{q_i}(\alpha T_i,\pm nq^{\nu})|}{\sqrt{|\alpha^2(T_i^2\mp 4a_j q^\nu)|_{q_i}'}}\rmd T_i= \int_{\QQ_{q_i}}\frac{|\theta_{q_i}(\alpha T_i,\pm nq^{\nu})|}{\sqrt{|T_i^2\mp 4a_jq^\nu|_{q_i}'}}\rmd T_i.
\end{align*}
Since $\theta_{q_i}(\gamma)$ is compactly supported in $\G(\QQ_{q_i})$, there exists $M_i>0$ such that $\theta_{q_i}(T,N)=0$ if $|T|_{q_i}>q_i^{M_i}$. Thus for $|T_i|_{q_i}>q_i^{M_i}$ we have $|\alpha T_i|_{q_i}=|T_i|_{q_i}>q_i^{M_i}$ and hence $\theta_{q_i}(\alpha T_i,\pm nq^{\nu})=0$. Since $\theta_{q_i}$ is compactly supported, it is bounded. Hence
\[
\int_{\omega_i(T_i)=\epsilon_i}\frac{\theta_{q_i}(T_i,\pm nq^{\nu})}{\sqrt{|T_i^2\mp 4n q^\nu|_{q_i}'}}\rmd T_i \ll_{q_i,f_{q_i}} \int_{|T_i|\leq q_i^{M_i}}\frac{\rmd T_i}{\sqrt{|T_i^2\mp 4a_jq^\nu|_{q_i}'}}\ll_{f_{q_i}} 1.
\]
The last estimate holds since the integral
\[
\int_{|T_i|\leq q_i^{M_i}}\frac{\rmd T_i}{\sqrt{|T_i^2\mp 4a_jq^\nu|_{q_i}'}}
\]
converges. Indeed, $|T_i|\leq q_i^{M_i}$ is a bounded set. The integrand is smooth except for the singularities if $T_i^2\mp 4a_jq^\nu=0$. If $T_i=\xi$ is a singularity, the asymptotic behavior of the integrand near $\xi$ is $|T_i-\xi|_{q_i}^{-1/2}$, thus it is integrable near $\xi$.
\end{proof}

\begin{proposition}\label{prop:sigma01estimate}
For any $\varepsilon>0$, we have
\[
\eqref{eq:sigma01}\ll n^\varepsilon,
\] 
where the implied constant depends only on $f_\infty$, $f_{q_i}$ and $\varepsilon$.
\end{proposition}
\begin{proof}
We have
\begin{align*}
  \eqref{eq:sigma01} & \ll \sqrt{n}\sum_{\pm}\sum_{\nu\in \ZZ^r}q^{\nu/2}\int_{x\in \RR}\int_{y\in \QQ_{S_\fin}}|\theta_\infty^\pm(x)||\theta_{q}(y,\pm nq^\nu)|
  \left[\frac{1}{2\uppi}\int_{(-1)}|\widetilde{F}(s)|\left|\frac{\zeta(2s+2)}{\zeta(s+2)} \right| \right.\\
  &\times\left.\prod_{i=1}^{r}\left|\frac{1-q_i^{-2s-2}}{1-q_i^{-s-2}}\right|\prod_{p\notin S}\left|\sum_{u=0}^{n_p}p^{-u(s+1)}\right|\left|\legendresymbol{(4nq^\nu)^{-1/2}}{|x^2\mp 1|_\infty^{1/2}|y^2\mp 4nq^\nu|_q'^{1/2}}^{-s}\right|\rmd |s|\right]\rmd x\rmd y\\
  & \ll \sqrt{n}\sum_{\pm}\sum_{\nu\in \ZZ^r}q^{\nu/2}\int_{x\in \RR}\int_{y\in \QQ_{S_\fin}}|\theta_\infty^\pm(x)||\theta_{q}(y,\pm nq^\nu)|\legendresymbol{(4nq^\nu)^{-1/2}}{|x^2\mp 1|_\infty^{1/2}|y^2\mp 4nq^\nu|_q'^{1/2}}\rmd x\rmd y\\
  &\times
  \int_{(-1)}|\widetilde{F}(s)|\left|\frac{\zeta(2s+2)}{\zeta(s+2)} \right| \prod_{i=1}^{r}\left|\frac{1-q_i^{-2s-2}}{1-q_i^{-s-2}}\right|\prod_{p\notin S}(1+n_p)\rmd |s|.
\end{align*}
Let $s=\sigma+\rmi t$. $\widetilde{F}(s)$ has rapid decay in $t$ when $\sigma$ is fixed by Proposition 3.7 of \cite{cheng2025}. Moreover
\[
\frac{\zeta(2s+2)}{\zeta(s+2)}\qquad\text{and}\qquad \frac{1-q_i^{-2s-2}}{1-q_i^{-s-2}}
\]
both have polynomial growth in $t$ when $\sigma$ is fixed. Indeed, $\zeta(s)$ has polynomial growth in $t$ when $\sigma$ is fixed by \cite[Part II, \SSec 1.6]{tenenbaum2015analytic}. Since $|q_i^{-s}|$ is constant when $\sigma$ is fixed, $1-q_i^{-s}$ has polynomial growth in $t$ when $\sigma$ is fixed. Thus so are the above two functions. Therefore
\[
\int_{(-1)}|\widetilde{F}(s)|\left|\frac{\zeta(2s+2)}{\zeta(s+2)} \right| \prod_{i=1}^{r}\left|\frac{1-q_i^{-2s-2}}{1-q_i^{-s-2}}\right|\rmd |s|\ll_{f_\infty,f_{q_i}} 1.
\]

By \cite[Corollary 8.7]{cheng2025}, the sum over $\nu$ is finite (which is independent of $n$). Hence by \autoref{lem:eisensteinestimate},
\begin{align*}
&4\sqrt{n}\sum_{\pm}\sum_{\nu\in \ZZ^r}q^{\nu/2}\int_{x\in \RR}\int_{y\in \QQ_{S_\fin}}|\theta_\infty^\pm(x)||\theta_{q}(y,\pm nq^\nu)|\legendresymbol{(4nq^\nu)^{-1/2}}{|x^2\mp 1|_\infty^{1/2}|y^2\mp 4nq^\nu|_q'^{1/2}}\rmd x\rmd y\\
  \ll&\sum_{\pm}\sum_{\nu\in \ZZ^r}\int_{x\in \RR}\int_{y\in \QQ_{S_\fin}}\frac{|\theta_\infty^\pm(x)||\theta_{q}(y,\pm nq^\nu)|}{\sqrt{|x^2\mp 1|_\infty|y^2\mp 4nq^\nu|_q'}}\rmd x\rmd y\ll_{f_\infty,f_{q_i}} 1.
\end{align*}
Therefore
\[
\eqref{eq:sigma01}\ll_{f_\infty,f_{q_i}} \prod_{p\notin S}(1+n_p)=\bm{d}(n)\ll_{f_\infty,f_{q_i},\varepsilon} n^\varepsilon.\qedhere
\]
\end{proof}

\begin{proposition}\label{prop:sigma02estimate}
For any $\varepsilon>0$, we have
\[
\eqref{eq:sigma02}\ll n^\varepsilon,
\] 
where the implied constant depends only on $f_\infty$, $f_{q_i}$ and $\varepsilon$.
\end{proposition}
\begin{proof}
First we recall the following fact: If $h(s)$ is a meromorphic function on the vertical strip $\sigma_1<\sigma<\sigma_2$ such that $0\in \lopen\sigma_1,\sigma_2\ropen$. Suppose that  $h$ has a simple pole at $s=0$ and is holomorphic otherwise, and for any $\sigma\in  \lopen\sigma_1,\sigma_2\ropen$ and $\sigma\neq 0$, the integral $\int_{(\sigma)}h(s)\rmd s$ converges. Then
\[
\frac{1}{\dpii}\int_{\cC_v}h(s)\rmd s=-\frac{1}{2}\res_{s=0}h(s)+\pv\frac{1}{\dpii}\int_{(0)}h(s)\rmd s,
\]
where $\pv$ denotes the principal value. 
Also, recall that 
\[
\widetilde{F}(s)\frac{\Gamma(\frac{\iota+s}{2})}{\Gamma(\frac{\iota+1-s}{2})} \frac{\zeta(2s)}{\zeta(s+1)}\prod_{i=1}^{r}\frac{1-\epsilon_i q_i^{s-1}}{1-\epsilon_i q_i^{-s}} \frac{1-q_i^{-2s}}{1-q_i^{-s-1}}
\]
is regular at $s=0$ unless $\iota=0$ and $\epsilon_i=1$ for all $i$, and in this exceptional case, the function has a simple pole at $s=0$ and the residue at $s=0$ is
$-2^r/\sqrt{\uppi}$ \cite[Theorem 6.1]{cheng2025}. 
Hence we obtain
\begin{align*}
  \eqref{eq:sigma02} =&2\sum_{\pm}\sum_{\nu\in \ZZ^r}\int_{x\in \RR}\int_{y\in \QQ_{S_\fin}}\frac{\theta_\infty^\pm(x)\theta_{q}(y,\pm nq^\nu)}{\sqrt{|x^2\mp 1|_\infty|y^2\mp 4nq^\nu|_q'}}\left[\pv\frac{\sqrt{\uppi}}{\dpii}\int_{(0)}\widetilde{F}(s) \frac{\Gamma(\frac{\iota+s}{2})}{\Gamma(\frac{\iota+1-s}{2})}\frac{\zeta(2s)}{\zeta(s+1)} \right.\\
  \times&\left. \prod_{i=1}^{r}\frac{1-\epsilon_i q_i^{s-1}}{1-\epsilon_i q_i^{-s}}\frac{1-q_i^{-2s}}{1-q_i^{-s-1}}\prod_{p\notin S}\frac{1-p^{-(n_p+1)s}}{1-p^{-s}} \legendresymbol{\uppi (4nq^\nu)^{-1/2}}{|x^2\mp 1|_\infty^{1/2}|y^2\mp 4nq^\nu|_q'^{1/2}}^{-s}\rmd s\right]
  \rmd x\rmd y\\
  +&2^{r}\bm{d}(n)\sum_{\pm}\sum_{\nu\in \ZZ^r}\int_{X_0}\int_{Y_\mathbf{1}}\frac{\theta_\infty^\pm(x)\theta_q(y,\pm nq^\nu)}{\sqrt{|x^2\mp 1|_\infty|y^2\mp 4nq^\nu|_q'}}\rmd x\rmd y.
\end{align*}
By \cite[Corollary 8.7]{cheng2025}, the sum over $\nu$ is finite (which is independent of $n$). Hence by \autoref{lem:eisensteinestimate},
\[
\sum_{\pm}\sum_{\nu\in \ZZ^r}\int_{X_0}\int_{Y_\mathbf{1}}\frac{\theta_\infty^\pm(x)\theta_q(y,\pm nq^\nu)}{\sqrt{|x^2\mp 1|_\infty|y^2\mp 4nq^\nu|_q'}}\rmd x\rmd y\ll 1.
\]

Let $\delta$ be a sufficiently small fixed positive number. The first term is bounded by
\begin{align*}
  & \sum_{\pm}\sum_{\nu\in \ZZ^r}\int_{x\in \RR}\int_{y\in \QQ_{S_\fin}}\frac{|\theta_\infty^\pm(x)||\theta_{q}(y,\pm nq^\nu)|}{\sqrt{|x^2\mp 1|_\infty|y^2\mp 4nq^\nu|_q'}}\left[\frac{\sqrt{\uppi}}{\dpii}\left(\int_{-\rmi\infty}^{-\rmi\delta}+ \int_{\rmi\delta}^{\rmi\infty}\right)|\widetilde{F}(s)| \left|\frac{\Gamma(\frac{\iota+s}{2})}{\Gamma(\frac{\iota+1-s}{2})}\right|\left|\frac{\zeta(2s)}{\zeta(s+1)}\right| \right.\\
  \times&\left. \prod_{i=1}^{r}\left|\frac{1-\epsilon_i q_i^{s-1}}{1-\epsilon_i q_i^{-s}}\frac{1-q_i^{-2s}}{1-q_i^{-s-1}}\right|\prod_{p\notin S}\left|\sum_{u=0}^{n_p}p^{-us}\right|\left|\legendresymbol{\uppi (4nq^\nu)^{-1/2}}{|x^2\mp 1|_\infty^{1/2}|y^2\mp 4nq^\nu|_q'^{1/2}}^{-s}\right|\rmd |s|+1\right]
  \rmd x\rmd y\\
  \ll &\sum_{\pm}\sum_{\nu\in \ZZ^r}\int_{x\in \RR}\int_{y\in \QQ_{S_\fin}}\frac{|\theta_\infty^\pm(x)||\theta_{q}(y,\pm nq^\nu)|}{\sqrt{|x^2\mp 1|_\infty|y^2\mp 4nq^\nu|_q'}}\rmd x\rmd y\prod_{p\notin S}(1+n_p) \\
  \times&\left|\left(\int_{-\rmi\infty}^{-\rmi\delta}+ \int_{\rmi\delta}^{\rmi\infty}\right)\frac{\sqrt{\uppi}}{\dpii}\int_{(0)}\widetilde{F}(s) \frac{\Gamma(\frac{\iota+s}{2})}{\Gamma(\frac{\iota+1-s}{2})}\frac{\zeta(2s)}{\zeta(s+1)} \prod_{i=1}^{r}\frac{1-\epsilon_i q_i^{s-1}}{1-\epsilon_i q_i^{-s}}\frac{1-q_i^{-2s}}{1-q_i^{-s-1}}\rmd |s|\right|+1.
\end{align*}
Observe that $\widetilde{F}(s)$ has rapid decay in $t=\Im s$ when $\sigma=\Re s$ is fixed, and
\[
\frac{\zeta(2s)}{\zeta(s+1)},\qquad \frac{1-\epsilon_i q_i^{s-1}}{1-\epsilon_i q_i^{-s}},\qquad \frac{1-q_i^{-2s}}{1-q_i^{-s-1}}
\]
all have polynomial growth in $t$. Now we prove that $\Gamma(\frac{\iota+s}{2})/\Gamma(\frac{\iota+1-s}{2})$ has polynomial growth in $t$. Indeed, it is the direct consequence of the Stirling formula \cite[Chapter II.0]{tenenbaum2015analytic} that
\[
\frac{\Gamma(\frac{\iota+s}{2})}{\Gamma(\frac{\iota+1-s}{2})}\asymp |t|^{-1/2}
\]
for $|t|\gg 1$. Hence we obtain
\[
\left|\pv\frac{\sqrt{\uppi}}{\dpii}\int_{(0)}\widetilde{F}(s) \frac{\Gamma(\frac{\iota+s}{2})}{\Gamma(\frac{\iota+1-s}{2})}\frac{\zeta(2s)}{\zeta(s+1)} \prod_{i=1}^{r}\frac{1-\epsilon_i q_i^{s-1}}{1-\epsilon_i q_i^{-s}}\frac{1-q_i^{-2s}}{1-q_i^{-s-1}}\rmd s\right|\ll_{f_\infty,f_{q_i}} 1.
\]
Therefore, we conclude that 
\[
\eqref{eq:sigma02}\ll_{f_\infty,f_{q_i}} \prod_{p\notin S}(1+n_p)=\bm{d}(n)\ll_{f_\infty,f_{q_i},\varepsilon} n^\varepsilon.\qedhere
\]
\end{proof}

Combining \autoref{prop:sigma01estimate} and \autoref{prop:sigma02estimate} together, we obtain
\begin{proposition}\label{prop:estimatesigma0}
For any $\varepsilon>0$, we have
\[
\Sigma(0)\ll n^\varepsilon,
\]
where the implied constant depends only on $f_\infty$, $f_{q_i}$ and $\varepsilon$.
\end{proposition}

\subsection{Estimate of $\Sigma(\square)$}
We have
\begin{equation}\label{eq:defsigmasquare}
\begin{split}
\Sigma(\square)=&2\sum_{\pm}\sum_{\nu\in \ZZ^r}\sum_{\substack{T\in \ZZ^S\\ T^2\mp 4nq^\nu= \square}}\sum_{f^2\mid T^2\mp 4nq^\nu}  \sum_{k\in \ZZ_{(S)}^{>0}}\frac{1}{kf}\legendresymbol{(T^2\mp 4nq^\nu)/f^2}{k}\theta_\infty^\pm\legendresymbol{T}{2n^{1/2}q^{\nu/2}}\\
    \times &\prod_{i=1}^{r}\theta_{q_i}(T,\pm nq^\nu)\left[F\legendresymbol{kf^2}{|T^2\mp 4nq^\nu|_{\infty,q}'^{1/2}}+\frac{kf^2}{\sqrt{|T^2\mp 4nq^\nu|_{\infty,q}'}}V\legendresymbol{kf^2}{|T^2\mp 4nq^\nu|_{\infty,q}'^{1/2}}\right],
\end{split}
\end{equation}
which comes from the complement of the approximate functional equation of the non-square terms. It suggests that we can use the approximate functional equation in the reverse direction.

Note that when $T^2\mp 4nq^\nu=0$, then the contribution to $\Sigma(\square)$ is $0$. Hence we only need to consider the summation  over $T\in \ZZ^S$ such that $T^2\mp 4nq^\nu=\square$ and $T^2\mp 4nq^\nu\neq 0$.

The definition of the partial Zagier $L$-function naturally extends to $\delta=\sigma^2\in \ZZ^S$ that is a square and is nonzero. Write $\delta=\sigma^2$. Then we have
\begin{equation}\label{eq:zagierzeta}
\begin{split}
   L^{S}(s,\delta) &=\prod_{i=1}^{r}\left(1-q_i^{-s}\right)\sum_{f\mid {\sigma^{(q)}}}\frac{1}{f^{2s-1}}L\left(s,\legendresymbol{(\sigma^{(q)}/f)^2}{\cdot}\right)  \\
     &  =\prod_{i=1}^{r}\left(1-q_i^{-s}\right)\sum_{f\mid {\sigma^{(q)}}}\frac{1}{f^{2s-1}}\prod_{p\mid \frac{{\sigma^{(q)}}}{f}}\left(1-p^{-s}\right)\zeta(s)
     =\prod_{i=1}^{r}\left(1-q_i^{-s}\right)\sum_{f\mid {\sigma^{(q)}}}\frac{1}{f^{2s-1}}\sum_{k\mid \frac{{\sigma^{(q)}}}{f}}\frac{\bm{\mu}(k)}{k^s}\zeta(s),
\end{split}
\end{equation}
which is an entire function times $\zeta(s)$, where $\bm\mu(k)$ denotes the M\"obius function. Hence it is regular at $s\neq 1$ and has a simple pole at $s=1$ with residue
\[
\res_{s=1}L^{S}(s,\delta)=\prod_{i=1}^{r}(1-q_i^{-1})\sum_{f\mid {\sigma^{(q)}}}\frac{1}{f}\prod_{p\mid\frac{{\sigma^{(q)}}}{f}}(1-p^{-1})=\prod_{i=1}^{r} (1-q_i^{-1})\sum_{f\mid {\sigma^{(q)}}}\frac{1}{f}\sum_{k\mid \frac{{\sigma^{(q)}}}{f}}\frac{\bm{\mu}(k)}{k}.
\]

By \autoref{thm:funceqndelta}, we obtain
\begin{theorem}\label{thm:funceqndeltasquare}
We have
\begin{equation}\label{eq:funceqndeltasquare}
L^{S}(s,\delta)=\legendresymbol{|\delta|_{\infty,q}'}{\uppi}^{1/2-s} \prod_{i=1}^{r}\frac{1-q_i^{-s}}{1-q_i^{s-1}} \frac{\Gamma((1-s)/2)}{\Gamma(s/2)}L^{S}(1-s,\delta).
\end{equation}
\end{theorem}
Moreover, we establish the following approximate functional equation for $L^S(s,\delta)$.
\begin{theorem}\label{thm:afesquare}
For any $A>0$, we have
\begin{equation}\label{eq:approximatefesquare}
\begin{split}
   \res_{s=0}\widetilde{F}(s)L^{S}(1+s,\delta)A^s& =\sum_{f^2\mid \delta}\sum_{k\in \ZZ_{(S)}^{>0}}\frac{1}{kf}\legendresymbol{\delta/f^2}{k}F\legendresymbol{kf^2}{A}\\
     & +\sum_{f^2\mid \delta}\sum_{k\in \ZZ_{(S)}^{>0}}\frac{kf^2}{\sqrt{|\delta|_{\infty,q}'}}\frac{1}{kf}\legendresymbol{\delta/f^2}{k}V_{0,\mathbf{1}}\legendresymbol{kf^2A}{|\delta|_{\infty,q}'},
\end{split}
\end{equation}
where $\mathbf{1}=(1,\dots,1)\in \{0,\pm 1\}^r$.
\end{theorem}
\begin{proof}
The proof is just a slight modification of the original proof of \autoref{thm:afe}. The only difference is that we want to move the contour integral with integrand $\widetilde{F}(u)L^{S}(1+u,\delta)A^u$, which only has a double pole at $u=0$.
\end{proof}
By \eqref{eq:defsigmasquare} with $A=|T^2\mp 4nq^\nu|_{\infty,q}'^{1/2}$, we obtain
\begin{corollary}\label{cor:sigmasquarel}
We have
\begin{equation}\label{eq:sigmasquarel}
\begin{split}
   \Sigma(\square)= &2\sum_{\pm}\sum_{\nu\in \ZZ^r}\sum_{\substack{T\in \ZZ^S\\ T^2\mp 4nq^\nu= \square\\ T^2\mp 4nq^\nu\neq 0}}\theta_\infty^\pm\legendresymbol{T}{2n^{1/2}q^{\nu/2}}
    \prod_{i=1}^{r}\theta_{q_i}(T,\pm nq^\nu) \\
     & \times \res_{s=0}\widetilde{F}(s)L^{S}(1+s,T^2\mp 4nq^\nu)|T^2\mp 4nq^\nu|_{\infty,q}'^{s/2}.
\end{split}
\end{equation}
\end{corollary}

Now we give the estimate of $\Sigma(\square)$. First we prove the following lemma for arithmetic functions.
\begin{lemma}\label{lem:arithmeticfunction}
Suppose that $n$ is a positive integer. Let $\bm{\mu}(n)$\index{mu@$\bm{\mu}(n)$} be the M\"obius function and $\bm{\Lambda}(n)$\index{lambda@$\bm{\Lambda}(n)$} be the von Mangoldt function.
\begin{enumerate}[itemsep=0pt,parsep=0pt,topsep=0pt, leftmargin=0pt,labelsep=2.5pt,itemindent=15pt,label=\upshape{(\arabic*)}]
\item We have
\[
\sum_{f\mid n}\frac{1}{f}\sum_{k\mid \frac{n}{f}}\frac{\bm{\mu}(k)}{k}=1.
\]
\item We have
\[
\sum_{f\mid n}\frac{\log f}{f}\sum_{k\mid \frac{n}{f}}\frac{\bm{\mu}(k)}{k}=\sum_{d\mid n}\frac{\bm{\Lambda}(d)}{d}
\qquad\text{and}\qquad
\sum_{f\mid n}\frac{1}{f}\sum_{k\mid \frac{n}{f}}\frac{\bm{\mu}(k)\log k}{k}=-\sum_{d\mid n}\frac{\bm{\Lambda}(d)}{d}.
\]
\end{enumerate}
\end{lemma}
\begin{proof}
(1) This is because
\[
\sum_{f\mid n}\frac{1}{f}\sum_{k\mid \frac{n}{f}}\frac{\bm{\mu}(k)}{k}=\sum_{kf\mid n}\frac{\bm{\mu}(k)}{kf}=\sum_{d\mid n}\frac{1}{d}\sum_{k\mid d}\bm{\mu}(k)=1.
\]

(2) We begin by proving the first equality. We have
\[
\sum_{k\mid \frac{n}{f}}\frac{\bm{\mu}(k)}{k}=\prod_{p\mid \frac{n}{f}}(1-p^{-1})=\frac{\bm{\phi}(n/f)}{n/f},
\] 
where $\bm{\phi}$\index{phi@$\bm{\phi}(n)$} denotes the Euler totient function. Hence
\[
\sum_{f\mid n}\frac{\log f}{f}\sum_{k\mid \frac{n}{f}}\frac{\bm{\mu}(k)}{k}=\frac{1}{n}\sum_{f\mid n}\log f\bm{\phi}\legendresymbol{n}{f}=\frac{1}{n}\sum_{f\mid n}\sum_{d\mid f}\bm{\Lambda}(d)\bm{\phi}\legendresymbol{n}{f}=\frac{1}{n}\sum_{d\mid n}\bm{\Lambda}(d)\sum_{f\mid \frac{n}{d}}\bm{\phi}(f)=\sum_{d\mid n}\frac{\bm{\Lambda}(d)}{d}.
\]

The second equality follows from 
\[
\sum_{f\mid n}\frac{\log f}{f}\sum_{k\mid \frac{n}{f}}\frac{\bm{\mu}(k)}{k}+
\sum_{f\mid n}\frac{1}{f}\sum_{k\mid \frac{n}{f}}\frac{\bm{\mu}(k)\log k}{k}= \sum_{f\mid n}\sum_{k\mid \frac{n}{f}}\frac{\bm{\mu}(k)\log (kf)}{kf}=\sum_{d\mid n}\frac{\log d}{d}\sum_{k\mid d}\bm{\mu}(k)=0.\qedhere
\]
\end{proof}

\begin{proposition}\label{prop:estimatesigmasquare}
For any $\varepsilon>0$, we have
\[
\Sigma(\square)\ll n^\varepsilon,
\]
where the implied constant depends only on $f_\infty$, $f_{q_i}$ and $\varepsilon$.
\end{proposition}
\begin{proof}
Since the sum over $\nu$ is finite (independent of $n$) by the proof of \cite[Theorem 2.14]{cheng2025}, it suffices to show that for fixed $\nu\in \ZZ^r$ and sign $\pm$, we have
\[
\sum_{\substack{T\in \ZZ^S\\ T^2\mp 4nq^\nu= \square\\ T^2\mp 4nq^\nu\neq 0}}\theta_\infty^\pm\legendresymbol{T}{2n^{1/2}q^{\nu/2}}\\
    \prod_{i=1}^{r}\theta_{q_i}(T,\pm nq^\nu)\res_{s=0}\widetilde{F}(s)L^{S}(1+s,T^2\mp 4nq^\nu)|T^2\mp 4nq^\nu|_{\infty,q}'^{s/2}\ll_{f_\infty,f_{q_i},\varepsilon} n^\varepsilon.
\]
Assume $T^2\mp 4nq^\nu=\sigma^2$ for some $\sigma\in \ZZ^S$. Then $(T+\sigma)(T-\sigma)=\pm 4nq^\nu$. Hence there exist $\alpha,\beta\in \ZZ^r$ and $n_1,n_2\in \ZZ_{(S)}^{>0}$ with $\alpha+\beta=\nu$ and $n=n_1n_2$, such that $T+\sigma=\pm 2n_1q^{\alpha}$ and $T-\sigma=\pm 2n_2q^{\beta}$.
\[
T=\pm(n_1q^\alpha \pm n_2q^\beta)\qquad\text{and}\qquad \sigma=\pm(n_1q^\alpha \mp n_2q^\beta).
\]

Now we fix such $n_1$ and $n_2$. Note that the number of such pair $(n_1,n_2)$ is $\bm{d}(n)$, which is $\ll_\varepsilon n^\varepsilon$ for any $\varepsilon>0$.

For any $i\in \{1,\dots,r\}$, there exists $M_i>0$ such that $\theta_{q_i}(T,\pm nq^\nu)=0$ if $v_{q_i}(T)<-M_i$. Since 
\[
v_{q_i}(T)=v_{q_i}(n_1q^\alpha \pm n_2q^\beta)\geq \min\{v_{q_i}(n_1q^\alpha) ,v_{q_i}(n_2q^\beta)\}=\min\{\alpha_i,\beta_i\},
\]
we must have $\alpha_i,\beta_i\geq -M_i$ if $T$ contributes to this sum. Since $\alpha_i+\beta_i=\nu_i$ is fixed, $\alpha_i$ and $\beta_i$ have finitely many choices. Since this holds for all $i$, $\alpha$ and $\beta$ have finitely many choices. 

By the above, we know that it suffices to show that for any fixed $\alpha$ and $\beta$ with $\alpha+\beta=\nu$ and $n_1,n_2\in \ZZ_{(S)}^{>0}$ with $n=n_1n_2$, and any $\varepsilon>0$, we have
 \[
\theta_\infty^\pm\legendresymbol{T}{2n^{1/2}q^{\nu/2}}\\
    \prod_{i=1}^{r}\theta_{q_i}(T,\pm nq^\nu)\res_{s=0}\widetilde{F}(s)L^{S}(1+s,T^2\mp 4nq^\nu)|T^2\mp 4nq^\nu|_{\infty,q}'^{s/2}\ll_{f_\infty,f_{q_i},\varepsilon} n^\varepsilon.
\]

Without loss of generality, we may assume $T=n_1q^\alpha \pm n_2q^\beta$ so that $\sigma^2=T^2\mp 4nq^\nu=(n_1q^\alpha \mp n_2q^\beta)^2$. Clearly we have
\[
\theta_\infty^\pm\legendresymbol{T}{2n^{1/2}q^{\nu/2}}\\
    \prod_{i=1}^{r}\theta_{q_i}(T,\pm nq^\nu)\ll_{f_\infty,f_{q_i}} 1.
\]
Hence it suffices to show that
\[
\res_{s=0}\widetilde{F}(s)L^{S}(1+s,T^2\mp 4nq^\nu)|T^2\mp 4nq^\nu|_{\infty,q}'^{s/2}\ll_{\varepsilon} n^\varepsilon.
\]

Since $\widetilde{F}(s)$ is an odd function, we have $\fp_{s=0}\widetilde{F}(s)=0$. Therefore
\begin{align*}
   & \res_{s=0}\widetilde{F}(s)L^{S}(1+s,T^2\mp 4nq^\nu)|T^2\mp 4nq^\nu|_{\infty,q}'^{s/2}  \\
  = & \fp_{s=0}\widetilde{F}(s)\res_{s=1}L^{S}(s,T^2\mp 4nq^\nu)+\res_{s=0}\widetilde{F}(s)\fp_{s=1}L^{S}(s,T^2\mp 4nq^\nu)\\
  +& \res_{s=0}\widetilde{F}(s)\res_{s=1}L^{S}(s,T^2\mp 4nq^\nu)(|T^2\mp 4nq^\nu|_{\infty,q}'^{s/2})'|_{s=0} \\
  \ll &\mathopen{|}\res_{s=1}L^{S}(s,T^2\mp 4nq^\nu)\mathclose{|}\log|T^2\mp 4nq^\nu|_{\infty,q}'+\mathopen{|}\fp_{s=1}L^{S}(s,T^2\mp 4nq^\nu)\mathclose{|}.
\end{align*}

For $T^2\mp 4nq^\nu=\delta=\sigma^2$, by \eqref{eq:zagierzeta} we have 
\[
L^S(s,\delta)=\prod_{i=1}^{r}\left(1-q_i^{-s}\right)\zeta(s)r(s),
\]
where
\[
r(s)=\sum_{f\mid {\sigma^{(q)}}}\frac{1}{f^{2s-1}}\prod_{p\mid \frac{{\sigma^{(q)}}}{f}}\left(1-p^{-s}\right)=\sum_{f\mid {\sigma^{(q)}}}\frac{1}{f^{2s-1}}\sum_{k\mid \frac{{\sigma^{(q)}}}{f}}\frac{\bm{\mu}(k)}{k^s}.
\]
Hence
\[
\res_{s=1}L^S(s,\delta)\ll_q \res_{s=1}\zeta(s)r(1)=\sum_{f\mid {\sigma^{(q)}}}\frac{1}{f}\sum_{k\mid \frac{{\sigma^{(q)}}}{f}}\frac{\bm{\mu}(k)}{k}=1
\]
and 
\begin{align*}
   \fp_{s=1}L^S(s,\delta) & \ll_q \mathopen{|}\fp_{s=1}\zeta(s)r(1)\mathclose{|}+|r'(1)| \\
     & \ll\left|\sum_{f\mid {\sigma^{(q)}}}\frac{1}{f}\sum_{k\mid \frac{{\sigma^{(q)}}}{f}}\frac{\bm{\mu}(k)}{k}\right|+\left| \sum_{f\mid {\sigma^{(q)}}}\frac{2\log f}{f}\sum_{k\mid \frac{{\sigma^{(q)}}}{f}}\frac{\bm{\mu}(k)}{k}+\sum_{f\mid {\sigma^{(q)}}}\frac1f\sum_{k\mid \frac{{\sigma^{(q)}}}{f}}\frac{\bm{\mu}(k)\log k}{k}\right|\\
     &= 1+\sum_{d\mid \sigma^{(q)}}\frac{\bm{\Lambda}(d)}{d}\ll 1+\sum_{d\mid \sigma^{(q)}}\bm{\Lambda}(d)\ll \log|{\sigma^{(q)}}|
\end{align*}
by \autoref{lem:arithmeticfunction}.

Since $|T^2\mp 4nq^\nu|_{\infty,q}'=|{\sigma^{(q)}}|^2$, we obtain
\[
   \res_{s=0}\widetilde{F}(s)L^{S}(1+s,T^2\mp 4nq^\nu)|T^2\mp 4nq^\nu|_{\infty,q}'^{s/2}  \\
  \ll_q 2\log|{\sigma^{(q)}}|+\log|{\sigma^{(q)}}|\ll \log|{\sigma^{(q)}}|.
\]
Finally we have ${\sigma^{(q)}}=\sigma/q^\gamma$, where $q^\gamma=\sigma_{(q)}$. Thus
\[
\gamma_i=v_{q_i}(\sigma)=v_{q_i}(n_1q^\alpha \mp n_2q^\beta)\geq \min\{\alpha_i,\beta_i\}.
\]
Hence 
\[
|{\sigma^{(q)}}|=|\sigma|/q^\gamma\leq |\sigma|/q^{\min\{\alpha,\beta\}}\ll |\sigma|=|n_1q^\alpha \mp n_2q^\beta|\ll_q n.
\]
Therefore
\[
\res_{s=0}\widetilde{F}(s)L^{S}(1+s,T^2\mp 4nq^\nu)|T^2\mp 4nq^\nu|_{\infty,q}'^{s/2}  \\
  \ll  \log|{\sigma^{(q)}}|\ll \log n\ll_{\varepsilon} n^\varepsilon.\qedhere
\]
\end{proof}

Combining \autoref{cor:estimateellipticremain}, \autoref{prop:estimatesigma0} and \autoref{prop:estimatesigmasquare} yields
\begin{corollary}\label{cor:estimateellipticremain2}
For any $\varepsilon>0$, we have
\[
I_{\mathrm{cusp}}(f^n)=\Sigma(\xi\neq 0)+O_{f_\infty,f_{q_i},\varepsilon}(n^\varepsilon).
\]
\end{corollary}

\section{Contribution of the elliptic part: Part 2}\label{sec:elliptic2}
The main goal of this section is to prove that for any $\varepsilon>0$, we have
\[
\Sigma(\xi\neq 0)\ll_{f_{\infty},f_{q_i},\varepsilon} n^{\frac14+\varepsilon}.
\]

Recall that
\begin{align*}
  \Sigma(\xi\neq 0)=&4\sqrt{n}\sum_{\pm}\sum_{\nu\in \ZZ^r}q^{\nu/2}\sum_{k,f\in \ZZ_{(S)}^{>0}}\frac{1}{k^2f^3}\sum_{\xi\in \ZZ^S-\{0\}}\Kl_{k,f}^S(\xi,\pm nq^\nu)\\
   \times&\int_{x\in\RR}\int_{y\in\QQ_{S_\fin}}\theta_\infty^\pm(x)\theta_{q}(y,\pm nq^\nu) \left[F\legendresymbol{kf^2(4nq^\nu)^{-1/2}}{|x^2\mp 1|_\infty^{1/2}|y^2\mp 4nq^\nu|_q'^{1/2}}+\frac{kf^2n^{-1/2}q^{-\nu/2}}{2\sqrt{|x^2\mp 1|_\infty|y^2\mp 4nq^\nu|_q'}}\right.\\
     \times&\left.V\legendresymbol{kf^2(4nq^\nu)^{-1/2}}{|x^2\mp 1|_\infty^{1/2}|y^2\mp 4nq^\nu|_q'^{1/2}}\right]\rme\legendresymbol{-2x\xi n^{1/2}q^{\nu/2}}{kf^2}\rme_{q}\legendresymbol{-y\xi}{kf^2}\rmd x\rmd y,
\end{align*}
where \index{vx@$V(x)$}
\[
V\legendresymbol{kf^2(4nq^\nu)^{\vartheta-1}}{|x^2\mp 1|_\infty^{1-\vartheta}|y^2\mp 4nq^\nu|_q'^{1-\vartheta}}=V_{\iota,\epsilon}\legendresymbol{kf^2(4nq^\nu)^{\vartheta-1}}{|x^2\mp 1|_\infty^{1-\vartheta}|y^2\mp 4nq^\nu|_q'^{1-\vartheta}}
\]
with $\iota=\omega(x)$ and $\epsilon_i=\omega_i(y_i)$\index{omegayi@$\omega_i(y_i)$}.

Since the sum over $\nu$ is finite, it suffices to show that for fixed sign $\pm$ and $\nu\in \ZZ^r$, the summand is bounded by $n^{1/4+\varepsilon}$. In other words, it suffices to prove that the following is bounded by $n^{1/4+\varepsilon}$:
\begin{equation}\label{eq:xinot0}
\begin{split}
  &\sqrt{n}\sum_{k,f\in \ZZ_{(S)}^{>0}}\frac{1}{k^2f^3}\sum_{\xi\in \ZZ^S-\{0\}}\Kl_{k,f}^S(\xi,\pm nq^\nu)
   \int_{x\in\RR}\int_{y\in\QQ_{S_\fin}}\theta_\infty^\pm(x)\theta_{q}(y,\pm nq^\nu) \\ \times&\left[F\legendresymbol{kf^2(4nq^\nu)^{-1/2}}{|x^2\mp 1|_\infty^{1/2}|y^2\mp 4nq^\nu|_q'^{1/2}}+\frac{kf^2n^{-1/2}q^{-\nu/2}}{2\sqrt{|x^2\mp 1|_\infty|y^2\mp 4nq^\nu|_q'}}V\legendresymbol{kf^2(4nq^\nu)^{-1/2}}{|x^2\mp 1|_\infty^{1/2}|y^2\mp 4nq^\nu|_q'^{1/2}}\right]\\
   \times&\rme\legendresymbol{-2x\xi n^{1/2}q^{\nu/2}}{kf^2}\rme_{q}\legendresymbol{-y\xi}{kf^2}\rmd x\rmd y.
\end{split}
\end{equation}

Establishing the bound of $\Sigma(\xi\neq 0)$ involves technical complexities. To enhance readability, we outline the proof strategy here, with full details developed in this section and the appendices.
The proof proceeds in three steps:

First, we analyze the Fourier transform in \eqref{eq:xinot0}, which exhibits \emph{singularities} (using results from \autoref{subsec:singularities}). This requires deriving \emph{uniform} bounds in parameters $\xi,k,f,n$, with the nonarchimedean and the global theory developed in \autoref{sec:fourier} and with archimedean case following \cite[Appendix A]{altug2017}. The main result is \autoref{cor:xi0firstestimate}.
 
Next, we establish estimates for generalized Kloosterman sums in \autoref{sec:kloosterman} (extending methods from \cite[Appendix B]{altug2017}), with key results in \autoref{lem:kloosterman}.

Finally, we directly bound the simplified expression to obtain the desired bound by using the estimate of certain "heights" on the ring of $S$-integers (\autoref{lem:estimatexilargesum} and \autoref{lem:estimatexismallsum}). The main results are \autoref{prop:xi0kloosterman1} and \autoref{prop:xi0kloosterman2}.

Readers familiar with \cite[Section 4.2 and Appendices A and B]{altug2017} will see numerous similarities in the proof below.

\subsection{Estimate of the Fourier transform}\label{subsec:fourierestimate}
\eqref{eq:xinot0} can be separated into the following two terms
\begin{equation}\label{eq:xi0firstterm}
\begin{split}
  &\sum_{k,f\in \ZZ_{(S)}^{>0}}\frac{1}{k^2f^3}\sum_{\xi\in \ZZ^S-\{0\}}\Kl_{k,f}^S(\xi,\pm nq^\nu)
   \int_{x\in\RR}\int_{y\in\QQ_{S_\fin}}\theta_\infty^\pm(x)\theta_{q}(y,\pm nq^\nu) \\
   \times&F\legendresymbol{kf^2(4nq^\nu)^{-1/2}}{|x^2\mp 1|_\infty^{1/2}|y^2\mp 4nq^\nu|_q'^{1/2}}\rme\legendresymbol{-2x\xi n^{1/2}q^{\nu/2}}{kf^2}\rme_{q}\legendresymbol{-y\xi}{kf^2}\rmd x\rmd y
\end{split}
\end{equation}
and
\begin{equation}\label{eq:xi0secondterm}
\begin{split}
  &\sum_{k,f\in \ZZ_{(S)}^{>0}}\frac{1}{k^2f^3}\sum_{\xi\in \ZZ^S-\{0\}}\Kl_{k,f}^S(\xi,\pm nq^\nu)
   \int_{x\in\RR}\int_{y\in\QQ_{S_\fin}}\theta_\infty^\pm(x)\theta_{q}(y,\pm nq^\nu)\frac{kf^2n^{-1/2}q^{-\nu/2}}{2\sqrt{|x^2\mp 1|_\infty|y^2\mp 4nq^\nu|_q'}} \\ \times&V\legendresymbol{kf^2(4nq^\nu)^{-1/2}}{|x^2\mp 1|_\infty^{1/2}|y^2\mp 4nq^\nu|_q'^{1/2}}\rme\legendresymbol{-2x\xi n^{1/2}q^{\nu/2}}{kf^2}\rme_{q}\legendresymbol{-y\xi}{kf^2}\rmd x\rmd y.
\end{split}
\end{equation}

For $\xi\in \ZZ^S$ and $a\in \RR$, we define \index{astarxi@$\llbracket a \star \xi\rrbracket$}
\[
\llbracket a \star \xi\rrbracket=(1+|a\xi|_\infty)\prod_{i=1}^{r}(1+|\xi|_{q_i}),
\]
which is $\asymp|a\xi|$ if $\xi\in \ZZ-\{0\}$ and $a\gg 1$.
Note that if $a\in \ZZ_{(S)}$, then $\llbracket a \star \xi\rrbracket=\llbracket 1 \star a\xi\rrbracket$.

For $k,f\in \ZZ_{(S)}^{>0}$, we consider the following two terms separately:
\begin{equation}\label{eq:xi0firsttermsingle}
\int_{x\in\RR}\int_{y\in\QQ_{S_\fin}}\!\theta_\infty^\pm(x)\theta_{q}(y,\pm nq^\nu) F\legendresymbol{kf^2(4nq^\nu)^{-1/2}}{|x^2\mp 1|_\infty^{1/2}|y^2\mp 4nq^\nu|_q'^{1/2}} \rme\legendresymbol{-2x\xi n^{1/2}q^{\nu/2}}{kf^2}\rme_{q}\!\legendresymbol{-y\xi}{kf^2}\!\rmd x\rmd y.
\end{equation}
and
\begin{equation}\label{eq:xi0secondtermid}
\begin{split}
   & \int_{x\in\RR}\int_{y\in\QQ_{S_\fin}}\theta_\infty^\pm(x)\theta_{q}(y,\pm nq^\nu)\frac{kf^2n^{-1/2}q^{-\nu/2}}{2\sqrt{|x^2\mp 1|_\infty|y^2\mp 4nq^\nu|_q'}} \\ \times&V\legendresymbol{kf^2(4nq^\nu)^{-1/2}}{|x^2\mp 1|_\infty^{1/2}|y^2\mp 4nq^\nu|_q'^{1/2}}\rme\legendresymbol{-2x\xi n^{1/2}q^{\nu/2}}{kf^2}\rme_{q}\legendresymbol{-y\xi}{kf^2}\rmd x\rmd y.
\end{split}
\end{equation}

\begin{proposition}\label{prop:xi0firstterm1}
Suppose that 
\[
\left\llbracket\frac{\sqrt{n}}{kf^2}\star \xi\right\rrbracket\frac {(kf^2)^2}{n} \gg 1.
\]
Then for any $M>0$, we have
\[
\eqref{eq:xi0firsttermsingle}
\ll \left(\left\llbracket\frac{\sqrt{n}}{kf^2}\star \xi\right\rrbracket\frac {(kf^2)^2}{n}\right)^{-M}\left\llbracket\frac{\sqrt{n}}{kf^2}\star \xi\right\rrbracket^{-2}\frac{n}{(kf^2)^2},
\]
where the implied constant depends only on $M,f_\infty,f_{q_i},\pm$ and $\nu$.
\end{proposition}

\begin{proof}
Fix $\iota\in\{0,1\}$ and $\epsilon_i\in \{0,\pm 1\}$.
$\widetilde{F}(s)$ and $\theta_{\infty,\sigma}$, $\theta_{q_i,\tau_i}$ satisfy the properties in \autoref{thm:mainfourierestimate} by \autoref{prop:propertyf} and the results in \autoref{subsec:singularities}. Hence for any $M_1$ and $\sigma,\tau_i\in \{0,1\}$, we have
\begin{align*}
&\int_{X_\iota}\int_{Y_\epsilon}\theta_{\infty,\sigma}^\pm(x)\theta_{q,\tau}(y,\pm nq^\nu) F\legendresymbol{kf^2(4nq^\nu)^{-1/2}}{|x^2\mp 1|_\infty^{1/2}|y^2\mp 4nq^\nu|_q'^{1/2}} \rme\legendresymbol{-2x\xi n^{1/2}q^{\nu/2}}{kf^2}\rme_{q}\legendresymbol{-y\xi}{kf^2}\rmd x\rmd y\\
\ll &\left[\frac{(kf^2)^2}{4nq^\nu}\left(1+\left|\frac{2\xi\sqrt{nq^\nu}}{kf^2}\right|\right) \prod_{i=1}^{r}\left(1+\left|\frac{\xi}{kf^2}\right|_{q_i}\right)\right]^{-M_1} \left(1+\left|\frac{2\xi\sqrt{nq^\nu}}{kf^2}\right|\right)^{-1-\frac{\sigma}{2}} \prod_{i=1}^{r}\left(1+\left|\frac{\xi}{kf^2}\right|_{q_i}\right)^{-1-\frac{\tau_i}{2}}\\
\ll & \left(\frac{(kf^2)^2}{n}\left\llbracket\frac{\sqrt{n}}{kf^2}\star \xi\right\rrbracket\right) ^{-M_1}\left\llbracket\frac{\sqrt{n}}{kf^2}\star \xi\right\rrbracket^{-1}.
\end{align*}
Letting $M_1=1+M$, the estimate is reduced to
\begin{align*}
&\int_{X_\iota}\int_{Y_\epsilon}\theta_{\infty,\sigma}^\pm(x)\theta_{q,\tau}(y,\pm nq^\nu) F\legendresymbol{kf^2(4nq^\nu)^{-1/2}}{|x^2\mp 1|_\infty^{1/2}|y^2\mp 4nq^\nu|_q'^{1/2}} \rme\legendresymbol{-2x\xi n^{1/2}q^{\nu/2}}{kf^2}\rme_{q}\legendresymbol{-y\xi}{kf^2}\rmd x\rmd y\\
\ll &\left(\left\llbracket\frac{\sqrt{n}}{kf^2}\star \xi\right\rrbracket\frac {(kf^2)^2}{n}\right)^{-M}\left\llbracket\frac{\sqrt{n}}{kf^2}\star \xi\right\rrbracket^{-2}\frac{n}{(kf^2)^2}.
\end{align*}
Now we sum over all $\iota$, $\epsilon_i$, $\sigma$, $\tau_i$, and get the desired conclusion.
\end{proof}

\begin{proposition}\label{prop:xi0firstterm2}
Suppose that 
\[
\left\llbracket\frac{\sqrt{n}}{kf^2}\star \xi\right\rrbracket\frac {(kf^2)^2}{n} \ll 1.
\]
Then for any $\varepsilon>0$ we have
\[
\eqref{eq:xi0firsttermsingle}\ll \frac{kf^2}{\sqrt{n}}\left\llbracket\frac{\sqrt{n}}{kf^2}\star \xi\right\rrbracket^{-\frac12}\left(\frac{(kf^2)^2}{n}\left\llbracket\frac{\sqrt{n}}{kf^2}\star \xi\right\rrbracket\right)^{-\varepsilon},
\]
where the implied constant depends only on $\varepsilon$, $f_\infty$, $f_{q_i}$, $\pm$ and $\nu$.
\end{proposition}
\begin{proof}
Using the notations as above, we want to estimate 
\[
\int_{x\in\RR}\int_{Y_\epsilon}\theta_{\infty,\sigma}^\pm(x)\theta_{q,\tau}(y,\pm nq^\nu) F\legendresymbol{kf^2(4nq^\nu)^{-1/2}}{|x^2\mp 1|_\infty^{1/2}|y^2\mp 4nq^\nu|_q'^{1/2}} \rme\legendresymbol{-2x\xi n^{1/2}q^{\nu/2}}{kf^2}\rme_{q}\legendresymbol{-y\xi}{kf^2}\rmd x\rmd y.
\]

\underline{\emph{Case 1:}}\ \ $\sigma=1$.
$\widetilde{F}(s)$ and $\theta_{\infty,\sigma}$, $\theta_{q_i,\tau_i}$ satisfy the properties in \autoref{thm:mainfourierestimate2} by \autoref{prop:propertyf} and the results in \autoref{subsec:singularities}. Hence the integral above can be bounded by
\begin{align*}
\left(1+\left|\frac{2\xi\sqrt{nq^\nu}}{kf^2}\right|\right)^{-1-\frac{1}{2}} \prod_{i=1}^{r}\left(1+\left|\frac{\xi}{kf^2}\right|_{q_i}\right)^{-1-\frac{\tau_i}{2}}&\ll \left\llbracket\frac{\sqrt{n}}{kf^2}\star \xi\right\rrbracket^{-1}
\left(1+\left|\frac{2\xi\sqrt{nq^\nu}}{kf^2}\right|\right)^{-1}\\
&\ll \left\llbracket\frac{\sqrt{n}}{kf^2}\star \xi\right\rrbracket^{-\frac12}\left(\frac{\sqrt{n}}{kf^2}\right)^{-1}.
\end{align*}

\underline{\emph{Case 2:}}\ \ $\sigma=0$.
$\widetilde{F}(s)$ and $\theta_{\infty,\sigma}$, $\theta_{q_i,\tau_i}$ satisfy the properties in the last statement of \autoref{thm:mainfourierestimate2} by \autoref{prop:propertyf} and the results in \autoref{subsec:singularities}. Hence the integral above can be bounded by
\begin{align*}
&\left[\left|\frac{kf^2}{\sqrt{4nq^\nu}}\right|^{2-2\varepsilon} \prod_{i=1}^{r}\left(1+\left|\frac{\xi}{kf^2}\right|_{q_i}\right) ^{1-\varepsilon} +\left(1+\left|\frac{2\xi\sqrt{nq^\nu}}{kf^2}\right|\right)^{-2}\right] \prod_{i=1}^{r}\left(1+\left|\frac{\xi}{kf^2}\right|_{q_i}\right)^{-1-\frac{\tau_i}{2}}\\
\ll& \legendresymbol{\sqrt{n}}{kf^2}^{-2+2\varepsilon} +\left\llbracket\frac{\sqrt{n}}{kf^2}\star \xi\right\rrbracket^{-1}\left(\frac{\sqrt{n}}{kf^2}\right)^{-1}
\ll \left(\frac{\sqrt{n}}{kf^2}\right)^{-1}\left\llbracket\frac{\sqrt{n}}{kf^2}\star \xi\right\rrbracket^{-\frac12}\left(\frac{(kf^2)^2}{n}\left\llbracket\frac{\sqrt{n}}{kf^2}\star \xi\right\rrbracket\right)^{-\varepsilon},
\end{align*}
since 
\[
\frac{kf^2}{\sqrt{n}} \ll \left\llbracket\frac{\sqrt{n}}{kf^2}\star \xi\right\rrbracket^{-\frac12}
\]
by assumption.
The conclusion now follows by summing over $\sigma$ and $\tau_i$.
\end{proof}

By \autoref{prop:xi0firstterm1} and \autoref{prop:xi0firstterm2} we obtain
\begin{corollary}\label{cor:xi0firsttermfinal}
For any $M,\varepsilon>0$, we have
\begin{align*}
\eqref{eq:xi0firstterm}\ll &\sum_{\substack{k,f\in \ZZ_{(S)}^{>0},\ \xi\in \ZZ^S-\{0\} \\ \left\llbracket\frac{\sqrt{n}}{kf^2}\star \xi\right\rrbracket\frac {(kf^2)^2}{n} \gg 1}}\frac{1}{k^2f^3}|\Kl_{k,f}^S(\xi,\pm nq^\nu)|\left(\left\llbracket\frac{\sqrt{n}}{kf^2}\star \xi\right\rrbracket\frac {(kf^2)^2}{n}\right)^{-M}\left\llbracket\frac{\sqrt{n}}{kf^2}\star \xi\right\rrbracket^{-2}\frac{n}{(kf^2)^2}\\
+&\sum_{\substack{k,f\in \ZZ_{(S)}^{>0},\ \xi\in \ZZ^S-\{0\} \\ \left\llbracket\frac{\sqrt{n}}{kf^2}\star \xi\right\rrbracket\frac {(kf^2)^2}{n} \ll 1}}\frac{1}{k^2f^3}|\Kl_{k,f}^S(\xi,\pm nq^\nu)|\frac{kf^2}{\sqrt{n}}\left\llbracket\frac{\sqrt{n}}{kf^2}\star \xi\right\rrbracket^{-\frac12}\left(\frac{(kf^2)^2}{n}\left\llbracket\frac{\sqrt{n}}{kf^2}\star \xi\right\rrbracket\right)^{-\varepsilon},
\end{align*}
where the implied constant depends only on $M$, $\varepsilon$, $f_\infty$, $f_{q_i}$, $\pm$ and $\nu$.
\end{corollary}

\begin{proposition}\label{prop:xi0secondterm1}
Suppose that $k,f\in \ZZ_{(S)}^{>0}$ such that
\[
\left\llbracket\frac{\sqrt{n}}{kf^2}\star \xi\right\rrbracket\frac {(kf^2)^2}{n} \gg 1.
\]
Then for any $M>0$, we have
\[
\eqref{eq:xi0secondtermid}\ll \left(\frac{(kf^2)^2}{n}\left\llbracket\frac{\sqrt{n}}{kf^2}\star \xi\right\rrbracket\right)^{-M} \left\llbracket\frac{\sqrt{n}}{kf^2}\star \xi\right\rrbracket^{-2}\frac{n}{(kf^2)^2},
\]
where the implied constant depends only on $M$, $f_\infty$, $f_{q_i}$, $\pm$ and $\nu$.
\end{proposition}
\begin{proof}
Fix $\iota\in\{0,1\}$, $\epsilon_i\in \{0,\pm 1\}$, and $\sigma,\tau_i\in \{0,1\}$. We consider the following integral
\begin{equation}\label{eq:xi0secondtermsingle}
\begin{split}
& kf^2\int_{X_\iota}\int_{Y_\epsilon}\frac{\theta_{\infty,\sigma}^\pm(x)\theta_{q,\tau}(y,\pm nq^\nu)n^{-1/2}q^{-\nu/2}}{2\sqrt{|x^2\mp 1|_\infty|y^2\mp 4nq^\nu|_q'}}V_{\iota,\epsilon}\legendresymbol{kf^2 (4nq^\nu)^{-1/2}}{|x^2\mp 1|_\infty^{1/2}|y^2\mp 4nq^\nu|_q'^{1/2}}
  \\
  \times&
  \rme\legendresymbol{-2x\xi n^{1/2}q^{\nu/2}}{kf^2}\rme_{q}\legendresymbol{-y\xi}{kf^2}\rmd x\rmd y.
\end{split}
\end{equation}

By \autoref{thm:mainfourierestimate} for $\Phi=V_{\iota,\epsilon}$ and
\[
\varphi(x)=\frac{\theta_{\infty,\sigma}^\pm(x)}{\sqrt{|x^2\mp 1|_\infty}}\quad\text{and}\quad \psi_i(y_i)=\frac{\theta_{q_i,\tau_i}(y_i,\pm nq^\nu)}{\sqrt{|y_i^2\mp 4nq^\nu|_{q_i}'}}
\]
(which is valid by \autoref{prop:propertyh} and the results in \autoref{subsec:singularities}), we know that for any $M_1$ and $\sigma,\tau_i\in \{0,1\}$, \eqref{eq:xi0secondtermsingle} is bounded by
\begin{align*}
&\frac{kf^2}{\sqrt{n}} \left[\frac{(kf^2)^2}{4nq^\nu}\!\left(1+\left|\frac{2\xi\sqrt{nq^\nu}}{kf^2}\right|\right) \prod_{i=1}^{r}\left(1+\left|\frac{\xi}{kf^2}\right|_{q_i}\!\right)\!\right]^{\!-M_1} \!\! \left(1+\left|\frac{2\xi\sqrt{nq^\nu}}{kf^2}\right|\right)^{\!\!-1-\frac{\sigma-1}{2}} \prod_{i=1}^{r}\left(\!1+\left|\frac{\xi}{kf^2}\right|_{q_i}\right)^{\!\!-1-\frac{\tau_i-1}{2}}\\
\ll& \frac{kf^2}{\sqrt{n}} \left[\frac{(kf^2)^2}{4nq^\nu}\left(1+\left|\frac{2\xi\sqrt{nq^\nu}}{kf^2}\right|\right) \prod_{i=1}^{r}\left(1+\left|\frac{\xi}{kf^2}\right|_{q_i}\right)\right]^{-M_1} \left(1+\left|\frac{2\xi\sqrt{nq^\nu}}{kf^2}\right|\right)^{-\frac12} \prod_{i=1}^{r}\left(1+\left|\frac{\xi}{kf^2}\right|_{q_i}\right)^{-\frac12}\\
\ll&  \frac{kf^2}{\sqrt{n}}\left(\frac{(kf^2)^2}{n}\left\llbracket\frac{\sqrt{n}}{kf^2}\star \xi\right\rrbracket\right)^{-M_1}\left\llbracket\frac{\sqrt{n}}{kf^2}\star \xi\right\rrbracket^{-\frac12}.
\end{align*}
Let $M_1=3/2+M$, we obtain
\[
\eqref{eq:xi0secondtermsingle}
\ll \left(\frac{(kf^2)^2}{n}\left\llbracket\frac{\sqrt{n}}{kf^2}\star \xi\right\rrbracket\right)^{-M} \left\llbracket\frac{\sqrt{n}}{kf^2}\star \xi\right\rrbracket^{-2}\frac{n}{(kf^2)^2}.
\]
Now by summing over all $\iota$, $\epsilon_i$, $\sigma$, $\tau_i$ we get the desired result.
\end{proof}

\begin{proposition}\label{prop:xi0secondterm2}
Suppose that $k,f\in \ZZ_{(S)}^{>0}$ such that
\[
\left\llbracket\frac{\sqrt{n}}{kf^2}\star \xi\right\rrbracket\frac {(kf^2)^2}{n} \ll 1.
\]
Then for any $\varepsilon>0$ we have
\[
  \eqref{eq:xi0secondtermsingle}
   \ll \frac{kf^2}{\sqrt{n}}\left\llbracket\frac{\sqrt{n}}{kf^2}\star \xi\right\rrbracket^{-\frac 12}\left(\frac{(kf^2)^2}{n}\left\llbracket\frac{\sqrt{n}}{kf^2}\star \xi\right\rrbracket\right)^{-\varepsilon},
\]
where the implied constant depends only on $\varepsilon$, $f_\infty$, $f_{q_i}$, $\pm$ and $\nu$.
\end{proposition}
\begin{proof}
Fix $\iota\in\{0,1\}$, $\epsilon_i\in \{0,\pm 1\}$, and $\sigma,\tau_i\in \{0,1\}$. 
By \autoref{thm:mainfourierestimate} for $\Phi=V_{\iota,\epsilon}$, and
\[
\varphi(x)=\frac{\theta_{\infty,\sigma}^\pm(x)}{\sqrt{|x^2\mp 1|_\infty}}\quad\text{and}\quad \psi_i(y_i)=\frac{\theta_{q_i,\tau_i}(y_i,\pm nq^\nu)}{\sqrt{|y_i^2\mp 4nq^\nu|_{q_i}'}}
\]
(which is valid by \autoref{prop:propertyh} and the results in \autoref{subsec:singularities}), we know that for any $\varepsilon>0$ and $\sigma,\tau_i\in \{0,1\}$, \eqref{eq:xi0secondtermsingle} is bounded by
\begin{align*}
&\frac{kf^2}{\sqrt{n}} \left[\frac{(kf^2)^2}{4nq^\nu}\!\left(1+\left|\frac{2\xi\sqrt{nq^\nu}}{kf^2}\right|\right) \prod_{i=1}^{r}\left(1+\left|\frac{\xi}{kf^2}\right|_{q_i}\right)\right] ^{-\varepsilon} \left(1+\left|\frac{2\xi\sqrt{nq^\nu}}{kf^2}\right|\right)^{\!\!-1-\frac{\sigma-1}{2}} \prod_{i=1}^{r}\left(1+\left|\frac{\xi}{kf^2}\right|_{q_i}\right)^{\!\!-1-\frac{\tau_i-1}{2}}\\
\ll&\frac{kf^2}{\sqrt{n}} \left[\frac{(kf^2)^2}{4nq^\nu}\left(1+\left|\frac{2\xi\sqrt{nq^\nu}}{kf^2}\right|\right) \prod_{i=1}^{r}\left(1+\left|\frac{\xi}{kf^2}\right|_{q_i}\right)\right]^{-\varepsilon} \left(1+\left|\frac{2\xi\sqrt{nq^\nu}}{kf^2}\right|\right)^{-\frac12} \prod_{i=1}^{r}\left(1+\left|\frac{\xi}{kf^2}\right|_{q_i}\right)^{-\frac12}\\
\ll&  \frac{kf^2}{\sqrt{n}} \left(\frac{(kf^2)^2}{n}\left\llbracket\frac{\sqrt{n}}{kf^2}\star \xi\right\rrbracket\right)^{-\varepsilon} \left\llbracket\frac{\sqrt{n}}{kf^2}\star \xi\right\rrbracket^{-\frac12}.
\end{align*}
Summing over all $\iota$, $\epsilon_i$, $\sigma$, $\tau_i$ yields the desired conclusion.
\end{proof}

By \autoref{prop:xi0secondterm1} and \autoref{prop:xi0secondterm2} we obtain
\begin{corollary}\label{cor:xi0secondtermfinal}
For any $M>0$ and $\varepsilon>0$, we have
\begin{align*}
\eqref{eq:xi0secondterm}&\ll \sum_{\substack{k,f\in \ZZ_{(S)}^{>0},\ \xi\in \ZZ^S-\{0\}\\\left\llbracket\frac{\sqrt{n}}{kf^2}\star \xi\right\rrbracket\frac {(kf^2)^2}{n} \gg 1}}\frac{|\Kl_{k,f}^S(\xi,\pm nq^\nu)|}{k^2f^3}\left(\frac{(kf^2)^2}{n}\left\llbracket\frac{\sqrt{n}}{kf^2}\star \xi\right\rrbracket\right)^{-M} \left\llbracket\frac{\sqrt{n}}{kf^2}\star \xi\right\rrbracket^{-2}\frac{n}{(kf^2)^2}\\
&+ \sum_{\substack{k,f\in \ZZ_{(S)}^{>0},\ \xi\in \ZZ^S-\{0\}\\\left\llbracket\frac{\sqrt{n}}{kf^2}\star \xi\right\rrbracket\frac {(kf^2)^2}{n} \ll 1}}\frac{|\Kl_{k,f}^S(\xi,\pm nq^\nu)|}{k^2f^3}\frac{kf^2}{\sqrt{n}}\left\llbracket\frac{\sqrt{n}}{kf^2}\star \xi\right\rrbracket^{-\frac 12}\left(\frac{(kf^2)^2}{n}\left\llbracket\frac{\sqrt{n}}{kf^2}\star \xi\right\rrbracket\right)^{-\varepsilon},
\end{align*}
where the implied constant depends only on $M$, $\varepsilon$, $f_\infty$, $f_{q_i}$, $\pm$ and $\nu$.
\end{corollary}

Finally, we conclude that
\begin{corollary}\label{cor:xi0firstestimate}
For any $M>0$ and $\varepsilon>0$, we have
\begin{align*}
\eqref{eq:xinot0}&\ll \sqrt{n} \sum_{\substack{k,f\in \ZZ_{(S)}^{>0},\ \xi\in \ZZ^S-\{0\}\\\left\llbracket\frac{\sqrt{n}}{kf^2}\star \xi\right\rrbracket\frac {(kf^2)^2}{n} \gg 1}}\frac{|\Kl_{k,f}^S(\xi,\pm nq^\nu)|}{k^2f^3}\left(\frac{(kf^2)^2}{n}\left\llbracket\frac{\sqrt{n}}{kf^2}\star \xi\right\rrbracket\right)^{-M} \left\llbracket\frac{\sqrt{n}}{kf^2}\star \xi\right\rrbracket^{-2}\frac{n}{(kf^2)^2}\\
&+ \sum_{\substack{k,f\in \ZZ_{(S)}^{>0},\ \xi\in \ZZ^S-\{0\}\\\left\llbracket\frac{\sqrt{n}}{kf^2}\star \xi\right\rrbracket\frac {(kf^2)^2}{n} \ll 1}}\frac{|\Kl_{k,f}^S(\xi,\pm nq^\nu)|}{kf}\left\llbracket\frac{\sqrt{n}}{kf^2}\star \xi\right\rrbracket^{-\frac 12}\left(\frac{(kf^2)^2}{n}\left\llbracket\frac{\sqrt{n}}{kf^2}\star \xi\right\rrbracket\right)^{-\varepsilon},
\end{align*}
where the implied constant depends only on $M,\varepsilon,f_\infty,f_{q_i},\pm$ and $\nu$.
\end{corollary}
\begin{proof}
The conclusion follows from \autoref{cor:xi0firsttermfinal}, \autoref{cor:xi0secondtermfinal} and that $
\eqref{eq:xinot0}\ll n^{1/2}(\eqref{eq:xi0firstterm}+\eqref{eq:xi0secondterm})$.
\end{proof}

\subsection{Estimate of the Kloosterman sum part} From the previous subsection we know that \eqref{eq:xinot0} is bounded by the sum of the following two terms
\begin{equation}\label{eq:xi0kloosterman1}
\sqrt{n}\sum_{\substack{k,f\in \ZZ_{(S)}^{>0},\ \xi\in \ZZ^S-\{0\}\\\left\llbracket\frac{\sqrt{n}}{kf^2}\star \xi\right\rrbracket\frac {(kf^2)^2}{n} \gg 1}}\frac{|\Kl_{k,f}^S(\xi,\pm nq^\nu)|}{k^2f^3}\left(\frac{(kf^2)^2}{n}\left\llbracket\frac{\sqrt{n}}{kf^2}\star \xi\right\rrbracket\right)^{-M} \left\llbracket\frac{\sqrt{n}}{kf^2}\star \xi\right\rrbracket^{-2}\frac{n}{(kf^2)^2},
\end{equation}
and
\begin{equation}\label{eq:xi0kloosterman2}
\sum_{\substack{k,f\in \ZZ_{(S)}^{>0},\ \xi\in \ZZ^S-\{0\}\\\left\llbracket\frac{\sqrt{n}}{kf^2}\star \xi\right\rrbracket\frac {(kf^2)^2}{n} \ll 1}}\frac{|\Kl_{k,f}^S(\xi,\pm nq^\nu)|}{kf}\left\llbracket\frac{\sqrt{n}}{kf^2}\star \xi\right\rrbracket^{-\frac 12}\left(\frac{(kf^2)^2}{n}\left\llbracket\frac{\sqrt{n}}{kf^2}\star \xi\right\rrbracket\right)^{-\varepsilon} .
\end{equation}

\begin{lemma}\label{lem:kloosterman}
Let $k,f\in \ZZ_{(S)}^{>0}$ and $\xi\in \ZZ^S-\{0\}$. Let $F(k,f,\xi)$ be a nonnegative function. Let $\mf{C}(k,f,\xi)$ be a set of conditions, each of which is of the form $C(k,f,\xi)\ll 1$ or $C(k,f,\xi)\gg 1$ for some nonnegative function $C$. Then for any $\varepsilon>0$ we have
\[
\sum_{\substack{k,f\in \ZZ_{(S)}^{>0},\ \xi\in \ZZ^S-\{0\}\\\mf{C}(k,f,\xi)}}F(k,f,\xi)\Kl_{k,f}^S(\xi,\pm nq^\nu)\ll_\varepsilon \sum_{\substack{d^2\mid n,\, a,k,f\in \ZZ_{(S)}^{>0},\ \xi\in \ZZ^S-\{0\}\\\mf{C}(ak,df,ad\xi)}}(ad^2kf^2)^\varepsilon adk^{1/2}F(ak,df,ad\xi).
\]
\end{lemma}
\begin{proof}
By \autoref{cor:kloostermanfinal} we have
\begin{align*}
    & \Kl_{k,f}^S(\xi,\pm nq^\nu)  \\
   \ll_{\varepsilon}  &  \delta(\pm4nq^\nu;(f^{(q)})^2)\begin{dcases}
             (kf^2)^\varepsilon\sqrt{k\gcd(n,f^2)}\sqrt{\gcd\left( \frac{\xi^{(q)}}{\sqrt{\gcd(n,f^2)}},k\right)}, & \frac{k\sqrt{\gcd(n,f^2)}}{\rad k}\mid \xi^{(q)},  \\
              0, & \text{otherwise}.
            \end{dcases}
\end{align*}
If $\Kl_{k,f}^S(\xi,\pm nq^\nu)\neq 0$, then $x^2\equiv \pm 4nq^\nu\, (f^2)$ has solutions. Thus $\gcd(n,f^2)$ must be a square. Write $d=\sqrt{\gcd(n,f^2)}$. Then we obtain
\[
     \Kl_{k,f}^S(\xi,\pm nq^\nu)  
   \ll_\varepsilon   \begin{dcases}
             (kf^2)^\varepsilon   dk^{1/2}\sqrt{\gcd\left(\xi^{(q)}/d,k\right)}, & \frac{dk}{\rad k}\mid \xi^{(q)},  \\
              0, & \text{otherwise}.
            \end{dcases}
\]
Hence 
\[
\Kl_{k,f}^S(\xi,\pm nq^\nu) \ll_\varepsilon (kf^2)^\varepsilon dk^{1/2}\sqrt{\gcd\left(\xi^{(q)}/d,k\right)}.
\]

Let $a=\gcd(\xi^{(q)}/d,k)\in \ZZ_{(S)}^{>0}$. By making the change of variable $k\mapsto ak$, $f\mapsto df$ and $\xi\mapsto ad\xi$ we obtain
\begin{align*}
     &\sum_{\substack{k,f\in \ZZ_{(S)}^{>0},\ \xi\in \ZZ^S-\{0\} \\\mf{C}(k,f,\xi)}}F(k,f,\xi)\Kl_{k,f}^S(\xi,\pm nq^\nu) \\   
   \ll_\varepsilon & \sum_{\substack{d^2\mid n,\,k,f\in \ZZ_{(S)}^{>0},\ \xi\in \ZZ^S-\{0\} \\ \mf{C}(k,df,d\xi)}}F(k,df,d\xi)(kd^2f^2)^\varepsilon dk^{1/2}\sqrt{\gcd\left(\xi^{(q)}/d,k\right)} \\
   \ll_\varepsilon  & \sum_{\substack{d^2\mid n,\,a,k,f\in \ZZ_{(S)}^{>0},\ \xi\in \ZZ^S-\{0\}\\\mf{C}(ak,df,ad\xi)}}(ad^2kf^2)^\varepsilon adk^{1/2}F(ak,df,ad\xi).\qedhere
\end{align*}
\end{proof}

First we prove that $\eqref{eq:xi0kloosterman1}\ll n^{1/4+\varepsilon}$. We need the following lemmas:
\begin{lemma}\label{lem:estimatexilargesum}
Suppose that $M>1$. 
\begin{enumerate}[itemsep=0pt,parsep=0pt,topsep=0pt, leftmargin=0pt,labelsep=2.5pt,itemindent=15pt,label=\upshape{(\arabic*)}]
\item For any $\varepsilon>0$, we have
\[
\sum_{\xi\in \ZZ^S-\{0\}}\llbracket a\star \xi\rrbracket^{-M} \ll_{\varepsilon} |a|^{-M+\varepsilon}.
\]
\item For any $\varepsilon>0$, we have
\[
\sum_{\llbracket a\star\xi\rrbracket\gg b}\llbracket a\star\xi\rrbracket^{-M}\ll_\varepsilon |a|^{-1}|b|^{-M+1+\varepsilon}.
\]
\end{enumerate}
\end{lemma}
\begin{proof}
(1) follows from
\begin{align*}
   &\sum_{\xi\in \ZZ^S-\{0\}}\llbracket a\star \xi\rrbracket^{-M} \ll_\varepsilon  \sum_{\xi\in \ZZ^S-\{0\}}\left(1+|a\xi|_\infty\right)^{-M+\varepsilon}\prod_{i=1}^{r} (1+|\xi|_{q_i})^{-M} \\
   =&\sum_{I\subseteq \{1,\dots,r\}}\sum_{\substack{c_i\in \ZZ_{<0}\ i\in I\\ c_j\in \ZZ_{\geq 0}\ j\notin I}}\sum_{\substack{\xi\in \ZZ^S-\{0\}\\ v_{q_i}(\xi)=c_i}}\left(1+|a\xi|_\infty\right)^{-M+\varepsilon}\prod_{i=1}^{r} (1+|\xi|_{q_i})^{-M}\\
     \ll_\varepsilon &\sum_{I\subseteq \{1,\dots,r\}}\prod_{i\in I}\sum_{\substack{c_i=-\infty\\ i\in I}}^{-1}\sum_{\xi\in \ZZ-\{0\}}(|a\xi| q_1^{c_1}\dots q_r^{c_r})^{-M+\varepsilon}\prod_{i=1}^{r}(1+q_i^{-c_i})^{-M}
     \\
     \ll_\varepsilon &\sum_{\xi\in \ZZ-\{0\}}|a\xi|^{-M+\varepsilon}\sum_{I\subseteq \{1,\dots,r\}}\prod_{i\in I}\sum_{\substack{c_i=-\infty\\ i\in I}}^{-1} \frac{q_i^{c_i(M-\varepsilon)}}{(1+q_i^{c_i})^{M}}\ll_\varepsilon |a|^{-M+\varepsilon},
\end{align*}
and
(2) follows from
\begin{align*}
\sum_{\llbracket a\star\xi\rrbracket\gg b}\llbracket a\star\xi\rrbracket^{-M}&\ll \sum_{u\in \ZZ_{\leq 0}^r}\sum_{\substack{\xi\in \ZZ-\{0\}\\ \llbracket a\star q^u\xi\rrbracket\gg b}}(1+|aq^u\xi|)^{-M}\prod_{i=1}^{r}(1+q_i^{-u_i})^{-M}\\
&\ll \left(\sum_{q^{-u}\ll b}\sum_{|\xi|\gg \frac{1}{|aq^u|}\frac{b}{q^{-u}}}+\sum_{q^{-u}\gg b}\sum_{|\xi|\geq 1}\right)(1+|aq^u\xi|)^{-M}\prod_{i=1}^{r}(1+q_i^{-u_i})^{-M}\\
&\ll  \sum_{q^{-u}\ll b}|aq^u|^{-1}|bq^u|^{-M+1}|q^{-u}|^{-M}+\sum_{q^{-u}\gg b}|aq^u|^{-1-\varepsilon}\prod_{i=1}^{r}(1+q_i^{-u_i})^{-M}\\
&\ll |a|^{-1}|b|^{-M+1}\log^r |b|+|a|^{-1-\varepsilon}\sum_{q^{-u}\gg b}(q^u)^{M-1-\varepsilon}\ll_\varepsilon |a|^{-1}|b|^{-M+1+\varepsilon}.\qedhere
\end{align*}
\end{proof}

\begin{proposition}\label{prop:xi0kloosterman1}
For $M$ sufficiently large and any $\varepsilon>0$, we have
\[
\eqref{eq:xi0kloosterman1}\ll n^{\frac14+\varepsilon}.
\]
The implied constant depends only on $M$, $\varepsilon$, $f_\infty$, $f_{q_i}$, $\pm$ and $\nu$.
\end{proposition}
\begin{proof}
Since $\llbracket a \star \xi\rrbracket=\llbracket 1 \star a\xi\rrbracket$ for $a\in \ZZ_{(S)}$, we have
\[
\left\llbracket\frac{\sqrt{n}}{ad^2kf^2}\star ad\xi\right\rrbracket=\left\llbracket\frac{\sqrt{n}}{dkf^2}\star \xi\right\rrbracket
\]
for $a\in \ZZ_{(S)}^{>0}$ and $d^2\mid n$. Thus by \autoref{lem:kloosterman}, \eqref{eq:xi0kloosterman1} is bounded by
\[
\sqrt{n}\sum_{d^2\mid n}\sum_{\substack{a,k,f\in \ZZ_{(S)}^{>0},\ \xi\in \ZZ^S-\{0\}\\\left\llbracket\frac{\sqrt{n}}{dkf^2}\star \xi\right\rrbracket\frac {(ad^2kf^2)^2}{n} \gg 1}}\frac{(ad^2kf^2)^\varepsilon adk^{1/2}}{d^3f^3a^2k^2} \left(\frac{(ad^2kf^2)^2}{n}\left\llbracket\frac{\sqrt{n}}{dkf^2}\star \xi\right\rrbracket\right)^{-M}\left\llbracket\frac{\sqrt{n}}{dkf^2}\star \xi\right\rrbracket^{-2}\frac{n}{(ad^2kf^2)^2}
\]
for any $\varepsilon>0$.
We decompose the above sum as
\[
\sum_{\substack{a,k,f\in \ZZ_{(S)}^{>0},\ \xi\in \ZZ^S-\{0\}\\\left\llbracket\frac{\sqrt{n}}{dkf^2}\star \xi\right\rrbracket\frac {(ad^2kf^2)^2}{n} \gg 1}}=\sum_{\substack{a^2\gg \sqrt {n}\\k,f,\xi}}+ \sum_{\substack{a^2\ll \sqrt {n}\\ f^2\gg \frac{\sqrt{n}}{a^2}\\k,\xi}}+ \sum_{\substack{a^2\ll \sqrt {n}\\ f^2\ll \frac{\sqrt{n}}{a^2}\\ k\gg \frac{\sqrt{n}}{a^2f^2} \\\xi}}+\sum_{\substack{a^2\ll \sqrt {n}\\ f^2\ll \frac{\sqrt{n}}{a^2}\\ k\ll \frac{\sqrt{n}}{a^2f^2} \\\left\llbracket\frac{\sqrt{n}}{dkf^2}\star \xi\right\rrbracket\frac {(ad^2kf^2)^2}{n} \gg 1}}
\]
and bound these four terms separately. 

For the first term, by \autoref{lem:estimatexilargesum} (1) we have
\begin{align*}
\sqrt{n}\sum_{d^2\mid n}\sum_{\substack{a^2\gg \sqrt {n}\\k,f,\xi}}\cdots
\ll &\sqrt{n}\sum_{d^2\mid n}\sum_{\substack{a^2\gg \sqrt {n}\\k,f}}\frac{(ad^2kf^2)^\varepsilon adk^{1/2}}{a^2d^3k^2f^3}\left(\frac{(ad^2kf^2)^2}{n} \right)^{-M-1}\sum_{\xi\in \ZZ^S-\{0\}}\left\llbracket\frac{\sqrt{n}}{dkf^2}\star \xi\right\rrbracket^{-M-2} \\
\ll &\sqrt{n}\sum_{d^2\mid n}\sum_{\substack{a^2\gg \sqrt {n}\\k,f}}\frac{(ad^2kf^2)^\varepsilon}{ad^2k^{3/2}f^3}\left(\frac{(ad^2kf^2)^2}{n} \right)^{-M-1}\left(\frac{\sqrt{n}}{dkf^2}\right)^{-M-2+\varepsilon}\\
=&n^{\frac{M+1+\varepsilon}{2}}\sum_{d^2\mid n}\sum_{a^2\gg \sqrt{n}}\frac{1}{a^{3+2M-\varepsilon}d^{4+3M-\varepsilon}}\sum_{k,f}\frac{1}{k^{3/2+M}f^{3+2M}}\\
\ll &n^{\frac{M+1+\varepsilon}{2}}\sum_{a^2\gg \sqrt{n}}\frac{1}{a^{3+2M-\varepsilon}}\ll n^{\frac{M+1+\varepsilon}{2}} (n^{\frac14})^{-2-2M+\varepsilon}\ll n^{\varepsilon}.
\end{align*}

For the second term we have
\begin{align*}
\sqrt{n}\sum_{d^2\mid n}\sum_{\substack{a^2\ll \sqrt {n}\\ f^2\gg \frac{\sqrt{n}}{a^2}\\k,\xi}}\cdots
\ll&\sqrt{n}\sum_{d^2\mid n}\sum_{\substack{a^2\ll \sqrt {n}\\ f^2\gg \frac{\sqrt{n}}{a^2}\\k}}\frac{(ad^2kf^2)^\varepsilon adk^{1/2}}{a^2d^3k^2f^3}\left(\frac{(ad^2kf^2)^2}{n} \right)^{-M-1}\left(\frac{\sqrt{n}}{dkf^2}\right)^{-M-2+\varepsilon}\\
\ll&n^{\frac{M+1+\varepsilon}{2}}\sum_{d^2\mid n}\sum_{a^2\ll \sqrt{n}}\frac{1}{a^{3+2M-\varepsilon}d^{4+3M-\varepsilon}}\sum_{f\gg \frac{\sqrt{n}}{a^2}}\frac{1}{f^{3+2M}}\sum_{k}\frac{1}{k^{3/2+M}}\\
\ll &n^{\frac{M+1+\varepsilon}{2}}\sum_{a^2\ll \sqrt{n}}\frac{1}{a^{3+2M-\varepsilon}}\sum_{f\gg\frac{\sqrt{n}}{a^2}}\frac{1}{f^{3+2M}}\ll n^{\frac{M+1+\varepsilon}{2}}\sum_{a^2\ll \sqrt{n}}\frac{1}{a^{3+2M-\varepsilon}} \left(\frac{\sqrt{n}}{a^2}\right)^{-2-2M}\\
\ll &n^{\frac{M+1+\varepsilon}{2}}(n^{\frac14})^{-2-2M+\varepsilon}\ll n^{\varepsilon}.
\end{align*}

For the third term we have
\begin{align*}
&\sqrt{n}\sum_{d^2\mid n}\sum_{\substack{a^2\ll \sqrt {n}\\ f^2\ll \frac{\sqrt{n}}{a^2}\\k\gg \frac{\sqrt{n}}{a^2f^2}\\\xi}}\cdots
\ll\sqrt{n} \sum_{d^2\mid n}\sum_{\substack{a^2\ll \sqrt {n}\\ f^2\ll \frac{\sqrt{n}}{a^2}\\k\gg \frac{\sqrt{n}}{a^2f^2}}}\frac{(ad^2kf^2)^\varepsilon adk^{1/2}}{a^2d^3k^2f^3} \left(\frac{(ad^2kf^2)^2}{n} \right)^{-M-1}\left(\frac{\sqrt{n}}{dkf^2}\right)^{-M-2+\varepsilon}\\
\ll &n^{\frac{M+1+\varepsilon}{2}}\sum_{d^2\mid n}\sum_{a^2\ll \sqrt{n}}\frac{1}{ a^{3+2M-\varepsilon}d^{4+3M-\varepsilon}}\sum_{f\ll \frac{\sqrt{n}}{a^2}}\frac{1}{f^{3+2M}}\sum_{k\gg \frac{\sqrt{n}}{a^2f^2}}\frac{1}{k^{3/2+M}}\\
\ll &n^{\frac{M+1+\varepsilon}{2}}\sum_{a^2\ll \sqrt{n}}\frac{1}{a^{3+2M-\varepsilon}}\sum_{f\ll \frac{\sqrt{n}}{a^2}}\frac{1}{f^{3+2M}} \left(\frac{\sqrt{n}}{a^2f^2}\right)^{-\frac12-M}\ll n^{\frac14+\varepsilon}\sum_{a^2\ll \sqrt{n}}\frac{1}{a^{2-3\varepsilon}}\sum_{f\ll \frac{\sqrt{n}}{a^2}}\frac{1}{f^{2}}\ll n^{\frac14+\varepsilon}.
\end{align*}

For the fourth term, by \autoref{lem:estimatexilargesum} (2) we have
\begin{align*}
&\sqrt{n}\sum_{d^2\mid n}\sum_{\substack{a^2\ll \sqrt {n}\\ f^2\ll \frac{\sqrt{n}}{a^2}\\k\ll \frac{\sqrt{n}}{a^2f^2}\\\left\llbracket\frac{\sqrt{n}}{dkf^2}\star \xi\right\rrbracket\frac {(ad^2kf^2)^2}{n} \gg 1}}\cdots\\
\ll &\sqrt{n}\sum_{d^2\mid n}\sum_{\substack{a^2\ll \sqrt {n}\\ f^2\ll \frac{\sqrt{n}}{a^2}\\k\ll \frac{\sqrt{n}}{a^2f^2}}}\frac{(ad^2kf^2)^\varepsilon adk^{1/2}}{a^2d^3k^2f^3} \left(\frac{(ad^2kf^2)^2}{n} \right)^{-M-1} \sum_{\left\llbracket\frac{\sqrt{n}}{dkf^2}\star \xi\right\rrbracket\frac{(adkf^2)^2}{n} \gg 1}\left\llbracket\frac{\sqrt{n}}{dkf^2}\star \xi\right\rrbracket^{-M-2}\\
\ll &\sqrt{n}\sum_{d^2\mid n}\sum_{\substack{a^2\ll \sqrt {n}\\ f^2\ll \frac{\sqrt{n}}{a^2}\\k\ll \frac{\sqrt{n}}{a^2f^2}}}\frac{(ad^2kf^2)^{\varepsilon}adk^{1/2}}{a^2d^3k^2f^3} \left(\frac{(ad^2kf^2)^2}{n} \right)^{-M-1} \left(\frac{\sqrt{n}}{dkf^2}\right)^{-1}\left(\frac{(ad^2kf^2)^2}{n} \right)^{M+1-\varepsilon}\\
\ll &n^{\varepsilon}\sum_{d^2\mid n}\sum_{a^2\ll \sqrt{n}}\frac{1}{a^{1+\varepsilon}d^{1+2\varepsilon}}\sum_{f\ll \frac{\sqrt{n}}{a^2}}\frac{1}{f^{1+2\varepsilon}}\sum_{k\ll \frac{\sqrt{n}}{a^2f^2}}\frac{1}{k^{1/2+\varepsilon}}\ll n^{\varepsilon}\sum_{a^2\ll \sqrt{n}}\frac{1}{a}\sum_{f\ll \frac{\sqrt{n}}{a^2}}\frac{1}{f}\left(\frac{\sqrt{n}}{a^2f^2} \right)^{\frac12} \\ \ll &n^{\frac14+\varepsilon}\sum_{a^2\ll \sqrt{n}}\frac{1}{a^2}\sum_{f\ll \frac{\sqrt{n}}{a^2}}\frac{1}{f^2}\ll  n^{\frac14+\varepsilon}.
\end{align*}
Hence we obtain $\eqref{eq:xi0kloosterman1}\ll n^{1/4+\varepsilon}$.
\end{proof}

Now we prove that $\eqref{eq:xi0kloosterman2}\ll n^{1/4+\varepsilon}$. We need the following lemma.

\begin{lemma}\label{lem:estimatexismallsum}
Suppose that $b\gg 1$ and $\alpha<1$. Then we have
\[
\sum_{\substack{\xi\in \ZZ^S\\\llbracket a\star \xi\rrbracket \ll b}}\llbracket a\star \xi\rrbracket^{-\alpha}\ll \frac{b^{1-\alpha}}{a}\log^r b.
\]
\end{lemma}
\begin{proof}
Suppose that $I\subseteq\{1,\dots,r\}$ such that $u_i=v_{q_i}(\xi)<0$ if and only if $i\in I$. Then $\llbracket a\star \xi\rrbracket=(1+|a\xi|)\prod_{i=1}^{r}(1+|\xi|_{q_i})\ll b$ only if 
\[
|a\xi|\ll \frac{b}{\prod_{i\in I}(1+q_i^{-u_i})}.
\]
Since $\prod_{i\in I}q_i^{-u_i}\xi\in \ZZ$ we know that the sum of such $\xi$ is 
\begin{align*}
&\ll \sum_{\substack{v_{q_i}(\xi)=u_i\\\llbracket a\star \xi\rrbracket \ll b}} (1+|a\xi|)^{-\alpha}\prod_{i=1}^{r}(1+|\xi|_{q_i})^{-\alpha}\ll \sum_{\substack{\xi\in \ZZ-\{0\}\\|\xi|\leq \frac{b\prod_{i\in I}q_i^{-u_i}}{a\prod_{i\in I}(1+q_i^{-u_i})}}}a^{-\alpha}\left(\prod_{i\in I}q_i^{-u_i}\right)^\alpha \prod_{i\in I}(1+q_i^{-u_i})^{-\alpha}|\xi|^{-\alpha}\\
&\ll a^{-\alpha}\legendresymbol{\prod_{i\in I}q_i^{-u_i}}{\prod_{i\in I}(1+q_i^{-u_i})}^\alpha \legendresymbol{b\prod_{i\in I}q_i^{-u_i}}{a\prod_{i\in I}(1+q_i^{-u_i})}^{1-\alpha}\ll a^{-\alpha} \legendresymbol{b}{a}^{1-\alpha}\frac{\prod_{i\in I}q_i^{-u_i}}{\prod_{i\in I}(1+q_i^{-u_i})}\ll \frac{b^{1-\alpha}}{a}.
\end{align*}
Moreover, such $u_i$ must satisfy $\prod_{i\in I}q_i^{-u_i}\ll b$. Hence the number of such $(u_i)_{i\in I}$ is $\ll \log^{\# I} b\ll \log^r b$. This establishes the desired conclusion.
\end{proof}

\begin{proposition}\label{prop:xi0kloosterman2}
For any $\varepsilon>0$, we have
\[
\eqref{eq:xi0kloosterman2}\ll n^{\frac14+\varepsilon}.
\]
The implied constant depends only on $\varepsilon$, $f_\infty$, $f_{q_i}$, $\pm$ and $\nu$.
\end{proposition}
\begin{proof}
By \autoref{lem:kloosterman} we have
\[
   \eqref{eq:xi0kloosterman2}\ll  \sum_{d^2\mid n}\sum_{\substack{a,k,f\in \ZZ_{(S)}^{>0},\,\xi\in \ZZ^S-\{0\}\\\left\llbracket\frac{\sqrt{n}}{dkf^2}\star \xi\right\rrbracket\frac {(ad^2kf^2)^2}{n} \ll 1}}\frac{(ad^2kf^2)^\varepsilon adk^{1/2}}{adkf}\left\llbracket\frac{\sqrt{n}}{dkf^2}\star \xi\right\rrbracket^{-\frac 12}\left(\frac{(ad^2kf^2)^2}{n}\left\llbracket\frac{\sqrt{n}}{dkf^2}\star \xi\right\rrbracket\right)^{-\varepsilon}.
\]

Therefore by \autoref{lem:estimatexismallsum} with $\alpha=1/2+\varepsilon$, we obtain
\begin{align*}
   \eqref{eq:xi0kloosterman2}\ll & \sum_{d^2\mid n}\sum_{a^2\ll \sqrt{n}}\sum_{f^2\ll \frac{\sqrt{n}}{a^2}}\sum_{k\ll \frac{\sqrt{n}}{a^2f^2}}\sum_{\xi:\left\llbracket\frac{\sqrt{n}}{dkf^2}\star \xi\right\rrbracket\frac {(ad^2kf^2)^2}{n} \ll 1}\frac{(ad^2kf^2)^\varepsilon adk^{1/2}}{adkf}\\
\times&\left\llbracket\frac{\sqrt{n}}{dkf^2}\star \xi\right\rrbracket^{-\frac 12}\left(\frac{(ad^2kf^2)^2}{n}\left\llbracket\frac{\sqrt{n}}{dkf^2}\star \xi\right\rrbracket\right)^{-\varepsilon} \\
\ll & \sum_{d^2\mid n}\sum_{a^2\ll \sqrt{n}}\sum_{f^2\ll \frac{\sqrt{n}}{a^2}}\sum_{k\ll \frac{\sqrt{n}}{a^2f^2}}\frac{(ad^2kf^2)^\varepsilon k^{1/2}}{kf}\left(\frac{(ad^2kf^2)^2}{n}\right)^{ -\varepsilon}\left(\frac{n}{(ad^2kf^2)^2}\right)^{\frac 12-\varepsilon}\frac{dkf^2}{\sqrt{n}}\log^r\frac{n}{(ad^2kf^2)^2}\\
\ll& \sum_{d^2\mid n}\sum_{a^2\ll \sqrt{n}}\sum_{f^2\ll \frac{\sqrt{n}}{a^2}}\sum_{k\ll \frac{\sqrt{n}}{a^2f^2}}\frac{(ad^2kf^2)^{\varepsilon} k^{1/2}}{kf}\left(\frac{n}{(ad^2kf^2)^2}\right)^{\frac 12+\varepsilon}\frac{dkf^2}{\sqrt{n}}\\
\ll & n^{\varepsilon}\sum_{d^2\mid n}\sum_{a^2\ll \sqrt{n}}\sum_{f^2\ll \frac{\sqrt{n}}{a^2}}\sum_{k\ll \frac{\sqrt{n}}{a^2f^2}} 
\frac{1}{a^{1+\varepsilon} d^{1+2\varepsilon} k^{1/2+\varepsilon}f^{1+2\varepsilon}}\ll n^{\varepsilon}\sum_{a^2\ll \sqrt{n}}\sum_{f^2\ll \frac{\sqrt{n}}{a^2}}\frac{1}{af} \legendresymbol{\sqrt{n}}{a^2f^2}^{\frac12}\\
\ll & n^{\frac 14+\varepsilon} \sum_{a^2\ll \sqrt{n}}\frac{1}{a^2}\sum_{f^2\ll \frac{\sqrt{n}}{a^2}}\frac{1}{f^2}
\ll n^{\frac14+\varepsilon}.\qedhere
\end{align*}
\end{proof}

By \autoref{cor:xi0firstestimate}, \autoref{prop:xi0kloosterman1}, and \autoref{prop:xi0kloosterman2}, we obtain the main result in this section.
\begin{theorem}\label{thm:xinot0estimate}
For any $\varepsilon>0$, we have
\[
\Sigma(\xi\neq 0)\ll n^{\frac14+\varepsilon},
\]
where the implied constant depends only on $f_\infty$, $f_{q_i}$ and $\varepsilon$.
\end{theorem}

Combining \autoref{cor:estimateellipticremain2} and \autoref{thm:xinot0estimate} establishes our main result \autoref{thm:totalestimate}.

\appendix
\section{Estimate of the Fourier transform on the semilocal space}\label{sec:fourier}
In this section we will derive several estimates of the Fourier transform on the semilocal space $\QQ_S=\RR\times\QQ_{q_1}\times  \dots\times \QQ_{q_r}$ with \emph{singularities}, generalizing the results in \cite[Appendix A]{altug2017} for the case over $\RR$. 
The archimedean part was analyzed by Altu\u{g}. Hence we mainly need to derive the estimate for the nonarchimedean part, which is a direct computation. 

We will often use the notation $s=\sigma+\rmi t$ in this section. Recall that $\cS$ is the set of smooth functions $\Phi$ on $\lopen 0,+\infty\ropen$ such that for any $k\in \ZZ_{\geq 0}$, $\Phi^{(k)}(x)$ are of rapid decay as $x\to +\infty$.
Our main goal of this section is to prove the following theorems:

\begin{theorem}\label{thm:mainfourierestimate}
Suppose that $\Phi\in \cS$. Let $\widetilde{\Phi}(s)$ be its Mellin transform. Suppose that $\widetilde{\Phi}(s)$ has a holomorphic continuation on $\Re s>0$, and is $\ll \rme^{-\uppi|t|/2}$ when $\sigma>0$ is fixed. Fix a sign $\pm$, $\nu\in \ZZ^r$ and $\iota\in \{0,1\}$, $\epsilon_i\in \{0,\pm 1\}$. Let $N=\pm nq^\nu$. Let $a_0,a_1,\dots,a_r$ be complex numbers such that $\Re a_i>-2$ for all $i$. Let 
\[
\varphi(x)=|x^2\mp 1|^{\Re a_0/2} \overline{\varphi}(x)\quad\text{and}\quad \psi_i(y_i)=|y_i^2-4N|_{q_i}'^{\Re a_i/2} \overline{\psi_i}(y_i),
\]
where $\overline{\varphi}(x)$ is a smooth function on the open set $X_{\iota}=\{x\in \RR\ |\ \omega_\infty(x)=\iota\}$ up to the boundary, supported on a bounded set of $X_{\iota}$, and $\overline{\psi_i}(y_i)$ are smooth functions on $Y_{\epsilon_i}=\{y_i\in \QQ_{q_i}\ |\ \omega_i(y_i)=\epsilon_i\}$ up to the boundary, supported on a bounded set of  $Y_{\epsilon_i}$. Suppose that for the $\iota=0$ case, around $x=\pm 1$, $\varphi(x)$ has asymptotic expansions
\[
\varphi(\pm(1-x))\sim |x|^{\frac{a_0}{2}}\sum_{m=0}^{+\infty}c_m^\pm x^m.
\]
Let $Y_\epsilon=Y_{\epsilon_1}\times\dots\times Y_{\epsilon_r}$ and 
\[
\bF_{\iota,\epsilon}(\xi,\eta)=\int_{X_\iota}\int_{Y_\epsilon}\varphi(x)\prod_{i=1}^{r}\psi_i(y_i)\Phi\legendresymbol{C}{|x^2\mp 1|^{1/2}|y^2-4N|_q'^{1/2}}\rme(-x\xi)\rme_{q}(-y\eta)\rmd x\rmd y.
\]
Let
\[
\llbracket\xi,\eta\rrbracket=(1+|\xi|_\infty)\prod_{i=1}^{r}(1+|\eta_i|_{q_i}).
\]
Then for any $M>0$,
  \[
  \bF_{\iota,\epsilon}(\xi,\eta)\ll_{\varphi,\psi,\pm,\nu,M}\left(|C|^2\llbracket\xi,\eta\rrbracket\right)^{-M} (1+|\xi|_\infty)^{-1-\frac{a_0}{2}}\prod_{i=1}^{r}(1+|\eta_i|_{q_i})^{-1-\frac{a_i}{2}}.
  \]
\end{theorem}
Note that $M$ can be taken arbitrarily large or arbitrarily close to $0$ in the above theorem.

\begin{theorem}\label{thm:mainfourierestimate2}
Suppose that $\Phi\in \cS$. Let $\widetilde{\Phi}(s)$ be its Mellin transform. Suppose that $\widetilde{\Phi}(s)$ has a meromorphic continuation to $\Re s\geq -2$ with only a simple pole at $s=0$, and is $\ll \rme^{-\uppi|t|/2}$ when $\sigma\geq -2$ is fixed. Fix a sign $\pm$, $\nu\in \ZZ^r$. Let $N=\pm nq^\nu$. Let $a_0,a_1,\dots,a_r\in \ZZ_{\geq -1}$. Let 
\[
\varphi(x)=|x^2\mp 1|^{a_0/2} \overline{\varphi}(x)\quad\text{and}\quad \psi_i(y_i)=|y_i^2-4N|_{q_i}'^{a_i/2} \overline{\psi_i}(y_i),
\]
where $\overline{\varphi}(x)$ is a smooth and compactly supported function on the open set $\RR$ up to the boundary, and $\overline{\psi_i}(y_i)$ are smooth and compactly supported functions on $Y_{\epsilon_i}$. Suppose that for the $\iota=0$ case, around $x=\pm 1$, $\varphi(x)$ has an asymptotic expansion
\[
\varphi(\pm(1-x))\sim |x|^{\frac{a_0}{2}}\sum_{m=0}^{+\infty}c_m^\pm x^m.
\]
Let $Y_\epsilon=Y_{\epsilon_1}\times\dots\times Y_{\epsilon_r}$ and 
\[
\bF_\epsilon(\xi,\eta)=\int_{\RR}\int_{Y_\epsilon}\varphi(x)\prod_{i=1}^{r}\psi_i(y_i)\Phi\legendresymbol{C}{|x^2\mp 1|^{1/2}|y^2-4N|_q'^{1/2}}\rme(-x\xi)\rme_{q}(-y\eta)\rmd x\rmd y.
\]
Then
\[
\bF_\epsilon(\xi,\eta)\ll_{\varphi,\psi,\pm,\nu} (1+|\xi|_\infty)^{-1-\frac{a_0}{2}}\prod_{i=1}^{r}(1+|\eta_i|_{q_i})^{-1-\frac{a_i}{2}}.
\]
Moreover, if $a_0=0$, then the factor $(1+|\xi|_\infty)^{-1-\frac{a_0}{2}}$ can be replaced by 
\[
\left(|C|^2\prod_{i=1}^{r}(1+|\eta_i|_{q_i})\right)^{1-\varepsilon}+(1+|\xi|_\infty)^{-2}
\]
for any $\varepsilon>0$. (The implied constant then also depends on $\varepsilon$.)
\end{theorem}

\subsection{Nonarchimedean analysis}
First we consider the nonarchimedean part.
We want to estimate the $p$-adic integral of the form
\begin{equation}\label{eq:ladicfourier}
\bI(s,\eta)=\int_{\omega_p(y)=\epsilon}\psi(y)|y^2-4N|_p'^{s/2} \rme_p(-y\eta)\rmd y,
\end{equation}
where $N\in \QQ_p^\times$,
\[
\omega_p(y)=\legendresymbol{(y^2-4N)|y^2-4N|_p'}{p},
\]
$\epsilon\in \{0,\pm 1\}$, $\psi(y)$ is a smooth function on $\{y\in \QQ_p\ |\ \omega_p(y)=\epsilon\}$ up to the boundary, with bounded support. 

We use the notation $2\nu=v_p(N)$ in this subsection.

\begin{proposition}\label{prop:ladicsmooth}
If $N$ does not have square roots in $\QQ_p$, then $\bI(s,\eta)$ is an entire function for $s$, and
\[
\bI(s,\eta)\ll_{\nu,\sigma,\psi,M} (1+|\eta|_p)^{-M}
\]
for any $M>0$.
\end{proposition}
\begin{proof}
We first prove that $\bI(s,\eta)$ is entire in $s$. We have
\[
|y^2-4N|_p\gg 1.
\]
Indeed, if $v_p(4N)$ is odd, then $v_p(y^2)\neq v_p(4N)$. Hence
\[
|y^2-4N|_p=\max\{|y^2|_p,|4N|_p\}\geq |4N|_p \gg_{\nu} 1.
\]
If $2\nu'=v_p(4N)$ is even, then for $v_p(y)\neq \nu'$ we have the same estimate. Suppose that $v_p(y)= \nu'$. Write $y=p^{\nu'}y_0$ and $4N=p^{2\nu'}n_0$. Then
\[
|y^2-4N|_p=p^{-2\nu'}|y_0^2-n_0|_p.
\]
We must have $|y_0^2-n_0|_p\geq 1/p^2$. Otherwise, by Hensel's lemma $y_0^2-n_0$ would have a root for $y_0$ in $\QQ_p$ and thus $y^2-4N$ would have a root for $y$ in $\QQ_p$. Hence $|y^2-4N|_p\gg 1$. 

Next we prove that $|y^2-4N|_p'$ is smooth at $y_0$ if $y_0^2-4N\neq 0$. Indeed, for any $\delta\in \QQ_p$ we have
\[
\frac{(y_0+\delta)^2-4N}{y_0^2-4N}=1+\delta\frac{2y_0+\delta}{y_0^2-4N}.
\]
Suppose that $|\delta|_p<\max\{|2y_0|_p,1\}$ and $|\delta|_p\leq p^{-2}|y_0^2-4N|_p/\max\{|2y_0|_p,1\}$. Then we have
\[
\left|\delta\frac{2y_0+\delta}{y_0^2-4N}\right|_p\leq |\delta|_p\frac{\max\{|2y_0|_p,1\}}{|y_0^2-4N|_p}\leq p^{-2}.
\]
Hence $((y_0+\delta)^2-4N)/(y_0^2-4N)\in 1+p^2\ZZ_p$ and hence
\[
|(y_0+\delta)^2-4N|_p'=|y_0^2-4N|_p'.
\]
Hence $|y^2-4N|_p'$ is smooth at $y_0$.

Since $\psi(y)$ is smooth up to the boundary with bounded support, the integral converges for $s$ uniformly on every compact subset and thus defines an entire function for $s$. 
Moreover, $|y_0|_p$ and $|y_0^2-4N|_p$ are bounded when $y_0\in \supp(\psi)$ (depending only on $\nu$). Hence we can let the $\delta$ above to be independent of $y_0\in \supp(\psi)$. Hence there exists $\kappa$ depending only on $\nu$ and $\psi$ such that 
\[
\omega_p(y+\delta)=\legendresymbol{((y+\delta)^2-4N)|(y+\delta)^2-4N|_p'}{p}= \legendresymbol{(y^2-4N)|y^2-4N|_p'}{p}=\omega_p(y)
\]
and
\[
\psi(y+\delta)|(y+\delta)^2-4N|_p'^{s/2}=\psi(y)|y^2-4N|_p'^{s/2}
\]
for any $s$ and $|\delta|_p\leq p^{-\kappa}$. By making the change of variable $y\mapsto y+\delta$ in \eqref{eq:ladicfourier}, we obtain
\[
\bI(s,\eta)= \int_{\omega_p(y)=\epsilon}\psi(y+\delta)|(y+\delta)^2-4N|_p'^{s/2} \rme_p(-(y+\delta)\eta)\rmd y=\rme_p(-\delta\eta)\bI(s,\eta).
\]
Hence if $\bI(s,\eta)\neq 0$, we must have $\rme_p(-\delta\eta)=1$ for all $|\delta|_p\leq p^{-\kappa}$. Hence $|\eta|_p\leq p^\kappa$. That is, $\bI(s,\eta)=0$ if $|\eta|_p> p^\kappa$. For $|\eta|_p\leq p^\kappa$ we have
\[
|\bI(s,\eta)|\ll  \int_{\omega(y)=\epsilon}|\psi(y)||y^2-4N|_p'^{\sigma/2}\rmd y\ll_{\psi,\sigma} 1.
\]
Hence the conclusion follows by combining these two cases together.
\end{proof}

\begin{proposition}\label{prop:ladicnonsingular}
Suppose that $N$ has square roots $\pm \sqrt{N}$ in $\QQ_p$ and the support of $\psi(y)$ is disjoint from
\[
\{y\in \QQ_p\ |\ |y-2\sqrt{N}|_p \leq p^{-L}\}\cup \{y\in \QQ_p\ |\ |y+2\sqrt{N}|_p\leq p^{-L}\}.
\]
Then $\bI(s,\eta)$ is an entire function for $s$, and
\[
\bI(s,\eta)\ll_{\nu,\sigma,\psi,M,L} (1+|\eta|_p)^{-M}
\]
for any $M>0$.
\end{proposition}
\begin{proof}
It suffices to prove that $|y_0^2-4N|_p,|y_0^2-4N|_p'\gg_{\nu,L} 1$ for $y_0\in \supp(\psi)$ in this case. The remaining arguments follow verbatim from \autoref{prop:ladicsmooth}. Now we prove this claim. By \autoref{lem:modifynorm} it suffices to show that $|y^2-4N|_p\gg_{\nu,L} 1$ for $y\in \supp(\psi)$. Indeed, we have
\[
|y^2-4N|_p=|y+2\sqrt{N}|_p|y-2\sqrt{N}|_p> p^{-2L}.\qedhere
\]
\end{proof}

Suppose that $N$ has square roots $\pm \sqrt{N}$ in $\QQ_p$. For $L$ sufficiently large (we will determine it later), we consider the integral
\begin{equation}\label{eq:ladicfouriersupport}
\bJ(s,\eta)=\int_{\omega_p(y)=\epsilon, |y-2\sqrt{N}|_p\leq p^{-L}}|y^2-4N|_p'^{s/2} \rme_p(-y\eta)\rmd y.
\end{equation}
It has singularities at $y=\pm 2\sqrt{N}$ with exponent $s/2$. Hence $\bJ(s,\eta)$ converges for $\Re s>-2$ (by \autoref{lem:ladiconverge}) and defines a holomorphic function on this half plane.

\begin{theorem}\label{thm:ladicsingular}
For any $\eta\in \QQ_p$, $\bJ(s,\eta)$ is a sum of at most three functions of the following form
\[
A\frac{1}{1-p^{-2-s}}B^{s/2}\qquad\text{and}\qquad AB^{s/2},
\]
where $A\ll (1+|\eta|_p)^{-1}$ and $B\asymp (1+|\eta|_p)^{-1}$. The implied constants depend only on $p,\nu$ and $L$. 

In particular, for any $\sigma>-2$ we have
\[
\bJ(s,\eta)\ll (1+|\eta|_p)^{-1-\frac{\sigma}{2}},
\]
where the implied constant depends only on $\sigma,p,\nu$ and $L$. 
\end{theorem}
In the proof of the above theorem we will add an additional technical assumption for $L$ for convenience of our computation (which will be explicit in the proof). Since we can always choose $L$ satisfying this assumption in the following arguments, it will not affect our main results.
\begin{proof}
In the proof we omit the subscript $p$ in $\omega_p$, $|\cdot|_p$ and $|\cdot|_p'$. We denote $w=-v_p(\eta)$.

Let $\widetilde{\omega}(y)=\omega(y+2\sqrt{N})$. By making the change of variable $y\mapsto y+2\sqrt{N}$ we have
\[
\bJ(s,\eta)=\rme_p(-2\sqrt{N}\eta)\int_{y\in p^L\ZZ_p,\ \widetilde{\omega}(y)=\epsilon}|y(y+4\sqrt{N})|'^{s/2}\rme_p(-y\eta)\rmd y.
\]

We assume that $L\geq \nu+4$. In this case we have
\[
\frac{y(y+4\sqrt{N})}{4y\sqrt{N}}=1+\frac{y}{4\sqrt{N}}\in 1+p^2\ZZ_p.
\]
Hence $|y(y+4\sqrt{N})|'=|4y\sqrt{N}|'$. Therefore
\[
\bJ(s,\eta)=\rme_p(-2\sqrt{N}\eta)\int_{y\in p^L\ZZ_p,\ \widetilde{\omega}(y)=\epsilon}|4y\sqrt{N}|'^{s/2}\rme_p(-y\eta)\rmd y.
\]
Moreover
\[
\widetilde{\omega}(y)=\legendresymbol{y(y+4\sqrt{N})|y(y+4\sqrt{N})|'}{p}= \legendresymbol{4y\sqrt{N}|4y\sqrt{N}|'}{p}.
\]

\underline{\emph{Case 1:}}\ \ $p\neq 2$.

We have $\bJ(s,\eta)=\rme_p(-2\sqrt{N}\eta)\sum_{u=L}^{+\infty}J_{u}(s,\eta)$, where
\[
J_u(s,\eta)=\int_{y\in p^u\ZZ_p^\times,\ \widetilde{\omega}(y)=\epsilon}|4y\sqrt{N}|'^{s/2}\rme_p(-y\eta)\rmd y.
\]

\underline{\emph{Case 1.1:}}\ \ $\epsilon=0$.

For convenience, we choose $L\not\equiv \nu\ (2)$ in this case. 

In this case $\widetilde{\omega}(y)=\epsilon$ if and only if $v_p(4y\sqrt{N})$ is odd, that is, $u\not\equiv \nu \ (2)$. For such $u$, we have $|4y\sqrt{N}|'=p^{-u-\nu+1}$. Hence $J_u(s,\eta)=0$ if $|\eta| > p^u$ and
\[
J_u(s,\eta)=p^{-u}\left(1-\frac{1}{p}\right)p^{(-u-\nu+1)s/2}
\] 
otherwise.

For $w\geq L$ we have
\[
\bJ(s,\eta)=\rme_p(-2\sqrt{N}\eta)\sum_{\substack{u=w\\u\not\equiv \nu \ (2)}}^{+\infty}p^{-u}\left(1-\frac{1}{p}\right)p^{(-u-\nu+1)s/2},
\]
which equals
\[
\rme_p(-2\sqrt{N}\eta)\left(1-\frac{1}{p}\right)p^{(-\nu+1)s/2}\frac{p^{-w(1+s/2)}}{1-p^{-2-s}}= \rme_p(-2\sqrt{N}\eta)|\eta|^{-1} \left(1-\frac{1}{p}\right)\frac{1}{1-p^{-2-s}}|\sqrt{N}\eta^{-1}|'^{s/2}
\]
if $w\not\equiv \nu\ (2)$ and equals
\[
\rme_p(-2\sqrt{N}\eta)\left(1-\frac{1}{p}\right)p^{(-\nu+1)s/2}\frac{p^{-(w+1)(1+s/2)}}{1-p^{-2-s}}=
\rme_p(-2\sqrt{N}\eta)|\eta|^{-1} \left(1-\frac{1}{p}\right)\frac{p^{-1}}{1-p^{-2-s}}|\sqrt{N}\eta^{-1}|'^{s/2}
\]
if $w\equiv \nu\ (2)$.

For $w<L$ we have
\begin{align*}
   &\bJ(s,\eta)=\rme_p(-2\sqrt{N}\eta)\sum_{\substack{u=L\\u\not\equiv \nu \ (2)}}^{+\infty}p^{-u}\left(1-\frac{1}{p}\right)p^{(-u-\nu+1)s/2}, \\
      =&\rme_p(-2\sqrt{N}\eta)\left(1-\frac{1}{p}\right)p^{(-\nu+1)s/2}\frac{p^{-L(1+s/2)}}{1-p^{-2-s}} =\rme_p(-2\sqrt{N}\eta)p^{-L}\left(1-\frac{1}{p}\right)\frac{1}{1-p^{-2-s}}p^{(-\nu-L+1)s/2}.
\end{align*}

\underline{\emph{Case 1.2:}}\ \ $\epsilon=\pm 1$. 

We choose $L$ such that $L\equiv \nu\ (2)$ in this case.

In this case $\widetilde{\omega}(y)=\epsilon$ if and only if $v_p(4y\sqrt{N})$ is even and 
\[
\legendresymbol{4y_0n_0}{p}=\epsilon,
\]
where $y_0=y|y|$ and $n_0=\sqrt{N}|\sqrt{N}|$.
For $u\equiv \nu\ (2)$ we have
\[
J_u(s,\eta)=p^{-u}p^{(-u-\nu)s/2}\int_{y_0\in\ZZ_p^\times,\ \chi_p(y_0)=\epsilon\chi_p(n_0)} \rme_p(-p^uy_0\eta)\rmd y_0,
\]
where $\chi_p=\legendresymbol{\cdot}{p}$. Note that for $a\in \ZZ_p$,
\begin{align*}
   \int_{y_0\in\ZZ_p^\times,\ \chi_p(y_0)=\epsilon\chi_p(n_0)} \rme_p(-p^uy_0\eta)\rmd y_0 & =\int_{y_0\in\ZZ_p^\times,\ \chi_p(y_0)=\epsilon\chi_p(n_0)} \rme_p(-p^u(y_0+p a)\eta)\rmd y_0  \\
     & =\rme_p(-p^{u+1}a\eta)\int_{y_0\in\ZZ_p^\times,\ \chi_p(y_0)=\epsilon\chi_p(n_0)} \rme_p(-p^uy_0\eta)\rmd y_0.
\end{align*}
Hence for $|\eta|>p^{u+1}$ we have $J_u(s,\eta)=0$. For $|\eta| \leq p^u$ we have $\rme_p(-p^uy_0\eta)=1$ for all $y_0\in \ZZ_p^\times$. Hence
\[
J_u(s,\eta)=\frac{1}{2}p^{-u}p^{(-u-\nu)s/2}\left(1-\frac{1}{p}\right).
\]
Finally we consider the case $|\eta|=p^{u+1}$. Write $\eta=\eta_0p^{-u-1}$, where $\eta_0\in \ZZ_p^\times$. Then
\[
\int_{y_0\in\ZZ_p^\times,\ \chi_p(y_0)=\epsilon\chi_p(n_0)} \rme_p(-p^uy_0\eta)\rmd y_0 =\int_{y_0\in\ZZ_p^\times,\ \chi_p(y_0)=\epsilon\chi_p(n_0\eta_0)}\rme_p\left(-\frac{y_0}{p}\right)\rmd y_0=\frac{1}{p}\mf{g}_p^{\epsilon\chi_p(n_0\eta_0)},
\]
where for any sign $\pm$, the \emph{Gauss sum} $\mf{g}_p^\pm$ is defined by
\[
\mf{g}_p^\pm=\sum_{\substack{x\in \FF_p\\\legendresymbol{x}{p}=\pm 1}}\rme\legendresymbol{x}{p}.
\]

For $w\geq L+1$, we have
\[
\bJ(s,\eta)=\frac{1}{2}\rme_p(-2\sqrt{N}\eta)\sum_{\substack{u=w\\u\equiv \nu \ (2)}}^{+\infty}p^{-u}\left(1-\frac{1}{p}\right)p^{(-u-\nu)s/2}
\]
for $w\equiv \nu\ (2)$ and
\[
\bJ(s,\eta)=\frac{1}{2}\rme_p(-2\sqrt{N}\eta)\sum_{\substack{u=w\\u\equiv \nu \ (2)}}^{+\infty}p^{-u}\left(1-\frac{1}{p}\right)p^{(-u-\nu)s/2}+\rme_p(-2\sqrt{N}\eta)p^{-w+1}p^{(-w+1-\nu)s/2} \frac{1}{p}\mf{g}_p^{\epsilon\chi_p(n_0\eta_0)}
\]
for $w\not\equiv \nu\ (2)$. Hence for $w\equiv \nu\ (2)$ we have
\begin{align*}
\bJ(s,\eta)&=\frac{1}{2}\rme_p(-2\sqrt{N}\eta)\left(1-\frac{1}{p}\right)p^{-\nu s/2}\frac{p^{-w(1+s/2)}}{1-p^{-2-s}}\\
&=\frac{1}{2}\rme_p(-2\sqrt{N}\eta)\left(1-\frac{1}{p}\right)|\eta|^{-1} \frac{1}{1-p^{-2-s}}|\sqrt{N}\eta^{-1}|'^{s/2},
\end{align*}
while for $w\not\equiv \nu\ (2)$ we have
\begin{align*}
   \bJ(s,\eta) & =\frac{1}{2}\rme_p(-2\sqrt{N}\eta)\left(1-\frac{1}{p}\right)p^{-\nu s/2}\frac{p^{-(w+1)(1+s/2)}}{1-p^{-2-s}}+\rme_p(-2\sqrt{N}\eta)p^{-w}p^{(-w-\nu+1)s/2}\mf{g}_p^{\epsilon\chi_p(n_0\eta_0)}\\
     & =\frac{1}{2}\rme_p(-2\sqrt{N}\eta)\left(1-\frac{1}{p}\right)|\eta|^{-1} \frac{p^{-1}}{1-p^{-2-s}}(p^{-1}|\sqrt{N}\eta^{-1}|)^{s/2}\\
     &+\rme_p(-2\sqrt{N}\eta)|\eta|^{-1} |\sqrt{N}\eta^{-1}|'^{s/2}\mf{g}_p^{\epsilon\chi_p(n_0\eta_0)}.
\end{align*}

For $w\leq L$ we have
\begin{align*}
   \bJ(s,\eta) & =\frac{1}{2}\rme_p(-2\sqrt{N}\eta)\sum_{\substack{u=L\\u\equiv \nu \ (2)}}^{+\infty}p^{-u}\left(1-\frac{1}{p}\right)p^{(-u-\nu)s/2} \\
     & =\frac{1}{2}\rme_p(-2\sqrt{N}\eta)\left(1-\frac{1}{p}\right)p^{-\nu s/2}\frac{p^{-L(1+s/2)}}{1-p^{-2-s}}.
\end{align*}

\underline{\emph{Case 2:}}\ \ $p= 2$.

We have $\bJ(s,\eta)=\rme_p(-2\sqrt{N}\eta)\sum_{u=L}^{+\infty}J_{u}(s,\eta)$, where
\[
J_u(s,\eta)=\int_{p^u\ZZ_p^\times,\ \widetilde{\omega}(y)=\epsilon}|4y\sqrt{N}|'^{s/2}\rme_p(-y\eta)\rmd y.
\]

\underline{\emph{Case 2.1:}}\ \ $\epsilon=0$. 

We choose $L\equiv \nu \ (2)$ in this case.

In this case, $u\not\equiv \nu\ (2)$, or $u\equiv \nu\ (2)$ and $y_0n_0\equiv 3\ (4)$, where $y_0=y|y|$ and $n_0=\sqrt{N}|\sqrt{N}|$. 

We first consider the case $u\not\equiv \nu\ (2)$. We have $|4y\sqrt{N}|'=p^{3-2-u-\nu}=p^{1-u-\nu}$. Hence $J_u(s,\eta)=0$ if $|\eta| > p^u$ and
\[
J_u(s,\eta)=\frac{1}{2}p^{-u}p^{(-u-\nu+1)s/2}
\] 
otherwise.

Now we consider the case for $u\equiv \nu\ (2)$ and $y_0n_0\equiv 3\ (4)$. We have $|4y\sqrt{N}|'=p^{2-2-u-\nu}=p^{-u-\nu}$. Hence
\[
J_u(s,\eta)=p^{-u}p^{(-u-\nu)s/2}\int_{y_0\in\ZZ_p^\times,\ y_0n_0\equiv 3\ (4)} \rme_p(-p^uy_0\eta)\rmd y_0.
\]
Note that for any $a\in \ZZ_p$,
\begin{align*}
   \int_{y_0\in\ZZ_p^\times,\ y_0n_0\equiv 3\ (4)} \rme_p(-p^uy_0\eta)\rmd y_0 & =\int_{y_0\in\ZZ_p^\times,\ y_0n_0\equiv 3\ (4)} \rme_p(-p^u(y_0+p^2 a)\eta)\rmd y_0  \\
     & =\rme_p(-p^{u+2}a\eta)\int_{y_0\in\ZZ_p^\times,\ y_0n_0\equiv 3\ (4)} \rme_p(-p^uy_0\eta)\rmd y_0.
\end{align*}
Hence for $|\eta|>p^{u+2}$ we have $J_u(s,\eta)=0$. For $|\eta| \leq p^u$ we have $\rme_p(-p^uy_0\eta)=1$ for all $y_0\in \ZZ_p^\times$. Hence
\[
J_u(s,\eta)=\frac{1}{4}p^{-u}p^{(-u-\nu)s/2}.
\]
For $|\eta|= p^{u+1}$ we have 
\[
\rme_p(-p^u y_0\eta)=\rme_p\left(-\frac{y_0\eta_0}{p}\right)=-1
\]
if $\eta=p^{-u-1}\eta_0$. Hence 
\[
J_u(s,\eta)=-\frac{1}{4}p^{-u}p^{(-u-\nu)s/2}.
\]
Finally, we consider the case $|\eta|= p^{u+2}$. By setting $\eta=p^{-u-2}\eta_0$, we have
\[
\rme_p(-p^u y_0\eta)=\rme_p\left(-\frac{y_0\eta_0}{4}\right)=\rme\legendresymbol{3n_0^{-1}\eta_0}{4}= \rme\left(-\frac{n_0\eta_0}{4}\right).
\]
Hence
\[
J_u(s,\eta)=\frac14p^{-u}p^{(-u-\nu)s/2}\rme\left(-\frac{n_0\eta_0}{4}\right).
\]

For $w\geq L+1$ we have 
\begin{align*}
\bJ(s,\eta)&=\frac{1}{2}\rme_p(-2\sqrt{N}\eta)\sum_{\substack{u=w\\u\not\equiv \nu \ (2)}}^{+\infty}p^{-u}p^{(-u-\nu+1)s/2}+\frac{1}{4}\rme_p(-2\sqrt{N}\eta)\sum_{\substack{u=w\\u\equiv \nu \ (2)}}^{+\infty}p^{-u}p^{(-u-\nu)s/2}\\
&+\frac14 \rme_p(-2\sqrt{N}\eta)p^{-w+2}p^{(-w-\nu+2)s/2}\rme\left(-\frac{n_0\eta_0}{4}\right)\\
&=\frac{1}{2}\rme_p(-2\sqrt{N}\eta)|\eta|^{-1}\frac{1}{1-p^{-2-s}}|\sqrt{N}\eta^{-1}|^{s/2}+ \rme_p(-2\sqrt{N}\eta)|\eta|^{-1}(p^2|\sqrt{N}\eta^{-1}|)^{s/2}\rme\left(-\frac{n_0\eta_0}{4}\right)
\end{align*}
for $w\equiv \nu\ (2)$ and
\begin{align*}
\bJ(s,\eta)&=\frac{1}{2}\rme_p(-2\sqrt{N}\eta)\sum_{\substack{u=w\\u\not\equiv \nu \ (2)}}^{+\infty}p^{-u}p^{(-u-\nu+1)s/2}+\frac{1}{4}\rme_p(-2\sqrt{N}\eta)\sum_{\substack{u=w\\u\equiv \nu \ (2)}}^{+\infty}p^{-u}p^{(-u-\nu)s/2}\\
&-\frac14\rme_p(-2\sqrt{N}\eta) p^{-w+1}p^{(-w-\nu+1)s/2}\\
&=\frac{1}{2}\rme_p(-2\sqrt{N}\eta)|\eta|^{-1}\frac{1}{1-p^{-2-s}}(p|\sqrt{N}\eta^{-1}|)^{\frac s2}+ \frac{1}{8}\rme_p(-2\sqrt{N}\eta)|\eta|^{-1}\frac{1}{1-p^{-2-s}}(p^{-1}|\sqrt{N}\eta^{-1}|)^{\frac s2}\\
&-\frac12\rme_p(-2\sqrt{N}\eta)|\eta|^{-1}(p|\sqrt{N}\eta^{-1}|)^{\frac s2}
\end{align*}
for $w\not\equiv \nu\ (2)$. 

For $w\leq L$ we have
\begin{align*}
\bJ(s,\eta)&=\rme_p(-2\sqrt{N}\eta)\left[\frac{1}{2}\sum_{\substack{u=L\\u\not\equiv \nu \ (2)}}^{+\infty}p^{-u}p^{(-u-\nu+1)s/2}+\frac{1}{4}\sum_{\substack{u=L\\u\equiv \nu \ (2)}}^{+\infty}p^{-u}p^{(-u-\nu)s/2}\right]\\
&=\rme_p(-2\sqrt{N}\eta)\frac{1}{2}p^{-L}\frac{1}{1-p^{-2-s}}(p^{-L-\nu})^{s/2}.
\end{align*}

\underline{\emph{Case 2.2:}}\ \ $\epsilon=\pm 1$.

For simplicity we only deal with the case  $\epsilon=-1$. The other case is similar.
We choose $L\equiv \nu \ (2)$ in this case.

We have $|4y\sqrt{N}|'=p^{-2-u-\nu}$. Hence
\[
J_u(s,\eta)=p^{-u}p^{(-2-u-\nu)s/2}\int_{y_0\in\ZZ_p^\times,\ y_0n_0\equiv 5\ (8)} \rme_p(-p^uy_0\eta)\rmd y_0.
\]
Note that
\begin{align*}
   \int_{y_0\in\ZZ_p^\times,\ y_0n_0\equiv 5\ (8)} \rme_p(-p^uy_0\eta)\rmd y_0 & =\int_{y_0\in\ZZ_p^\times,\ y_0n_0\equiv 5\ (8)} \rme_p(-p^u(y_0+p^3 a)\eta)\rmd y_0  \\
     & =\rme_p(-p^{u+3}a\eta)\int_{y_0\in\ZZ_p^\times,\ y_0n_0\equiv 5\ (8)} \rme_p(-p^uy_0\eta)\rmd y_0.
\end{align*}
Hence for $|\eta|>p^{u+3}$ we have $J_u(s,\eta)=0$. For $|\eta| \leq p^u$ we have $\rme_p(-p^uy_0\eta)=1$ for all $y_0\in \ZZ_p^\times$. Hence
\[
J_u(s,\eta)=\frac{1}{8}p^{-u}p^{(-2-u-\nu)s/2}.
\]
For $|\eta|= p^{u+1}$ we have 
\[
\rme_p(-p^u y_0\eta)=\rme_p\left(-\frac{y_0\eta_0}{p}\right)=-1
\]
if $\eta=p^{-u-1}\eta_0$. Hence 
\[
J_u(s,\eta)=-\frac{1}{8}p^{-u}p^{(-2-u-\nu)s/2}.
\]
For $|\eta|= p^{u+2}$ we write $\eta=p^{-u-2}\eta_0$. Then we have
\[
\rme_p(-p^u y_0\eta)=\rme_p\left(-\frac{y_0\eta_0}{4}\right)=\rme\left(\frac{5n_0^{-1}\eta_0}{4}\right)= \rme\left(\frac{n_0\eta_0}{4}\right).
\]
Hence
\[
J_u(s,\eta)=\frac{1}{8}p^{-u}p^{(-2-u-\nu)s/2}\rme\left(\frac{n_0\eta_0}{4}\right).
\]
Finally, for $|\eta|= p^{u+3}$ we write $\eta=p^{-u-3}\eta_0$. Then we have
\[
\rme_p(-p^u y_0\eta)=\rme_p\left(-\frac{y_0\eta_0}{8}\right)=\rme\left(\frac{5n_0^{-1}\eta_0}{8}\right)= \rme\left(\frac{5n_0\eta_0}{8}\right).
\]
Hence
\[
J_u(s,\eta)=\frac{1}{8}p^{-u}p^{(-2-u-\nu)s/2}\rme\left(\frac{5n_0\eta_0}{8}\right).
\]

For $w\geq L+2$ we have 
\begin{align*}
\bJ(s,\eta)&=\frac{1}{8}\rme_p(-2\sqrt{N}\eta)\sum_{\substack{u=w\\u\equiv \nu \ (2)}}^{+\infty}p^{-u}p^{(-2-u-\nu)s/2}+ \frac{1}{8}\rme_p(-2\sqrt{N}\eta)p^{-w+2}p^{(-w-\nu)s/2}\rme\left(\frac{n_0\eta_0}{4}\right)\\
&=\frac{1}{8}\rme_p(-2\sqrt{N}\eta)|\eta|^{-1}\frac{1}{1-p^{-2-s}}|4\sqrt{N}\eta^{-1}|^{s/2}+ \frac12\rme_p(-2\sqrt{N}\eta)|\eta|^{-1} |\sqrt{N}\eta^{-1}|^{s/2}\rme\left(\frac{n_0\eta_0}{4}\right)
\end{align*}
for $w\equiv \nu\ (2)$ and
\begin{align*}
\bJ(s,\eta)&=\frac{1}{8}\rme_p(-2\sqrt{N}\eta)\sum_{\substack{u=w\\u\equiv \nu \ (2)}}^{+\infty}p^{-u}p^{(-2-u-\nu)s/2}-\frac18\rme_p(-2\sqrt{N}\eta) p^{-w+1}p^{(-1-w-\nu)s/2}\\
&+ \frac{1}{8}\rme_p(-2\sqrt{N}\eta)p^{-w+3}p^{(1+w-\nu)s/2}\rme\legendresymbol{5n_0\eta_0}{8}\\
&=\frac{1}{16}\rme_p(-2\sqrt{N}\eta)|\eta|^{-1}\frac{1}{1-p^{-2-s}}(p^{-3}|\sqrt{N}\eta^{-1}|)^{s/2}- \frac{1}{4}\rme_p(-2\sqrt{N}\eta)|\eta|^{-1}|2\sqrt{N}\eta^{-1}|^{s/2}\\
&+\rme_p(-2\sqrt{N}\eta)|\eta|^{-1}(p|\sqrt{N}\eta^{-1}|)^{s/2}\rme\legendresymbol{5n_0\eta_0}{8}
\end{align*}
for $w\not\equiv \nu\ (2)$. 

For $w\leq L$ we have
\begin{align*}
\bJ(s,\eta)&=\frac{1}{8}\rme_p(-2\sqrt{N}\eta)\sum_{\substack{u=L\\u\equiv \nu \ (2)}}^{+\infty}p^{-u}p^{(-2-u-\nu)s/2}= \frac{1}{8}\rme_p(-2\sqrt{N}\eta)\frac{p^{-L(1+s/2)}}{1-p^{-2-s}}p^{(-\nu-2)s/2}\\
&=\frac{1}{8}\rme_p(-2\sqrt{N}\eta)p^{-L}\frac{1}{1-p^{-2-s}}p^{(-L-\nu-2)s/2}.
\end{align*}

For $w=L+1$ we have
\begin{align*}
\bJ(s,\eta)&=\frac{1}{8}\rme_p(-2\sqrt{N}\eta)\sum_{\substack{u=w\\u\equiv \nu \ (2)}}^{+\infty}p^{-u}p^{(-2-u-\nu)s/2}-\frac18\rme_p(-2\sqrt{N}\eta) p^{-L}p^{(-2-L-\nu)s/2}\\
&=\frac{1}{32}\rme_p(-2\sqrt{N}\eta)p^{-L}\frac{1}{1-p^{-2-s}}p^{(-L-\nu-4)s/2}-\frac18 \rme_p(-2\sqrt{N}\eta) p^{-L}p^{(-2-L-\nu)s/2}.
\end{align*}

Finally, by combining all the computations of $\bJ(s,\eta)$ together we obtain our result.
\end{proof}

\subsection{Archimedean analysis}
Now we consider the archimedean part. We will mainly use the results of \cite[Appendix A]{altug2017}. We consider the integral
\begin{equation}\label{eq:archimedeanfourier}
\bI(s,\xi)=\int_{\omega_\infty(x)=\iota}\varphi(x)|x^2\mp 1|^{s/2} \rme(-x\xi)\rmd x,
\end{equation}
where 
\[
\omega_\infty(x)=\begin{cases}
             0, & x^2\mp 1>0, \\
             1, & x^2\mp 1<0,
           \end{cases}
\]
and $\varphi(x)$ is a smooth function with bounded support.
\begin{proposition}\label{prop:archimedeansmooth}
If $\sgn N=-1$, then $\bI(s,\xi)$ is an entire function for $s$, and
\[
\bI(s,\xi)\ll_{\sigma,\varphi,M} (1+|\xi|)^{-M}
\]
for any $M>0$.
\end{proposition}
\begin{proof}
In this case
\[
\bI(s,\xi)=\int_{\RR}\varphi(x)|x^2+1|^{s/2}\rme(-x\xi)\rmd x.
\]

Since $\varphi(x)$ is smooth up to the boundary, supported in a bounded set and $|x^2+1|\gg 1$, the integral converges absolutely and uniformly on every compact subset for $s$ and hence defines an entire function. 

We consider the cases $|\xi|\ll 1$ and $|\xi|\gg 1$ separately.

\underline{\emph{Case 1:}}\ \ $|\xi|\ll 1$.

In this case, we bound the integral trivially and obtain
\[
\bI(s,\xi)\ll\int_{\omega_\infty(x)=\iota}|\varphi(x)||x^2+1|^{\frac\sigma2}\rmd x\ll_{\sigma,\varphi} 1.
\]
Since $1+|\xi|\asymp 1$ in this case, we obtain the result in this case.

\underline{\emph{Case 2:}}\ \ $|\xi|\gg 1$. 

In this case, by integration by parts $K$ times we obtain
\[
\bI(s,\xi)=\frac{1}{(-\dpii\xi)^K}\int_{\RR}\frac{\rmd^K}{\rmd x^K}\left(\varphi(x)|x^2+1|^{\frac s2}\right)\rme(-x\xi)\rmd x.
\]
By Leibniz rule, we have
\[
\frac{\rmd^K}{\rmd x^K}\left(\varphi(x)|x^2+1|^{s/2}\right)=\sum_{j=0}^{K}\binom{K}{j}\varphi^{(K-j)}(x) \frac{\rmd^j}{\rmd x^j}|x^2+1|^{\frac s2}.
\]
By induction, $\frac{\rmd^j}{\rmd x^j}|x^2+1|^{s/2}=a_j(x)|x^2+1|^{s/2}$, where 
\[
a_0(x)=1,\qquad a_{j+1}(x)=a_j'(x)+2xa_j(x)\log|x^2+1|.
\]
Hence $a_j(x)$ are smooth functions. This yields
\[
\bI(s,\xi)\ll |\xi|^{-K}\int_{\RR}\sum_{j=0}^{K}\binom{K}{j}|\varphi^{(K-j)}(x)||a_j(x)| |x^2+1|^{\frac\sigma2}\rmd x\ll_{\varphi,\sigma,K} |\xi|^{-K}.
\]
Finally, let $K=\lceil M\rceil$ and noting that $1+|\xi|\asymp |\xi|$ we obtain our result in this case.
\end{proof}

\begin{proposition}\label{prop:archimedeannonsingular}
Suppose that $\sgn N=+1$ and the support of $\varphi(x)$ is disjoint from
\[
\{x\in \RR\ |\ |x-1| \leq \kappa\}\cup \{x\in \RR\ |\ |x+1|\leq \kappa\}
\]
for some $\kappa>0$.
Then $\bI(s,\xi)$ is an entire function for $s$, and
\[
\bI(s,\xi)\ll_{\sigma,\varphi,M,\kappa} (1+|\xi|)^{-M}
\]
for any $M>0$.
\end{proposition}
\begin{proof}
Follows almost verbatim from \autoref{prop:archimedeansmooth}. The only differences are the following:
\begin{enumerate}[itemsep=0pt,parsep=0pt,topsep=0pt, leftmargin=0pt,labelsep=2.5pt,itemindent=9pt,label=\textbullet]
  \item 
We need to consider the boundary terms at $x=\pm 1\pm \kappa$ but they all vanish by our assumption.
\item We need to consider $\frac{\rmd^j}{\rmd x^j}|x^2-1|^{s/2}=b_j(x)|x^2-1|^{s/2}$, where 
\[
b_0(x)=1,\qquad b_{j+1}(x)=b_j'(x)+2xb_j(x)\log|x^2-1|.
\]
$b_j(x)$ are smooth in $\supp(\varphi)$ in this case. \qedhere
\end{enumerate}
\end{proof}

Now we assume that $\sgn N=+1$.
Let $\phi\in C_c^\infty(\RR)$ such that $\supp(\phi)\subseteq [-2,2]$ and $\phi\equiv 1$ in $[-1,1]$. For $\kappa\in \lopen0,1/2\ropen$, we define
\[
\phi_\kappa(x)=\phi\left(\frac{|1-|x||}{\kappa}\right).
\]
Then $\phi_\kappa$ is supported near $\pm 1$. 

Finally, we state our main theorems for the archimedean case, which are just reformulations of the last three results in \cite[Appendix A]{altug2017}.

Let $\Xi(s)$ be a holomorphic function on the half plane $\Re s>0$. Also, we assume that $\Xi(s)$ has rapid decay in $t$ for each fixed $\sigma$.
Let $c_i^\pm$ be complex numbers. We define the integral transform
\[
\cA_{\varphi,m}^{\sigma,\pm}(\Xi)(x)=\frac{1}{\dpii}\int_{(\sigma)}\Xi(s) c_m^\pm\legendresymbol{s}{2}\Gamma\left(m+1+\frac{a+s}{2}\right)x^{-\frac s2}\rmd s
\]
with
\[
c_m^\pm \legendresymbol{s}{2}=\legendresymbol{\rmi}{2\uppi}^{1+m+\frac{a+s}{2}}2^{\frac s2}\sum_{\substack{j+k=m\\ j,k\geq 0}}\frac{c_k^\pm}{(-2)^j}\binom{s/2}{j}.
\]
Note that $c_m^\pm(s/2)$ is an entire function with at most polynomial growth in the vertical direction. On the other hand, $\Gamma(m+1+\frac{a+s}{2})$ has exponential decay in the vertical direction by Stirling formula.

\begin{remark}\label{rem:powerfunction}
In \cite{altug2017}, the author does not interpret the branch of the logarithm to define the function of the form $z^s$. By going over the proof of Lemma A.11 and Lemma A.13 of loc. cit., we can find that the functions $(\rmi x)^{s}$ and $x^s$ for $x\in \RR^\times$ (which are all the functions of such a form occurring in the definition) are defined as follows:
\begin{equation}\label{eq:imaginaryxs}
(\rmi x)^s=\rme^{s(\log|x|+\sgn x\rmi \uppi/2)}=\begin{cases}
                                              |x|^s\rme^{\rmi\uppi s/2}, & x>0, \\
                                              |x|^s\rme^{-\rmi\uppi s/2}, & x<0,
                                            \end{cases}
\end{equation}
and
\begin{equation}\label{eq:realxs}
x^s=\sgn x|x|^s=\begin{cases}
                                              |x|^s, & x>0, \\
                                              -|x|^s, & x<0.
                                            \end{cases}
\end{equation}
Note that such definition is \emph{not} equal to the usual value if $x<0$ and $s\in 2\ZZ$. Hence we will use new notation $[x]^s$ for \eqref{eq:realxs}. In this notation we actually have
\[
\cA_{\varphi,m}^{\sigma,\pm}(\Xi)(x)=\frac{1}{\dpii}\int_{(\sigma)}\Xi(s) c_m^\pm\legendresymbol{s}{2}\Gamma\left(m+1+\frac{a+s}{2}\right)[x]^{-\frac s2}\rmd s.
\]
\end{remark}

\begin{theorem}\label{thm:mainarchimedeantheorem1}
Let $\Xi(s)$ be a holomorphic function on the half plane $\Re s>0$. Also, we assume that $\Xi(s)\ll_\sigma \rme^{-\uppi|t|/2}$ for each fixed $\sigma$. Let $\varphi(x)=|1-x^2|^{\frac a2}\overline{\varphi}(x)$, where $\overline{\varphi}(x)$ is smooth in and up to the boundary of $\lopen -1,1\ropen$. Assume that around $x=\pm 1$ we have asymptotic expansions
\[
\varphi(\pm(1-x))\sim |x|^{\frac{a}{2}}\sum_{m=0}^{+\infty}c_m^\pm x^m
\]
as in the sense of \cite{altug2017}. Then for any $\sigma,\tau\in \RR_{>0}$, $|\xi|\gg 1$, and $K\in \ZZ_{\geq 0}$ we have
\begin{align*}
    & \frac{1}{\dpii}\int_{(\sigma)}\frac{\Xi(s)}{C^s}\int_{-1}^{1}\varphi(x)|1-x^2|^{\frac s2} \rme(-x\xi)\rmd x\rmd s \\
   =  & \sum_{\pm}\sum_{m=0}^{K}\frac{\rme(\mp \xi)\cA_{\varphi,m}^{\sigma,\pm}(\Xi)(\pm C^2\xi)}{(\pm\xi)^{m+1}[\pm \xi]^{\frac a2}}+ O\left(\|\Xi\|_{2\tau}|C^2\xi|^{-\tau}|\xi|^{-K-2-\frac{a}{2}}\right),
\end{align*}
where we define
\[
\|\Xi\|_{\sigma}=\left|\int_{(\sigma)}\Xi(s)\Omega(s)\rmd s\right|
\]
for a certain function $\Omega(s)$ which can be explicitly computed by the proof of \cite[Lemma A.13]{altug2017}.
The implied constant depends only on $\varphi,K,\sigma,\tau$ and $a$. 
If $\Xi(s)$ has a meromorphic continuation to $\Re s>-2$, which has at most a simple pole at $s=0$ and holomorphic otherwise, we can take $\tau=0$. (In this case, we take the principal value.)
\end{theorem} 
\begin{proof}
This theorem is essentially the content of \cite[Theorem A.14]{altug2017}. The only difference is that we have to take care of the bound of $\Xi(s)$ in the error term. By going through the proofs in loc. cit., the error term is obtained by shifting the contour to $\sigma=2\tau$. Hence the conclusion is immediate.
\end{proof}

\begin{theorem}\label{thm:mainarchimedeantheorem2}
Let $\Xi(s)$ be a holomorphic function on the half plane $\Re s>0$. Also, we assume that $\Xi(s)\ll_\sigma \rme^{-\uppi|t|/2}$ for each fixed $\sigma$. Let $\varphi(x)=|1-x^2|^{\frac a2}\overline{\varphi}(x)$, where $\overline{\varphi}(x)$ is a smooth up to the boundary of $\lopen -\infty,-1\ropen\cup\lopen 1,+\infty\ropen$ with bounded support. Assume that around $x=\pm 1$ we have asymptotic expansions
\[
\varphi(\pm(1-x))\sim |x|^{\frac{a}{2}}\sum_{m=0}^{+\infty}c_m^\pm x^m.
\]
Then for any $\sigma,\tau\in \RR_{>0}$, $|\xi|\gg 1$, and $K\in \ZZ_{\geq 0}$ we have
\begin{align*}
    & \frac{1}{\dpii}\int_{(\sigma)}\frac{\Xi(s)}{C^s}\int_{|x|>1}\varphi(x)|1-x^2|^{\frac s2} \rme(-x\xi)\rmd x\rmd s \\
   =  & \sum_{\pm}\sum_{m=0}^{K}\frac{\rme(\mp \xi)(-1)^m\cA_{\varphi,m}^{\sigma,\pm}(\Xi)(\mp C^2\xi)}{(\mp\xi)^{m+1}[\mp \xi]^{\frac a2}}+ O\left(\|\Xi\|_{2\tau}|C^2\xi|^{-\tau}|\xi|^{-K-2-\frac{a}{2}}\right),
\end{align*}
where $\cA_{\varphi,m}^{\sigma,\pm}(\Xi)(x)$ is defined in the previous theorem.
The implied constant depends only on $\varphi,K,\sigma,\tau$ and $a$. If $\Xi(s)$ has a meromorphic continuation to $\Re s>-2$, which has at most a simple pole at $s=0$ and holomorphic otherwise, we can take $\tau=0$.
\end{theorem} 
\begin{proof}
This theorem is essentially the content of \cite[Theorem A.15]{altug2017} and the fact that the error term is obtained by shifting the contour to $\sigma=\tau/2$. 
\end{proof}
\begin{remark}
In \cite{altug2017}, the author claimed that the above two theorems hold for $\Xi$ that has rapid decay vertically. However, this is not true. In the proof of Theorem A.13 of loc. cit., there is an integral on $u\in(\tau)$ of the function
\[
\legendresymbol{\rmi}{2\uppi Z}^{m+1+\frac{a+u}{2}},
\]
which is $\rme^{-\uppi t/4}$ as $|t|\to +\infty$. Hence we need a stronger estimate. Fortunately, the function $\widetilde{F}(s)$ satisfies such stronger estimate.
\end{remark}

\begin{remark}
We actually need a stronger estimate
\[
\int_{K}^{+\infty}x^\alpha\left(2+\frac{\gamma x}{K}\right)^\beta\ll_{\Re \alpha, \Re \beta, \gamma} \rme^{-K/2}
\]
of \cite[Lemma A.12]{altug2017}, which can be easily derived by a slight modification of the original proof. More precisely, the estimate only depends on $\Re \alpha$, $\Re \beta$ and $\gamma$ instead of $\alpha$, $\beta$, and $\gamma$. One also needs such estimate in the proof of the subsequent results of \cite[Appendix A]{altug2017}.
\end{remark}

\begin{corollary}\label{cor:mainarchimedeancor}
Let $\Xi(s)$ be a meromorphic function on the half plane $\Re s>-2$. Suppose that $\Xi(s)$ is holomorphic except for a simple pole at $s=0$.  Also, we assume that $\Xi(s)\ll_\sigma \rme^{-\uppi|t|/2}$ for each fixed $\sigma$. Let $\varphi(x)\in C_c^\infty(\RR)$. Assume that around $x=\pm 1$ we have asymptotic expansions
\[
\varphi(\pm(1-x))\sim \sum_{m=0}^{+\infty}c_m^\pm x^m.
\]
Then for any $\sigma,\varepsilon\in \RR_{>0}$, $|\xi|\gg 1$, and $M\in \ZZ_{>0}$ we have
\[
\frac{1}{\dpii}\int_{(\sigma)}\frac{\Xi(s)}{C^s}\int_{\RR}\varphi(x)|1-x^2|^{\frac s2} \rme(-x\xi)\rmd x\rmd s\ll \|\Xi\|_{0}|\xi|^{-2}+\|\Xi\|_{-2+2\varepsilon}|C|^{2-2\varepsilon}.
\]
The implied constant depends only on $\varphi$, $\sigma$ and $\varepsilon$.
\end{corollary}
\begin{proof}
By \autoref{thm:mainarchimedeantheorem1} and \autoref{thm:mainarchimedeantheorem2} with $a=0$, $K=0$ and $\tau=0$ (we can do this since $\Xi(s)$ has at most a simple pole at $s=0$), we have
\[
\frac{1}{\dpii}\int_{(\sigma)}\frac{\Xi(s)}{C^s}\int_{|x|<1}\varphi(x)|1-x^2|^{\frac s2} \rme(-x\xi)\rmd x\rmd s= \sum_{\pm}\frac{\rme(\mp \xi)\cA_{\varphi,0}^{\sigma,\pm}(\Xi)(\pm C^2\xi)}{\pm\xi}+ O\left(\|\Xi\|_0|\xi|^{-2}\right)
\]
and
\[
\frac{1}{\dpii}\int_{(\sigma)}\frac{\Xi(s)}{C^s}\int_{|x|>1}\varphi(x)|1-x^2|^{\frac s2} \rme(-x\xi)\rmd x\rmd s=\sum_{\pm}\frac{\rme(\mp \xi)\cA_{\varphi,0}^{\sigma,\pm}(\Xi)(\mp C^2\xi)}{\mp\xi}+ O\left(\|\Xi\|_0|\xi|^{-2}\right).
\]
Note that
\[
\cA_{\varphi,0}^{\sigma,\pm}(\Xi)(x)=\frac{1}{\dpii}\int_{(\sigma)}\Xi(s) c_0^\pm\legendresymbol{s}{2}\Gamma\left(1+\frac{s}{2}\right)x^{-\frac s2}\rmd s.
\]
The integrand is holomorphic on $\Re s>-2$ except for a simple pole at $s=0$. Hence we can shift the contour to $(-2+2\varepsilon)$ and by residue theorem we obtain
\begin{align*}
\cA_{\varphi,0}^{\sigma,\pm}(\Xi)(x)&=\frac{1}{\dpii}\int_{(-2+2\varepsilon)}\Xi(s) c_0^\pm\legendresymbol{s}{2}\Gamma\left(1+\frac{s}{2}\right)x^{-\frac s2}\rmd s+\res_{s=0} c_0^\pm(s)\Xi(s)\\
&= c_0^\pm(0)\res_{s=0}\Xi(s)+O\left(\|\Xi\|_{-2+2\varepsilon}|x|^{1-\varepsilon} \right).
\end{align*}

Thus
\begin{align*}
&\frac{1}{\dpii}\int_{(\sigma)}\frac{\Xi(s)}{C^s}\int_{|x|<1}\varphi(x)|1-x^2|^{\frac s2} \rme(-x\xi)\rmd x\rmd s\\
=&\res_{s=0}\Xi(s)\sum_{\pm} c_0^\pm(0)\frac{\rme(\mp \xi)}{\pm\xi}+ O\left(\|\Xi\|_{-2+2\varepsilon}|C^2|^{1-\varepsilon}|\xi|^{-\varepsilon}+\|\Xi\|_{0}|\xi|^{-2}\right)
\end{align*}
and
\begin{align*}
&\frac{1}{\dpii}\int_{(\sigma)}\frac{\Xi(s)}{C^s}\int_{|x|>1}\varphi(x)|1-x^2|^{\frac s2} \rme(-x\xi)\rmd x\rmd s\\
=& \res_{s=0}\Xi(s)\sum_{\pm}c_0^\pm(0)\frac{\rme(\mp \xi)}{\mp\xi}+ O\left(\|\Xi\|_{-2+2\varepsilon}|C^2|^{1-\varepsilon}|\xi|^{-\varepsilon}+\|\Xi\|_0|\xi|^{-2}\right).
\end{align*}
Adding them together and observing that $|\xi|\gg 1$, we obtain our result.
\end{proof} 

\subsection{Global estimates}
In this subsection we prove \autoref{thm:mainfourierestimate} and \autoref{thm:mainfourierestimate2}. Note that $v_{q_i}(N)=\nu_i$. Hence the constant $\nu$ in the local computation becomes $\nu_i/2$ in our case. 

For the proof of these two theorems we omit the variables on which the implied constants depend, which will only depend on the variables $\varphi,\psi_i,\pm,\nu,M,\varepsilon,L_i,\kappa,\sigma$ and $a_i$ if they occur.

\smallskip

\begin{proof}[Proof of \autoref{thm:mainfourierestimate}]
By Mellin inversion formula, for $\sigma>0$ we have
\[
\bF_{\iota,\epsilon}(\xi,\eta)=\int_{X_\iota}\int_{Y_\epsilon}\varphi(x)\prod_{i=1}^{r}\psi_i(y_i)\left[\frac{1}{\dpii}\int_{(\sigma)}\frac{\widetilde{\Phi}(s)}{C^s}|x^2\mp 1|^{s/2}|y^2-4N|_q'^{s/2}\rmd s\right]\rme(-x\xi)\prod_{i=1}^{r}\rme_{q_i}(-y_i\eta_i)\rmd x\rmd y.
\]
Since $\widetilde{\Phi}(s)$ decreases rapidly vertically, we can interchange the integrals over $s$ and $x,y$. Therefore, we conclude that
\begin{align*}
\bF_{\iota,\epsilon}(\xi,\eta)&=\frac{1}{\dpii}\int_{(\sigma)}\frac{\widetilde{\Phi}(s)}{C^s}\left[ \int_{X_\iota}\varphi(x)|x^2\mp 1|^{s/2}\rme(-x\xi)\rmd x\prod_{i=1}^{r}\int_{Y_{\epsilon_i}}\psi_i(y_i)|y_i^2-4N|_{q_i}'^{s/2}\rme_{q_i}(-y_i\eta_i)\rmd y_i\right]\rmd s\\
&=\frac{1}{\dpii}\int_{(\sigma)}\frac{\widetilde{\Phi}(s)}{C^s}\left[ \int_{X_\iota}\overline{\varphi}(x)|x^2\mp 1|^{\frac{a_0+s}{2}}\rme(-x\xi)\rmd x\prod_{i=1}^{r}\int_{Y_{\epsilon_i}}\overline{\psi_i}(y_i)|y_i^2-4N|_{q_i}'^{\frac{a_i+s}{2}}\rme_{q_i}(-y_i\eta_i)\rmd y_i\right]\rmd s.
\end{align*}

Fix $\kappa\in \lopen 0,1/2\ropen$. Write 
\[
\overline{\varphi}(x)=\overline{\varphi}(x)(1-\phi_\kappa(x))+\overline{\varphi}(x)\phi_\kappa(x) =\varphi^\blacktriangle(x)+\varphi^\blacktriangledown(x)
\]
if $\sgn N=+1$ and
\[
\overline{\varphi}(x)=\varphi^\blacktriangle(x)
\]
if $\sgn N=-1$. Moreover, write
\[
\overline{\psi_i}(y_i)=\overline{\psi_i}(y_i)(1-\triv_{U_i})+\overline{\psi_i}(y_i)\triv_{U_i} =\psi_i^\blacktriangle(y_i)+\psi_i^\blacktriangledown(y_i)
\]
if $N$ has square roots in $\QQ_{q_i}$ and
\[
\overline{\psi_i}(y_i)=\psi_i^\blacktriangle(y_i)
\]
if $N$ has no square roots in $\QQ_{q_i}$, where
\[
U_i=\{y_i\in \QQ_{q_i}\ |\ |y_i-2\sqrt{N}|_p\leq p^{-L_i}\}\sqcup \{y_i\in \QQ_{q_i}\ |\ |y_i+2\sqrt{N}|_p\leq p^{-L_i}\}
\]
for $L_i$ such that $\overline{\psi_i}(y_i)$ is constant on each of the two constituents of $U_i$ and that $L_i$ satisfies the technical conditions in \autoref{thm:ladicsingular}.

Thus we only need to estimate
\begin{equation}\label{eq:keyestimate}
\frac{1}{\dpii}\int_{(\sigma)}\frac{\widetilde{\Phi}(s)}{C^s}\left[ \int_{X_\iota}\varphi^{\blacklozenge_0}(x)|x^2\mp 1|^{\frac{a_0+s}{2}}\rme(-x\xi)\rmd x\prod_{i=1}^{r}\int_{Y_{\epsilon_i}}\psi_i^{\blacklozenge_i}(y_i)|y_i^2-4N|_{q_i}'^{\frac{a_i+s}{2}}\rme_{q_i}(-y_i\eta_i)\rmd y_i\right]\rmd s,
\end{equation}
where $\blacklozenge_i\in \{\blacktriangle,\blacktriangledown\}$.

Let
\[
\bI_0(s,\xi)=\int_{X_\iota}\varphi^{\blacklozenge_0}(x)|x^2\mp 1|^{\frac{a_0+s}{2}}\rme(-x\xi)\rmd x
\]
and
\[
\bI_i(s,\eta_i)=\int_{Y_{\epsilon_i}}\psi_i^{\blacklozenge_i}(y_i)|y_i^2-4N|_{q_i}'^{\frac{a_i+s}{2}}\rme_{q_i}(-y_i\eta_i)\rmd y_i.
\]
Then
\[
\eqref{eq:keyestimate}=\frac{1}{\dpii}\int_{(\sigma)}\frac{\widetilde{\Phi}(s)}{C^s}\bI_0(s,\xi) \prod_{i=1}^{r}\bI_i(s,\eta_i)\rmd s.
\]

For simplicity, we assume that $\blacklozenge_i=\blacktriangledown$ for $i=1,\dots,m$, and $\blacklozenge_i=\blacktriangle$ for all $i=m+1,\dots,r$. Hence by \autoref{prop:ladicsmooth} and \autoref{prop:ladicnonsingular}, for any $i=m+1,\dots,r$, $\bI_i(s,\eta_i)$ is an entire function for $s$, and for any $\tau_i>0$ we have
\[
\bI_i(s,\eta_i)\ll (1+|\eta_i|_{q_i})^{-\tau_i}.
\]

\underline{\emph{Case 1:}}\ \ $\blacklozenge_0=\blacktriangle$. 

By \autoref{prop:archimedeansmooth} and \autoref{prop:archimedeannonsingular}, $\bI_0(s,\xi)$ is an entire function for $s$, and for any $\tau>0$ we have 
\[
\bI_0(s,\xi)\ll (1+|\xi|)^{-\tau}.
\] 

Also, by \autoref{thm:ladicsingular}, we have
\[
\bI_i(s,\eta_i)\ll (1+|\eta_i|_{q_i})^{-1-\frac{a_i+\sigma}{2}}.
\]
for $i=1,\dots,m$.

By moving the contour to $\sigma=2M$ in \eqref{eq:keyestimate}, we obtain 
\begin{align*}
  \eqref{eq:keyestimate} & \ll \int_{(\sigma)}\frac{|\widetilde{\Phi}(s)|}{C^\sigma}(1+|\xi|)^{-\tau}\prod_{i=1}^{m} (1+|\eta_i|_{q_i})^{-1-\frac{a_i+\sigma}{2}} \prod_{i=m+1}^{r}(1+|\eta_i|_{q_i})^{-\tau_i} \\
 & \ll \left(|C|^2\llbracket\xi,\eta\rrbracket\right)^{-M} (1+|\xi|_\infty)^{-1-\frac{a_0}{2}}\prod_{i=1}^{r}(1+|\eta_i|_{q_i})^{-1-\frac{a_i}{2}}.
\end{align*}

\underline{\emph{Case 2:}}\ \ $\blacklozenge_0=\blacktriangledown$. 

We will consider the cases $|\xi|\ll 1$ and $|\xi|\gg 1$ separately.

\underline{\emph{Case 2.1:}} \ \ $|\xi|\ll 1$.

For any fixed $\sigma'>0$, we have $\Re a_0+\sigma'>-2$. Hence
\[
|\bI_0(s,\xi)|\leq\int_{X_\iota}|\varphi^{\blacktriangledown}(x)||x^2\mp 1|^{\frac{a_0+\sigma}{2}}\rmd x\ll_{\sigma'} 1.
\]
Therefore, a similar argument as in Case 1 shows that
\begin{align*}
  \eqref{eq:keyestimate} & \ll \int_{(\sigma)}\frac{|\widetilde{\Phi}(s)|}{C^\sigma}\prod_{i=1}^{m} (1+|\eta_i|_{q_i})^{-1-\frac{a_i+\sigma}{2}} \prod_{i=m+1}^{r}(1+|\eta_i|_{q_i})^{-\tau_i} \\
 & \ll \left(|C|^2\llbracket\xi,\eta\rrbracket\right)^{-M} (1+|\xi|_\infty)^{-1-\frac{a_0}{2}}\prod_{i=1}^{r}(1+|\eta_i|_{q_i})^{-1-\frac{a_i}{2}}.
\end{align*}

\underline{\emph{Case 2.2:}}\ \ $|\xi|\gg 1$.

This is the most complicated part and we will use the generalized results for the archimedean places. Since we are considering $\varphi^\blacktriangledown(x)$, we must have $\sgn N=1$. Thus $X_\iota=\lopen -1,1\ropen$ or $X_\iota=\lopen -\infty,-1\ropen \cup \lopen 1,+\infty\ropen$. By the definition of the bump function we know that $\varphi^\blacktriangledown(x)$ has the same asymptotic expansion as  $\varphi(x)$ near $\pm 1$.

For simplicity we only consider the case $X_\iota=\lopen -1,1\ropen$, another case is similar.
Let
\[
\Xi(s)=\widetilde{\Phi}(s)\prod_{i=1}^{r}\bI_i(s,\eta_i).
\]
Then we have
\[
\eqref{eq:keyestimate}=\frac{1}{\dpii}\int_{(\sigma)}\frac{\Xi(s)}{C^s}\bI_0(s,\xi)\rmd s=
     \frac{1}{\dpii}\int_{(\sigma)}\frac{\Xi(s)}{C^s}\int_{-1}^{1}\varphi^\blacktriangledown(x)|1-x^2|^{\frac s2} \rme(-x\xi)\rmd x\rmd s .
\]
By \autoref{thm:mainarchimedeantheorem1} with $K=0$ and $\tau=M$, we obtain
\[
\eqref{eq:keyestimate}=\sum_{\pm}\frac{\rme(\mp \xi)\cA_{\varphi,0}^{\sigma,\pm}(\Xi)(\pm C^2\xi)}{(\pm\xi)[\pm\xi]^{\frac a2}}+ O\left(\|\Xi\|_{2M}|C^2\xi|^{-M}|\xi|^{-2-\frac{a}{2}}\right).
\]

By moving the contour $(\sigma)$ to $(2M)$, we obtain
\[
\cA_{\varphi,0}^{\sigma,\pm}(\Xi)(x)=\frac{1}{\dpii}\int_{(2M)}\Xi(s) c_0^\pm\legendresymbol{s}{2}\Gamma\left(1+\frac{a+s}{2}\right)[x]^{-\frac s2}\rmd s \ll \|\Xi\|_{2M}|x|^{-M}
\]
since $\Gamma(s)$ has exponential decay vertically and $c_0^\pm$ has at most polynomial growth vertically. Therefore
\begin{align*}
  \eqref{eq:keyestimate} & \ll  \|\Xi\|_{2M}|C^2\xi|^{-M}|\xi|^{-1-\frac{a}{2}}+ \|\Xi\|_{2M}|C^2\xi|^{-M}|\xi|^{-2-\frac{a}{2}} \\
   & \ll \|\Xi\|_{2M}(|C^2\xi|)^{-M}|\xi|^{-1-\frac{a}{2}}.
\end{align*}

By \autoref{prop:ladicsmooth}, \autoref{prop:ladicnonsingular}, and \autoref{thm:ladicsingular} we have
\begin{align*}
\|\Xi\|_{\sigma}&=\left\|\widetilde{\Phi}(s)|\prod_{i=1}^{m}|\bI_i(s,\eta_i)|\prod_{i=m+1}^{r}|\bI(s,\eta_i)| \right\|_{\sigma}\ll \|\widetilde{\Phi}\|_{\sigma}\sup_{\Re s=\sigma}\prod_{i=1}^{m}|\bI(s,\eta_i)|\prod_{i=m+1}^{r}|\bI(s,\eta_i)|\\
&\ll \prod_{i=1}^{m}(1+|\eta_i|_{q_i})^{-1-\frac{a_i}{2}-M}\prod_{i=m+1}^{r}(1+|\eta_i|_{q_i})^{-\tau_i}.
\end{align*}
for $\sigma=2M$.
This yields
\begin{align*}
  \eqref{eq:keyestimate} & \ll \prod_{i=1}^{m}(1+|\eta_i|_{q_i})^{-1-\frac{a_i}{2}-M}\prod_{i=m+1}^{r}(1+|\eta_i|_{q_i})^{-\tau_i} (|C|^2|\xi|)^{-M}|\xi|^{-1-\frac{a}{2}}\\
   &  \ll \left(|C|^2\llbracket\xi,\eta\rrbracket\right)^{-M} (1+|\xi|_\infty)^{-1-\frac{a_0}{2}}\prod_{i=1}^{r}(1+|\eta_i|_{q_i})^{-1-\frac{a_i}{2}}.\qedhere
\end{align*}
\end{proof}

\begin{proof}[Proof of \autoref{thm:mainfourierestimate2}]
The first statement is implied in \autoref{thm:mainfourierestimate} by summing over $\iota$ and $\epsilon$, so we only need to prove the second one.

By Mellin inversion formula, for $\sigma>0$ we have
\[
\bF_\epsilon(\xi,\eta)=\int_{\RR}\int_{Y_\epsilon}\varphi(x)\prod_{i=1}^{r}\psi_i(y_i)\left[\frac{1}{\dpii}\int_{(\sigma)}\frac{\widetilde{\Phi}(s)}{C^s}|x^2\mp 1|^{s/2}|y^2-4N|_q'^{s/2}\rmd s\right]\rme(-x\xi)\prod_{i=1}^{r}\rme_{q_i}(-y_i\eta_i)\rmd x\rmd y.
\]
Since $\widetilde{\Phi}(s)$ decreases rapidly vertically, we can interchange the integral over $s$ and $x,y$. Thus we obtain
\begin{align*}
\bF_\epsilon(\xi,\eta)&=\frac{1}{\dpii}\int_{(\sigma)}\frac{\widetilde{\Phi}(s)}{C^s}\left[ \int_{\RR}\varphi(x)|x^2\mp 1|^{s/2}\rme(-x\xi)\rmd x\prod_{i=1}^{r}\int_{Y_{\epsilon_i}}\psi_i(y_i)|y_i^2-4N|_{q_i}'^{s/2}\rme_{q_i}(-y_i\eta_i)\rmd y_i\right]\rmd s\\
&=\frac{1}{\dpii}\int_{(\sigma)}\frac{\widetilde{\Phi}(s)}{C^s}\left[ \int_{\RR}\overline{\varphi}(x)|x^2\mp 1|^{\frac{a_0+s}{2}}\rme(-x\xi)\rmd x\prod_{i=1}^{r}\int_{Y_{\epsilon_i}}\overline{\psi_i}(y_i)|y_i^2-4N|_{q_i}'^{\frac{a_i+s}{2}}\rme_{q_i}(-y_i\eta_i)\rmd y_i\right]\rmd s.
\end{align*}

Write 
\[
\overline{\varphi}(x)=\overline{\varphi}(x)(1-\phi_\kappa(x))+\overline{\varphi}(x)\phi_\kappa(x) =\varphi^\blacktriangle(x)+\varphi^\blacktriangledown(x)
\]
if $\sgn N=+1$ and
\[
\overline{\varphi}(x)=\varphi^\blacktriangle(x)
\]
if $\sgn N=-1$. Moreover, write
\[
\overline{\psi_i}(y_i)=\overline{\psi_i}(y_i)(1-\triv_{U_i})+\overline{\psi_i}(y_i)\triv_{U_i} =\psi_i^\blacktriangle(y_i)+\psi_i^\blacktriangledown(y_i)
\]
if $N$ has square roots in $\QQ_{q_i}$ and
\[
\overline{\psi_i}(y_i)=\psi_i^\blacktriangle(y_i)
\]
if $N$ has no square roots in $\QQ_{q_i}$, where $U_i$ is as defined in the previous theorem.

It therefore suffices to bound
\begin{equation}\label{eq:keyestimate2}
\frac{1}{\dpii}\int_{(\sigma)}\frac{\widetilde{\Phi}(s)}{C^s}\left[ \int_{\RR}\varphi^{\blacklozenge_0}(x)|x^2\mp 1|^{\frac{a_0+s}{2}}\rme(-x\xi)\rmd x\prod_{i=1}^{r}\int_{Y_{\epsilon_i}}\psi_i^{\blacklozenge_i}(y_i)|y_i^2-4N|_{q_i}'^{\frac{a_i+s}{2}}\rme_{q_i}(-y_i\eta_i)\rmd y_i\right]\rmd s,
\end{equation}
where $\blacklozenge_i\in \{\blacktriangle,\blacktriangledown\}$.

Let
\[
\bI_0(s,\xi)=\int_{\RR}\varphi^{\blacklozenge_0}(x)|x^2\mp 1|^{\frac{a_0+s}{2}}\rme(-x\xi)\rmd x
\]
and
\[
\bI_i(s,\eta_i)=\int_{Y_{\epsilon_i}}\psi_i^{\blacklozenge_i}(y_i)|y_i^2-4N|_{q_i}'^{\frac{a_i+s}{2}}\rme_{q_i}(-y_i\eta_i)\rmd y_i.
\]

\underline{\emph{Case 1:}}\ \ $\blacklozenge_0=\blacktriangle$. 

We have proved that 
\[
  \bF_{\iota,\epsilon}(\xi,\eta) \ll (1+|\xi|_\infty)^{-\tau}\prod_{i=1}^{r}(1+|\eta_i|_{q_i})^{-1-\frac{a_i}{2}}.
\]
Thus the result follows immediately by summing over $\iota$.

\underline{\emph{Case 2:}}\ \ $\blacklozenge_0=\blacktriangledown$. 

We will consider the cases $|\xi|\ll 1$ and $|\xi|\gg 1$ separately.

\underline{\emph{Case 2.1:}} \ \ $|\xi|\ll 1$.

In this case, $(1+|\xi|)^{-1-\frac{a_0}{2}}\asymp 1 \asymp (1+|\xi|)^{-2}$. Hence the conclusion follows immediately from \autoref{thm:mainfourierestimate} by summing over $\iota$.

\underline{\emph{Case 2.2:}}\ \ $|\xi|\gg 1$. 

Let
\[
\Xi(s)=\widetilde{\Phi}(s)\prod_{i=1}^{r}\bI_i(s,\eta_i).
\]
Then we have
\[
\eqref{eq:keyestimate2}=\frac{1}{\dpii}\int_{(\sigma)}\frac{\Xi(s)}{C^s}\bI_0(s,\xi)\rmd s=
     \frac{1}{\dpii}\int_{(\sigma)}\frac{\Xi(s)}{C^s}\int_{\RR}\varphi^\blacktriangledown(x)|1-x^2|^{\frac s2} \rme(-x\xi)\rmd x\rmd s .
\]
By \autoref{cor:mainarchimedeancor}, we obtain
\[
  \eqref{eq:keyestimate2}  \ll  \|\Xi\|_0|\xi|^{-2}+\|\Xi\|_{-2+2\varepsilon}|C|^{2-2\varepsilon}.
\]

By \autoref{prop:ladicsmooth}, \autoref{prop:ladicnonsingular}, and \autoref{thm:ladicsingular} we have
\[
\|\Xi\|_{\sigma}\ll\|\widetilde{\Phi}\|_\sigma\sup_{\Re s=\sigma}\prod_{i=1}^{m}|\bI_i(s,\eta_i)|\prod_{i=m+1}^{r}|\bI(s,\eta_i)| \ll \prod_{i=1}^{m}(1+|\eta_i|_{q_i})^{-\frac{a_i}{2}-\varepsilon}\prod_{i=m+1}^{r}(1+|\eta_i|_{q_i})^{-\tau_i}\\
\]
for $\sigma=-2+2\varepsilon$ and
\[
\|\Xi\|_\sigma\ll \prod_{i=1}^{m}(1+|\eta_i|_{q_i})^{-1-\frac{a_i}{2}}\prod_{i=m+1}^{r}(1+|\eta_i|_{q_i})^{-\tau_i}
\]
for $\sigma=0$. This yields
\begin{align*}
  \eqref{eq:keyestimate2} 
  \ll  & \prod_{i=1}^{m}(1+|\eta_i|_{q_i})^{-1-\frac{a_i}{2}}\prod_{i=m+1}^{r}(1+|\eta_i|_{q_i})^{-\tau_i}|\xi|^{-2} +\prod_{i=1}^{m}(1+|\eta_i|_{q_i})^{-\frac{a_i}{2}-\varepsilon}\prod_{i=m+1}^{r}(1+|\eta_i|_{q_i})^{-\tau_i}|C|^{2-2\varepsilon}\\
  \ll & \left[\left(|C|^2\prod_{i=1}^{r}(1+|\eta_i|_{q_i})\right)^{1-\varepsilon}+(1+|\xi|_\infty)^{-2}\right] \prod_{i=1}^{r}(1+|\eta_i|_{q_i})^{-1-\frac{a_i}{2}}
\end{align*}
since $1+|\xi|\asymp |\xi|$ in this case.
\end{proof}

\section{Generalized Kloosterman sum estimates}\label{sec:kloosterman}
In this section, we establish an estimate for $\Kl_{k,f}^S(\xi,m)$, where $k,f\in \ZZ_{(S)}^{>0}$ and $\xi,m\in \ZZ^S$.
The proofs are based on the results in \cite[Appendix B]{altug2017}.

Let $p$ be a prime. For $a\in\QQ$, we define $a_{(p)}=p^{v_p(a)}$ and  $a^{(p)}=a/a_{(p)}$.

\begin{proposition}\label{prop:prodkloostermanlocal}
Let $k,f\in \ZZ_{(S)}^{>0}$ and $\xi,m\in \ZZ^S$. Then we have
\[
\Kl_{k,f}^S(\xi,m)=\prod_{p}\Kl_{k_{(p)},f_{(p)}}^{(p)}\left(((kf^2)^{(p)})^{-1}\xi,m\right),
\]
where $((kf^2)^{(p)})^{-1}$ denotes the inverse of $(kf^2)^{(p)}$ modulo $(kf^2)_{(p)}$.
\end{proposition}
\begin{proof}
Let
\[
kf^2=\prod_{j=1}^{t}(kf^2)_{(p_j)}
\]
be the prime factorization. By Chinese remainder theorem, we have an isomorphism
\[
\begin{split}
   \varphi: \prod_{j=1}^{t}\ZZ/(kf^2)_{(p_j)} &\xrightarrow{\cong} \ZZ/kf^2 \\
     (a_1,\dots,a_t) &\mapsto \sum_{j=1}^{t} a_j(kf^2)^{(p_j)}((kf^2)^{(p_j)})^{-1}.
\end{split}
\]
Observe that $\varphi(a_1,\dots,a_t)\equiv a_j\pmod{(kf^2)_{(p_j)}}$. Moreover, 
\[
\frac{\varphi(a_1,\dots,a_t)\xi}{kf^2}=\frac{\xi}{kf^2}\sum_{j=1}^{t} a_j(kf^2)^{(p_j)}((kf^2)^{(p_j)})^{-1}=\sum_{j=1}^{t} \frac{a_j((kf^2)^{(p_j)})^{-1}\xi}{(kf^2)_{(p_j)}}
\]
and thus
\[
\rme\legendresymbol{\varphi(a_1,\dots,a_t)\xi}{kf^2}=\prod_{j=1}^{t} \rme\legendresymbol{a_j((kf^2)^{(p_j)})^{-1}\xi}{(kf^2)_{(p_j)}},\quad \rme_q\legendresymbol{\varphi(a_1,\dots,a_t)\xi}{kf^2}=\prod_{j=1}^{t} \rme_q\legendresymbol{a_j((kf^2)^{(p_j)})^{-1}\xi}{(kf^2)_{(p_j)}}.
\]
Hence by using the equivalent definition \eqref{eq:localgeneralizedkloostermaneq},
\begin{align*}
  \Kl_{k,f}^S(\xi,m) & =\sum_{\substack{a_1,\dots,a_t\\a_j \bmod (kf^2)_{(p_j)}\\ a_j^2-4m\equiv 0 \ (f_{(p_j)}^2)}}\prod_{j=1}^{t}\legendresymbol{(\varphi(a_1,\dots,a_t)^2-4m)/f^2}{k_{(p_j)}} \rme\legendresymbol{\varphi(a_1,\dots,a_t)\xi}{kf^2}\rme_{q} \legendresymbol{\varphi(a_1,\dots,a_t)\xi}{kf^2} \\
   & =\sum_{\substack{a_1,\dots,a_t\\a_j \bmod (kf^2)_{(p_j)}\\ a_j^2-4m\equiv 0 \ (f_{(p_j)}^2)}}\prod_{j=1}^{t}\legendresymbol{(a_j^2-4m)/f^2}{k_{(p_j)}} \rme\legendresymbol{a_j((kf^2)^{(p_j)})^{-1}\xi}{(kf^2)_{(p_j)}} \rme_q\legendresymbol{a_j((kf^2)^{(p_j)})^{-1}\xi}{(kf^2)_{(p_j)}}\\
   & = \prod_{j=1}^{t}\Kl_{k_{(p_j)},f_{(p_j)}}^{(p_j)}\left(((kf^2)^{(p_j)})^{-1}\xi,m\right).\qedhere
\end{align*}
\end{proof}
Now we use the results of \cite[Appendix B]{altug2017} to obtain the estimates. We remark that the proof of \cite[Appendix B]{altug2017} uses the Weil bound \cite{weil1948} of the Kloosterman sums. More precisely, if
\[
S(a,b;c)=\sum_{x\in (\ZZ/c\ZZ)^\times}\rme\legendresymbol{ax+bx^{-1}}{c},
\]
then 
\[
|S(a,b;c)|\leq 2\sqrt{\gcd(a,b,c)}\sqrt{c}.
\]

For $a\in \ZZ^S$ and $b\in \ZZ_{>0}$, we define
\[
\delta(a;b)=\begin{cases}
              1, & \text{if $x^2\equiv a\ (b)$ has a solution,}  \\
              0, & \text{otherwise}.
            \end{cases}
\]

\begin{proposition}\label{prop:kloostermanunramified}
Suppose that $\xi,m\in \ZZ^S$ and $p\notin S$. Then
\[
\Kl_{p^u,p^v}^{(p)}(\xi,m)\ll\delta(4m;p^{2v})\begin{dcases}
             p^{u+\min\{2v,v_{p}(m)\}/2}, & v_p(\xi)\geq u+\frac{1}{2}\min\{2v,v_{p}(m)\},  \\
              \frac{p^{u+\min\{2v,v_{p}(m)\}/2}}{\sqrt{p}}, & v_p(\xi)= u-1+\frac{1}{2}\min\{2v,v_{p}(m)\},  \\
              0, & \text{otherwise}.
            \end{dcases}
\]
\end{proposition}
\begin{proof}
Choose $\xi',m'\in \ZZ$ such that $\xi\equiv \xi'\pmod {p^{u+2v}}$ and $m\equiv m'\pmod {p^{u+2v}}$. By \cite[Corollary B.6]{altug2017}, we have
\[
\Kl_{p^u,p^v}^{(p)}(\xi',m')\ll\delta(4m';p^{2v})\begin{dcases}
              p^{u+\min\{2v,v_{p}(m')\}/2}, & v_p(\xi')\geq u+\frac{1}{2}\min\{2v,v_{p}(m')\},  \\
              \frac{p^{u+\min\{2v,v_{p}(m')\}/2}}{\sqrt{p}}, & v_p(\xi')= u-1+\frac{1}{2}\min\{2v,v_{p}(m')\},  \\
              0, & \text{otherwise}.
            \end{dcases}
\]
By \autoref{prop:equalkloosterman}, the left hand sides of the two formulas are equal. A straightforward verification shows that the right hand sides of them are equal.
\end{proof}

Having established the local case, we now proceed to the global estimate.

\begin{corollary}\label{cor:kloostermanfinal}
Suppose that $\xi,m\in \ZZ^S$ and $k,f\in \ZZ_{(S)}^{>0}$. Then for any $\varepsilon>0$,
\[
\Kl_{k,f}^S(\xi,m)
   \ll_\varepsilon  \delta(4m;f^2)\!\begin{dcases}
             \!(kf^2)^\varepsilon\sqrt{k\gcd(m^{(q)},f^2)}\sqrt{\gcd\!\left( \frac{\xi^{(q)}}{\sqrt{\gcd(m^{(q)},f^2)}},k\right)}, &\frac{k\sqrt{\gcd(m^{(q)},f^2)}}{\rad k}\mid \xi^{(q)},  \\
              0, &\text{otherwise},
            \end{dcases}
\]
where for $a\in \ZZ_{>0}$,
\[
\rad a=\prod_{p\mid a}p.
\]
\end{corollary}
\begin{proof}
This follows directly from \autoref{prop:kloostermanunramified} by going through the proof of \cite[Corollary B.8]{altug2017}. Note that the last claim in the proof of loc. cit. should be corrected. The right bound is
\[
\prod_{p\mid kf^2}O(1)=O(1)^{\bm{\omega}(kf^2)}=O(1)^{O(\log\log kf^2)}\ll_\varepsilon(kf^2)^\varepsilon,
\]
where $\bm{\omega}(n)$\index{omega@$\bm{\omega}(n)$} denotes the number of prime divisors of $n$.
\end{proof}

\bibliography{ref.bib}
\bibliographystyle{amsalpha}

\printindex

\end{document}